\RequirePackage{fix-cm}
\documentclass{svjour3}  

\smartqed  

\usepackage{tikz}
\usepackage{xcolor}
\usepackage{amsmath}

\newcounter{proofcnt} 
\renewcommand{\theproofcnt}{\Alph{proofcnt}} 

\usepackage{hyperref}
\usepackage{etoolbox}
\usepackage{mathrsfs}
\newcommand{\eqf}[1]{%
  \ifstrequal{#1}{perpp}{\ensuremath{(\hyperlink{perpp}{\mathbf{P}_{(\rho,\xi)}})}}{}%
  \ifstrequal{#1}{perdp}{\ensuremath{(\hyperlink{perdp}{\mathbf{D}_{(\rho,\xi)}})}}{}%
  \ifstrequal{#1}{pp}{\ensuremath{(\hyperlink{pp}{\mathbf{P}})}}{}%
  \ifstrequal{#1}{dp}{\ensuremath{(\hyperlink{dp}{\mathbf{D}})}}{}%
}
\usepackage{subcaption}
\usepackage{graphicx}
\usepackage{enumitem}
\usepackage{multirow}
\usepackage{makecell}
\usepackage{array}
\usepackage{cite}
\usepackage{adjustbox} 
\usepackage{threeparttable}
\newcommand{\Step}[2]{%
\STATE \makebox[3.8em][l]{\textbf{Step #1:}}%
\begin{minipage}[t]{\dimexpr\linewidth-3.8em\relax}
\setlength{\parskip}{0pt}%
\setlength{\abovedisplayskip}{3pt}%
\setlength{\belowdisplayskip}{3pt}%
\setlength{\abovedisplayshortskip}{3pt}%
\setlength{\belowdisplayshortskip}{3pt}%
#2
\end{minipage}%
}

\begin{document}

\title{A Multiscale Primal--Dual Interior-Point Relaxation Method for Large-Scale Optimal Transport Problems
}

\titlerunning{MSIPRM FOR LARGE-SCALE OPTIMAL TRANSPORT PROBLEMS}        

\author{Shengyu~Sun         \and
        Rui-Jin  Zhang       \and
        Ruoyu Diao          \and
        Yu-Hong Dai
}

\authorrunning{S. Sun, R.J. Zhang, R. Diao, Y.H. Dai} 

 \institute{  Shengyu~Sun $\cdot$ Ruoyu~Diao $\cdot$ Yu-Hong~Dai \at
              State Key Laboratory of Mathematical Sciences, Academy of Mathematics and Systems Science, Chinese Academy of Sciences, and the University of Chinese Academy of Sciences, Beijing, China \\         
              \email{sunshengyu@lsec.cc.ac.cn, diaoruoyu18@mails.ucas.ac.cn, dyh@lsec.cc.ac.cn}
              \and
               Rui-Jin~Zhang \at
              School of Mathematical Sciences and LPMC, Nankai University, Tianjin, China\\
              \email{zhangrj@nankai.edu.cn} 
}

\date{Received: date / Accepted: date}

\maketitle

\begin{abstract}
Large-scale optimal transport (OT) problems involve a vast number of transport variables, leading to prohibitive memory and computational costs. To address these challenges, we propose a multiscale primal--dual interior-point relaxation method (MSIPRM). The multiscale outer framework constructs a hierarchy of standard OT problems at progressively finer levels. At each level, the OT problem is solved over a sequence of adaptively refined active sets initialized based on the solution support at the previous level. This yields a sequence of closely related sparse subproblems, thereby substantially reducing memory requirements. The primal--dual interior-point relaxation method (IPRM) serves as the inner solver for each sparse subproblem. Since IPRM does not require strictly interior iterates, it can readily use the solution of the previous subproblem as a warm start. To efficiently obtain the Newton direction, we solve a reduced Schur complement system derived from the normal equations. Furthermore, we develop an effective support-identification strategy based on the approximate solutions obtained by IPRM. We establish condition number estimates for the Schur complement matrices and analyze the global and local convergence properties of the algorithm. Numerical experiments on large-scale test problems demonstrate the computational efficiency and scalability of MSIPRM and show that it compares favorably with existing solvers. In particular, MSIPRM can handle instances whose full formulations contain trillions of transport variables.


\end{abstract}

\section{Introduction}
Optimal transport (OT) provides a fundamental framework for comparing distributions and transporting mass between them. It has become an important modeling and computational tool in areas such as
economics \cite{chiappori2010hedonic}, computer vision \cite{ge2021ota}, machine learning~\cite{Chen2026NeuralSF},
and fluid dynamics~\cite{ChenVisualizing}. In the discrete setting, an OT problem can be formulated as a special case of linear programming (LP) whose variables represent the amounts of mass transported between pairs of
source and target points. However, this classical formulation becomes extremely challenging at large scale. For a two-dimensional grid with resolution $\texttt{row}\times \texttt{col}$, the corresponding transport problem contains $(\texttt{row}\cdot \texttt{col})^2$ variables. Table~\ref{basicdata} shows that the rapidly growing memory requirement for storing the transport plan becomes a major computational bottleneck as the grid is refined.

\begin{table}[H]
\centering
\caption{Numbers of constraints and variables, and memory usage of transport variables in OT problems at different grid resolutions}
\label{basicdata}
\resizebox{\textwidth}{!}{ 
    \footnotesize 
    \begin{tabular}{c c c c c c c  c }
         \hline 
         Resolution  & $32\times 32$ & $64\times 64$ & $128\times128$ &$256\times256$& $512\times512$ & $1024\times1024$ & $2048\times2048$\\ [1.5ex] 
        \hline 
Constraints   &  2.0E+03 & 8.2E+03 & 3.3E+04 & 1.3E+05 & 5.2E+05 & 2.1E+06 & 8.4E+06\\
\hline
Variables   &   1.0E+06 & 1.7E+07 & 2.7E+08 & 4.3E+09& 6.9E+10 & 1.1E+12  & 1.8E+13 \\
\hline
Memory  & 8MB &  128MB & 2GB & 32GB &512GB   & 8TB  & 128TB\\
        \hline 
    \end{tabular}   
}\end{table}

Several classes of methods have been developed to solve OT problems. Since discrete OT can be viewed as a network flow problem \cite{tarjan1997dynamic}, classical combinatorial methods such as the network simplex \cite{kovacs2015minimum} 
remain important baselines.
As OT is a special case of LP, several LP-based algorithms, including the smoothing Newton method SqSN \cite{hou2024sparse} and interior-point methods (IPMs) \cite{cipolla2024regularized,zanetti2023interior}, exploit the specific structure of OT to improve computational efficiency.
Additionally, entropy-regularized Sinkhorn-type methods \cite{cuturi2013sinkhorn,dvurechensky2018computational} are widely used because of their simplicity and parallelizability. However, for very large problems, full-memory formulations remain expensive, limiting the applicability of many algorithms.

Two main strategies have been used to reduce this burden. The first is to replace the original OT formulation with an equivalent reduced model for a special class of cost functions. A representative example is the reformulation introduced by Auricchio et al. \cite{auricchio2018computing}, which can be efficiently solved by the more recent HOT algorithm \cite{zhang2025hot}. The second is to exploit a multiscale framework \cite{merigot2011multiscale,schmitzer2016sparse,oberman2015efficient}. In this approach, one first solves coarse OT problems and then transfers support information to finer scales, thereby restricting the set of variables for the fine-scale problem to a small subset. 
Building on this idea, several successful methods have been developed, including ShortCut \cite{schmitzer2016sparse}, stabilized sparse Sinkhorn (Sp-Sinkhorn) schemes \cite{schmitzer2019stabilized}, and the multiscale semismooth Newton method (MSSN) \cite{liu2022multiscale}.

The present work adopts the second strategy. A general multiscale method consists of an outer multiscale framework and an inner solver for the resulting sparse subproblems. In the outer framework, a hierarchy of standard OT problems is constructed by iteratively merging geometrically adjacent points.
At each level, the full OT problem is reduced by retaining only a small subset of variables. The active set is initialized using the support of the solution at the previous level. The method then updates the active set and solves the corresponding sparse subproblems until the solution at the current level is obtained. Based on this scheme, MSSN has achieved favorable numerical results. It employs a semismooth Newton method to solve the subproblems. The method further incorporates a fast algorithm for identifying redundant rows, together with a robust active-set update strategy. Nevertheless, numerical experiments reported in \cite{liu2022multiscale} suggest that MSSN may suffer from numerical instability on some difficult instances.

Beyond the outer multiscale framework, the choice of the inner solver is also crucial. At each level, one needs to solve a sequence of closely related sparse subproblems because the active set is updated repeatedly. An inner solver should therefore be able to exploit warm starts, which can reduce the number of inner iterations and improve the overall efficiency. However, despite extensive research on warm-start strategies for IPMs \cite{benson2007exact,gondzio2015new,yildirim2002warm}, their implementation remains challenging because IPM iterates must remain strictly interior.
This limitation motivates the use of the primal--dual interior-point relaxation method (IPRM), which is based on the barrier augmented Lagrangian (BAL) function~\cite{liu2022primal,zhang2023iprqp,zhang2024iprsocp,zhang2024iprsdp}. Unlike standard IPMs, IPRM is simple and suitable for warm starts because it does not require the iterates to lie in the interior.

To this end, we propose a multiscale primal--dual interior-point relaxation method (MSIPRM) that employs IPRM as the inner solver for each sparse subproblem. The main contributions of this work can be summarized as follows:

\begin{itemize}[leftmargin=3em]
\item[$\bullet$] Algorithmically, we integrate IPRM with a multiscale framework. For each sparse subproblem, we develop an efficient and robust preprocessing and decomposition procedure. Then the solution to the previous subproblem is used to warm-start IPRM.
To compute the Newton direction, we solve a reduced Schur complement system.
At each level, an effective initial active set is predicted using the OT solution computed by IPRM at the preceding level. Based on this initial  set, we simplify the active-set update strategy and introduce a few new indices during each update. In practice, the prescribed accuracy is typically reached after only a small number of updates.

\item[$\bullet$] Numerically, we demonstrate the advantages of MSIPRM in terms of scalability and computational efficiency.
By avoiding substantial memory overhead, MSIPRM can solve problems whose full formulations involve trillions of variables.
On benchmark datasets, MSIPRM is several times faster than HOT \cite{zhang2025hot} and LEMON's network simplex solver in certain cases, while also outperforming competing methods based on similar multiscale frameworks. Extensive experiments further confirm the efficiency of IPRM and the effectiveness of our active-set prediction strategy.

\item[$\bullet$] Theoretically, we investigate structural properties of OT problems from a graph-theoretic perspective. Building on this analysis, we prove that, for almost all source--target vectors, the primal constraint nondegeneracy condition holds at every optimal solution. To provide an effective initial active set, we derive a theoretical guarantee for support identification. In addition, upper bounds on the condition numbers of the Schur complement matrices (SCMs) are provided. Finally, a global error bound is derived, and the local quadratic convergence of the algorithm is established.
\end{itemize} 

The remainder of this paper is organized as follows. Section~\ref{preliminary} introduces balanced OT problems and the associated sparse subproblems. Section~\ref{Active-set}
presents a preprocessing procedure for these subproblems and explains how solutions to the standard OT problems can be recovered. This section also discusses the constraint degeneracy conditions. Section~\ref{sectioniprm} first introduces the multiscale framework,
then presents IPRM and the support identification strategy, and finally combines them to develop MSIPRM. It further analyzes the condition numbers of the SCMs. Section~\ref{Convergence analyses} establishes the global and local convergence properties of the proposed algorithm. Section~\ref{numerical results} reports numerical results that demonstrate the efficiency of MSIPRM. Section~\ref{conclusion} concludes the paper.

\textbf{Notation.} Throughout this paper, we use the following notation. All vectors are column vectors, and $w= (p, q)$ means $w = [p^\top,q^\top]^\top$. 
We denote the $n\times n$ identity matrix by $I_n$ and the $n$-dimensional all-ones vector by $\mathbf{1}_n$. For a matrix $X=(X_{ij})\in\mathbb{R}^{m\times n}$, we define its column-wise vectorization by $\vecc(X)=(X_{11},\cdots,X_{m1},X_{12},\cdots,X_{m2},\cdots,X_{1n},\cdots,X_{mn})^\top$. Conversely, for a vector $x\in\mathbb{R}^{mn}$, we define $\mat(x)=X$ if $x=\vecc(X)$. For vectors $u,\,v\in\mathbb{R}^n$,
we write $u\geq v$ if $u-v\in\mathbb{R}^n_+$, and $u> v$ if $u-v\in\mathbb{R}^n_{++}$. The Hadamard product of $u$ and $v$ is $u\circ v=(u_1v_1,\cdots,u_nv_n)^\top$. The negative part of $v$ is defined componentwise by $(v_-)_j=\min\{v_j,0\}$ for $1\leq j\leq n$. The notation $\diag(v)$ denotes the diagonal matrix with diagonal entries $v$. The superscripts $[i]$ and $(k)$ denote the outer and inner iteration indices, respectively, while a subscript $\ell$ denotes the level index.
The Euclidean and infinity norms of $v$ are denoted by $\|v\|$ and $\|v\|_\infty$, respectively. For a matrix $A\in\mathbb{R}^{m\times n}$, $\|A\|$ and $\|A\|_F$ denote its spectral and Frobenius norms, respectively. The distance from $u$ to a nonempty closed set $S\subseteq\mathbb{R}^n$ is defined by $\dist(u,S)=\min\{\|u-s\|\mid s\in S\}$.

\section{Preliminary}\label{preliminary}
In this paper, we consider balanced discrete OT problems and formulate them as finite-dimensional LP problems. Let $\U=\{u_1,\cdots,u_m\}$ and $\V=\{v_1,\cdots,v_n\}$ be the source and target sets, respectively. We assume that $2\leq m\leq n$. 
Let $\bm\mu$ and $\bm\nu$ be finite marginal measures on $\U$ and $\V$, respectively, with equal total mass, i.e.,
\[\bm\mu(u_i)=\mu_i\geq0,\,\bm\nu(v_j)=\nu_j\geq0,\,\mathbf{1}_m^\top\bm\mu=\mathbf{1}_n^\top\bm\nu,\quad 1\leq i\leq m,\,1\leq j\leq n.\]
Given a cost function $h:\U\times \V\to\mathbb{R}_+$, we define the associated cost matrix by $C=(C_{ij})\in\mathbb{R}^{m\times n}_+$, where $C_{ij}=h(u_i,v_j)$. The Kantorovich formulation is
\begin{equation}\label{OTpro}
\min\,\,\left<C,X\right>\quad
\textup{s.t.}\,X\mathbf{1}_n=\bm\mu,\,X^\top \mathbf{1}_m=\bm\nu,\,X_{ij}\ge 0,\,1\leq i\leq m,\,1\leq j\leq n.  
\end{equation}
Using properties of the Kronecker product, problem~\eqref{OTpro} is equivalent to
\begin{equation}\label{full}
 \min\,\, c^\top x\quad
 \textup{s.t.}\,\, \bar Ax=\bar b,\,x\in\mathbb{R}^{mn}_+,
\end{equation}
where $x=\vecc(X)$, $c=\vecc(C)$, $\bar b=(\bm\mu,\bm\nu)$, and
\begin{equation}\label{matrixA}
\bar A:=\begin{pmatrix}
			\mathbf{1}_n^\top\otimes I_m\\
				I_n\otimes \mathbf{1}^\top _m\end{pmatrix}\in\mathbb{R}^{(m+n)\times mn}.\end{equation}
If some $\mu_i$ or $\nu_j$ is zero, the corresponding row or column of $X$ can be directly removed. Accordingly, throughout the paper, we assume that $\bar b>0$ and define
\[\mathcal H(m,n)
:=\{(\bm\mu,\bm\nu)\,|\,
\bm\mu\in\mathbb{R}^m,\,\bm\nu\in\mathbb{R}^n,\,\bm 1_m^\top\bm\mu=\bm1_n^\top\bm\nu\},\,
\D(m,n):=\mathcal{H}(m,n)\cap\mathbb{R}^{m+n}_{++}.\]
The set $\mathcal{H}(m,n)$ is an $(m+n-1)$-dimensional subspace of $\mathbb{R}^{m+n}$.
For any $\bar b\in\D(m,n)$, define $\bar X=\frac{\bm\mu\bm\nu^\top}{\mathbf{1}_m^\top\bm\mu}$. Then $\bar x=\vecc(\bar X)$ is feasible for \eqref{full}. Hence, the feasible set $\F_{\bar b}:=\{x\mid \bar Ax=\bar b,\, x\in\mathbb{R}^{mn}_+\}$ is nonempty for every $\bar b\in\D(m,n)$.

Note that any optimal basic solution of \eqref{full} has at most $m+n-1$ nonzero components, which is typically smaller than $mn$. Motivated by this sparsity, we restrict problem~\eqref{full} to the variables indexed by a smaller active set $\N$:
   \begin{equation}\label{defineN}
      \N=\{(u_{i_t},v_{j_t})\}_{t=1}^{\cN}\subseteq\U\times\V.
   \end{equation}
  The elements of $\N$ are ordered according to their positions in the column-wise vectorization; that is,
$i_t+(j_t-1)m<i_{t+1}+(j_{t+1}-1)m$ for $1\leq t\leq\cN-1$.

Specifically, we consider the following sparse LP problem:
     \begin{equation}\label{sparse}
 \min\,\, c_\N^\top x\quad\textup{s.t.}\,\,\bar A_\N x=\bar b,\,x\in\mathbb{R}^{|\N|}_+,
\end{equation}
where $c_\N\in\mathbb{R}^{|\N|}$ is the cost vector restricted to $\N$, i.e., $(c_\N)_t=C_{i_tj_t}$, and $\bar A_\N \in\mathbb{R}^{(m+n)\times|\N|}$ is the submatrix of $\bar A$. The $t$-th column of $\bar A_\N $ is the $(i_t+(j_t-1)m)$-th column of $\bar A$.
The dual problem of \eqref{sparse} can be written as
\begin{equation*}
 \max\,\, \bm\mu^\top\alpha+\bm\nu^\top\beta\quad\textup{s.t.}\,\,\alpha_{i_t}+\beta_{j_t}+s_t= C_{i_tj_t},\,s_t\ge 0,\,(u_{i_t},v_{j_t})\in\N,\,1\le t\le \cN.
\end{equation*}

Each element $(u_{i_t},v_{j_t})\in\N$ represents an edge joining $u_{i_t}$ and $v_{j_t}$. Thus, $\bar A$ and $\bar A_\N$ are the incidence matrices of the bipartite graph $\G=(\U\cup\V,\U\times\V)$ and its spanning subgraph $\G_\N=(\U\cup\V,\N)$, respectively; see \cite[Section 8.2]{royle2001algebraic}. Since $\bar b>0$, feasibility of \eqref{sparse} requires $\G_\N$ to have no isolated vertices.

To facilitate the subsequent analysis, we introduce the following
notation and definitions. Let $\N$ be given by \eqref{defineN}. For a vector $x\in\mathbb{R}^{\cN}$, its support is defined by $\spt(x):=\{t\mid x_t>0,\,1\leq t\leq \cN\}$. We define the associated zero-filled matrix $\mat_\N(x)\in\mathbb R^{m\times n}$ by
\[
(\mat_\N (x))_{ij}
:=\begin{cases}
x_t, & \text{if } (u_i,v_j)=(u_{i_t},v_{j_t})\in\N,\\
0, & \text{otherwise}.
\end{cases}
\]
The zero extension of $x$ to $\mathbb{R}^{mn}$ is defined by
$\ext_\N (x):=\vecc\bigl(\mat_\N (x)\bigr)$. For a matrix $X=(X_{ij})\in\mathbb{R}^{m\times n}$, its support is defined by $\spt(X):=\{(i,j)\mid X_{ij}>0\}$. Correspondingly, its support on the complete bipartite graph $\G$ is defined by $\spt_\G(X):=\{(u_i,v_j)\,|\,(i,j)\in\spt(X)\}$. The restriction of $X$ to the entries indexed by $\N$ is denoted by $\vecc_\N(X):=(X_{i_1j_1},\cdots,X_{i_{\cN}j_{\cN}})^\top\in\mathbb{R}^{|\N|}$. In particular, $\vecc_\N(\mat_\N (x))=x$ for every $x\in\mathbb{R}^{\cN}$. 

To describe the vertex sets of the connected components, we next recall the notion of a partition.
A collection $\{\I_1,\cdots,\I_k\}$ is a partition of $\U$ if its members are nonempty and pairwise disjoint, and their union is $\U$. Such a partition is called nontrivial if $k\geq 2$. 

\section{Exploiting Sparsity and Graph Structure in OT Problems}\label{Active-set}
This section develops a framework for solving the full problem
\eqref{full}. First, we introduce a preprocessing and decomposition algorithm tailored to each subproblem~\eqref{sparse}. We then describe the active-set refinement strategy. Finally, we establish constraint nondegeneracy properties needed for the subsequent analysis.
    
\subsection{Preprocessing and Decomposition of a Sparse Subproblem}\label{reductionAN}
Given an active set $\N$, we propose a fast algorithm to test a necessary condition for the feasibility of the associated sparse subproblem~\eqref{sparse}, augmenting $\N$ with suitable edges whenever necessary. The resulting subproblem is then decomposed into smaller independent subproblems. This procedure is based on the following proposition, whose proof is provided in Appendix~\ref{proof:theorem-1}.

\begin{proposition}\label{rankA}
Let $\G_\N=(\U\cup\V,\N)$ be a bipartite graph containing no isolated vertices. Suppose that it has $c_0$ connected components, denoted by
\begin{equation}\label{connect}
\G_{\N_k}=(\I_k\cup\J_k,\N_k),
\quad \N_k=\N\cap(\I_k\times\J_k),
\quad 1\leq k\leq c_0,\end{equation}
where $\sum_{k=1}^{c_0}|\N_k|=|\N|$ and $\{\I_k\}_{k=1}^{c_0}$ and $\{\J_k\}_{k=1}^{c_0}$ are partitions of $\U$ and $\V$, respectively.
Then $\rank(\bar A_\N ) =m+n-c_0$ and problem~\eqref{sparse} can be decomposed into the following $c_0$ independent subproblems:
   \begin{equation}\label{subproblem}
\min\,\, c_{\N_k}^\top x_k\quad \textup{s.t.}\,\,\bar A_kx_k=\bar b_{\I_k,\J_k},\,x_k\in\mathbb{R}^{|\N_k|}_+,\quad 1\leq k\leq c_0,
\end{equation}
where $\bar A_k\in\mathbb{R}^{(|\I_k|+|\J_k|)\times|\N_k|}$ is the incidence matrix induced by $\G_{\N_k}$ and has rank $|\I_k|+|\J_k|-1$. Moreover, $c_{\N_k}\in\mathbb{R}^{|\N_k|}$ and $\bar b_{\I_k,\J_k}=(\bm\mu_{\I_k},\bm\nu_{\J_k})\in\mathbb{R}^{|\I_k|+|\J_k|}$ are the corresponding cost and marginal vectors, respectively. 
\end{proposition}

Proposition~\ref{rankA} shows that each connected component
$\G_{\N_k}$ induces a block $\bar A_k$, and each block
contains exactly one redundant row. 
We employ the breadth-first search (BFS) algorithm \cite[Section 22.2]{cormen2009introduction} to identify all connected components of the spanning subgraph $\G_\N=(\U\cup\V,\N)$. 
Roughly speaking, BFS starts from an initial vertex and explores all vertices reachable through adjacent vertices. Once no further unvisited adjacent vertices remain, the explored vertices form one connected component of $\G_\N$. Applying the same procedure to the remaining unvisited vertices yields all connected components of $\G_\N$.
Unlike the reduction method in~\cite{liu2022multiscale}, our method
avoids explicitly constructing the incidence matrix
$\bar A_\N \in\mathbb R^{(m+n)\times|\N|}$ while identifying the
redundant rows and yielding a decomposition of problem~\eqref{sparse}.

Suppose that all $c_0$ connected components of
$\G_\N $, as specified in \eqref{connect}, have been identified.
For each $k=1,\cdots,c_0$, let $
d_k:=\mathbf 1_{|\I_k|}^{\top}\bm\mu_{\I_k}-\mathbf 1_{|\J_k|}^{\top}\bm\nu_{\J_k}$ denote the mass imbalance of the $k$-th connected component. We then define the collection 
\begin{equation}\label{collection}
\mathscr{C}
=\left\{
(\I_k,\J_k,\N_k,d_k)
\right\}_{k=1}^{c_0}.\end{equation}
A necessary condition for the sparse
subproblem~\eqref{sparse} to be feasible is that the
mass-conservation condition $d_k=0$ hold for every
$1\leq k\leq c_0$.
Before removing redundant rows from $\bar A_\N$, we choose a tolerance $\varepsilon_0\ge0$ and check whether the mass imbalances
satisfy $\max_{1\leq k\leq c_0}|d_k|\leq \varepsilon_0$. If this condition fails, then at least one $d_k$ is nonzero. The index sets
\begin{equation}\label{l+l-}
\Gamma_+:=\{k\,{|}\, d_k>0,\,1\leq k\leq c_0\}\quad\textup{and}\quad
\Gamma_-:=\{k\,{|}\,d_k<0,\, 1\leq k\leq c_0\}\end{equation}
are both nonempty, since $\sum_{k\in \Gamma_+}d_k+\sum_{k\in\Gamma_-}d_k=\bm 1_m^\top\bm\mu-\bm1_n^\top\bm\nu= 0$.
We select indices $k_+\in\Gamma_+$ and $k_-\in\Gamma_-$ and add suitable edges from $\I_{k_+}\times \J_{k_-}$ and $\I_{k_-}\times \J_{k_+}$ to $\N$. 
These edges merge connected components that violate the mass-conservation condition. The detailed procedure is presented in Algorithm \ref{reduction}.

\begin{algorithm}[H]
\caption{Preprocessing and decomposition of a sparse subproblem}
\label{reduction}
\begin{algorithmic}[0]
\setlength{\itemsep}{0pt}
\setlength{\parsep}{0pt}
\setlength{\parskip}{0pt}

\Step{0}{Given $\G_\N=(\U\cup\V,\N)$, $\bar b=(\bm\mu,\bm\nu)$, and $\varepsilon_0\in[0,\min_{1\leq j\leq m+n}\bar b_j)$.}

\Step{1}{Use BFS to find the $c_0$ connected components of $\G_\N $ and form $\mathscr C$ as in~\eqref{collection}.}

\Step{2}{If $\max_{1\leq k\leq c_0}|d_k|\leq\varepsilon_0$, proceed directly to Step~4. Otherwise, define $
\Gamma_+$ and $\Gamma_-$ as in~\eqref{l+l-} and proceed to Step~3.}

\Step{3}{Choose $k_+\in\argmax_{k\in\Gamma_+}d_k$ and $k_-\in\argmin_{k\in\Gamma_-}d_k$. Select a nonempty set 
$\N_{\textup{add}}\subseteq (\I_{k_+}\times\J_{k_-})\cup(\I_{k_-}\times\J_{k_+})$. Update $\N\gets\N\cup\N_{\textup{add}}$ and 
$\G_\N\gets(\U\cup\V,\N)$. 
Update the collection of connected components as follows:
\[
\begin{aligned}
\mathscr C\leftarrow&
\mathscr C
\setminus
\big\{
(\I_{k_+},\J_{k_+},\N_{k_+},d_{k_+}),
(\I_{k_-},\J_{k_-},\N_{k_-},d_{k_-})
\big\}\\
&\cup
\big\{
(\I_{k_+}\cup\I_{k_-},
 \J_{k_+}\cup\J_{k_-},
 \N_{k_+}\cup\N_{k_-}\cup\N_{\textup{add}},
 d_{k_+}+d_{k_-})
\big\}.
\end{aligned}
\]

Set $c_0=c_0-1$. Relabel the elements of $\mathscr C$ by $1,\cdots,c_0$ and return to Step 2.}

\Step{4}{Decompose problem \eqref{sparse} into $c_0$ subproblems. For each connected component, delete a smallest component of
$\bm\nu_{\J_k}$ if $d_k\leq 0$, or of $\bm\mu_{\I_k}$ otherwise,
together with the corresponding row of $\bar A_k$.}

\end{algorithmic}
\end{algorithm}
Let $c_0'$ denote the initial number of connected components before Algorithm~\ref{reduction}. Since each
execution of Step~3 merges two connected components, Algorithm~\ref{reduction}
terminates after at most $c_0'-1$ iterations. It returns a repaired active
set $\N$ with $|d_k|\leq\varepsilon_0$ for each connected component and decomposes
the problem into $c_0$ reduced LPs. Although exact feasibility requires
$\varepsilon_0=0$, we allow a small tolerance $\varepsilon_0>0$ to account for
numerical errors. In this case, any solution $x_k^*$ to the $k$-th reduced LP obtained in Step~4 satisfies $\|\bar A_kx_k^*-\bar b_{\I_k,\J_k}\|_\infty=|d_k|\leq\varepsilon_0$.

\subsection{Active-Set Updates for Solving the Full Problem}\label{updateN}
Since the exact support of an optimal solution to the full problem~\eqref{full} is unknown a priori, the active set $\N$ is updated iteratively. Given an active set $\N^{[i]}$, let $x_{\N^{[i]}}$ be an optimal
solution to the sparse subproblem~\eqref{sparse} with $\N=\N^{[i]}$, and let $\lambda^{[i]}=(\alpha^{[i]},\beta^{[i]})$ be an optimal dual multiplier. Define $ x^{[i]}=\ext_{\N^{[i]}}( x_{\N^{[i]}})\in\mathbb{R}^{mn}_+$.
If the full problem~\eqref{full} admits an optimal solution $x^*$ such that $\spt_\G(\mat(x^*))\subseteq\N^{[i]}$,
then $ x^{[i]}$ is also optimal for problem~\eqref{full}. On the other hand, if $ x^{[i]}$ is not optimal
for problem~\eqref{full}, then $\lambda^{[i]}$ is infeasible for its dual problem.
Consequently, the union of the following sets of violated dual constraints
\begin{equation*}
\begin{aligned}
    \T^{[i]}_{\textup{rel}}&=\left\{(u_k,v_l)\mid(\alpha^{[i]}_k+\beta^{[i]}_l)/C_{kl}>1,\,C_{kl}>0
\right\},\\
\T^{[i]}_{\textup{abs}}&=\left\{(u_k,v_l)\mid\alpha^{[i]}_k+\beta^{[i]}_l-C_{kl}>0,\,C_{kl}=0\right\},
\end{aligned}
\end{equation*}
is nonempty, i.e., $|\T^{[i]}_{\text{rel}}|+|\T^{[i]}_{\text{abs}}|\ge 1$. Since this union may be large, we select only a limited number of its edges to construct the next active set $\N^{[i+1]}$. Arrange the elements of  $\T_{\mathrm{rel}}^{[i]}$ as
$\bigl((u_{\hat k_t},v_{\hat l_t})\bigr)_{t=1}^{|\T^{[i]}_{\text{rel}}|}$.
The elements are ordered by decreasing $(\alpha^{[i]}_{\hat k_t}+\beta^{[i]}_{\hat l_t})/C_{\hat k_t\hat l_t}$, with ties broken
by increasing $C_{\hat k_t\hat l_t}$ and then by increasing
$\hat k_t+(\hat l_t-1)m$.
Similarly, arrange the elements of $\T_{\textup{abs}}^{[i]}$ as $\bigl((u_{\tilde k_t},v_{\tilde l_t})\bigr)_{t=1}^{|\T^{[i]}_{\text{abs}}|}$. 
The elements are ordered by decreasing $\alpha_{\tilde k_t}^{[i]}+\beta_{\tilde l_t}^{[i]}$, with ties broken by increasing $\tilde k_t+(\tilde l_t-1)m$.
For integers $0\leq K_1\leq|\T^{[i]}_{\text{rel}}|$ and $0\leq K_2\leq|\T^{[i]}_{\text{abs}}|$, we select 
\begin{equation}\label{choose}
    \widehat{\T}^{[i]}(K_1)=\{(u_{\hat k_t},v_{\hat l_t})\}_{t=1}^{K_1}\quad\textup{and}\quad\widetilde{\T}^{[i]}(K_2)=\{(u_{\tilde k_t},v_{\tilde l_t})\}_{t=1}^{K_2}.
\end{equation}
Then we update the next active set $\N^{[i+1]}$ by $\N^{[i+1]}=\N^{[i]}\cup   \widehat{\T}^{[i]}(K_1)\cup\widetilde{\T}^{[i]}(K_2)$. 

Our update scheme is simpler than those proposed in \cite{liu2022multiscale,schmitzer2016sparse}. We omit the step of adding the geometric neighbors of $\N^{[i]}$ to $\N^{[i+1]}$. This simplification is possible because the initial set $\N^{[0]}$, constructed using the multiscale approach described later, already incorporates the relevant geometric information.
Building on
this initialization and update scheme, we propose Algorithm \ref{update} to obtain an optimal solution to the full problem~\eqref{full} by solving a sequence of sparse subproblems~\eqref{sparse}.
    
\begin{algorithm}[H]
\caption{Solving a sequence of sparse subproblems}
\label{update}
\begin{algorithmic}[0]
\setlength{\itemsep}{0pt}
\setlength{\parsep}{0pt}
\setlength{\parskip}{0pt}

\Step{0}{Input $\bar b,\,c$ and an initial active set $\N^{[0]}$. Choose $\theta_1,\theta_2>0$ and set $i=0$.}

\Step{1}{Apply Algorithm \ref{reduction} to preprocess the sparse
problem~\eqref{sparse} with $\N=\N^{[i]}$, thereby obtaining $c_0$
subproblems of the form~\eqref{subproblem}.}

\Step{2}{Solve these $c_0$ subproblems and assemble their
solutions to obtain a primal--dual solution
$(x_{\N^{[i]}},\lambda^{[i]},s_{\N^{[i]}})$ of the sparse problem~\eqref{sparse} with $\N=\N^{[i]}$.}

\Step{3}{Set $x^{[i]}=\ext_{\N^{[i]}}(x_{\N^{[i]}})$ and $s^{[i]}=c-\bar A^\top\lambda^{[i]}$. If the termination criterion is satisfied, return $(x^{[i]},\lambda^{[i]},s^{[i]})$; otherwise, go to Step~4.}

\Step{4}{Set $K_1=\min\{|\T^{[i]}_{\text{rel}}|,\lceil\theta_1|\N^{[i]}|\rceil\}$ and $K_2=\min\{|\T^{[i]}_{\text{abs}}|,\lceil\theta_2|\N^{[i]}|\rceil\}$.
According to the selection rule~\eqref{choose}, update the active set as $\N^{[i+1]}=\N^{[i]}\cup\widehat{\T}^{[i]}(K_1)\cup\widetilde{\T}^{[i]}(K_2)$. Set $i=i+1$ and return to Step 1.
}
\end{algorithmic}
\end{algorithm}

\begin{proposition}\label{converseries}
Suppose that the sparse problem~\eqref{sparse} associated with the
initial active set $\N^{[0]}$ is feasible. Let $\N^{[i]}$ be updated by Algorithm \ref{update}. For each iteration $i$, let $x_{\N^{[i]}}$ be an optimal solution of~\eqref{sparse} with $\N=\N^{[i]}$ and define its zero extension by $x^{[i]}=\ext_{\N^{[i]}}(x_{\N^{[i]}})$. Then there exists $i_0\ge 0$ such that $c^\top x^{[i]}\ge c^\top x^{[i+1]}$ for all $0\leq i\leq i_0-1$, and $ x^{[i_0]}$ is an optimal solution of~\eqref{full}.
\end{proposition}
\begin{proof}At any iteration before termination, since $ \T^{[i]}_{\textup{rel}}\cup
\T^{[i]}_{\textup{abs}}\ne\varnothing$, $\N^{[i]}\ne\varnothing$, and $\theta_1,\theta_2>0$, we have
$|\T^{[i]}_{\text{rel}}|+|\T^{[i]}_{\text{abs}}|\ge 1$ and $\lceil\theta_1|\N^{[i]}|\rceil,\lceil\theta_2|\N^{[i]}|\rceil\ge1$. Thus $K_1+K_2\ge 1$, and $\N^{[i]}\subsetneq\N^{[i+1]}$. Define $\bar x^{[i]}=\vecc_{\N^{[i+1]}}(\mat(x^{[i]}))$. Then $\bar x^{[i]}$ is feasible for the sparse problem with $\N=\N^{[i+1]}$. The optimality of $x_{\N^{[i+1]}}$ therefore implies
$c^\top x^{[i]}=c_{\N^{[i+1]}}^\top \bar x^{[i]}\ge c_{\N^{[i+1]}}^\top x_{\N^{[i+1]}}=c^\top x^{[i+1]}$. The full problem is finite-dimensional and the active set is strictly enlarged. Thus, there exists $i_0\ge 0$ such that $x^{[i_0]}$ is an optimal solution of~\eqref{full}.
\qed
\end{proof}

Algorithm~\ref{update} requires only the restricted cost vector
$c_\N\in\mathbb{R}^{\cN}$ when solving each sparse subproblem. If $C_{ij}$ is given by $C_{ij}=h(u_i,v_j)$, the required entries of $C$ can be generated on demand
to form $c_\N$, update the active set and detect violated dual constraints by evaluating $C_{ij}-\alpha_i-\beta_j$. Thus, the full cost matrix need not be stored, substantially reducing memory usage for large-scale problems.

The practical efficiency of Algorithm~\ref{update} depends primarily
on two factors.
First, the choice of $\N^{[0]}$ is crucial for reducing the number of active-set updates and accelerating the overall algorithm~\cite{liu2022multiscale}. Accordingly, Section~\ref{multiscale method} presents a multiscale framework that constructs the fine-scale initial active set from the support of a coarse-scale solution, while Section~\ref{supportiden} develops a strategy for identifying this support.
Second, an efficient LP solver tailored to the subproblems~\eqref{sparse} is needed. Since Algorithm~\ref{update} yields a
sequence of closely related sparse subproblems, the solver should exploit
their special structure and support warm starts. To this end, we
give a warm-start strategy in Section~\ref{multiscale method} and develop IPRM in Section~\ref{interior-point relaxation method}.

\subsection{Exploring the Constraint Nondegeneracy Condition}
We end this section by discussing the primal and dual constraint nondegeneracy conditions for problem \eqref{sparse}, which play a critical role in the local
convergence analysis of LP solvers~\cite{hou2024sparse} and in sensitivity
analysis~\cite{cartis2016active}.
By Proposition~\ref{rankA},
problem~\eqref{sparse} can be decomposed into $c_0$ subproblems, each of
the form~\eqref{subproblem}. It therefore suffices to consider a single connected component. Accordingly, we assume without loss of generality that the subgraph $\G_\N$ is connected, so that $\rank(\bar A_\N)=m+n-1$. Let $A$ and $A_\N $ be the matrices obtained by removing the last rows of $\bar A$ and $\bar A_\N $, respectively, and let $b$ be obtained by deleting the last entry of $\bar b$. We consider the following primal--dual LP pair:
\begin{equation*}
\begin{tabular*}{\textwidth}{@{\extracolsep{\fill}} l r @{}}
$\displaystyle \qquad\qquad\quad\min\,\, c_\N^\top x
\quad {\rm s.t.}\,\,A_\N x = b,\,x\in\mathbb{R}^{|\N|}_+$,
&
$(\hypertarget{pp}{\mathbf{P}})$
\\
$\displaystyle \qquad\qquad\quad\max\,\, b^\top\lambda
\quad {\rm s.t.}\,\,A_\N^\top \lambda+s = c_\N,\,
\lambda \in \mathbb{R}^{m+n-1},\,s\in \mathbb{R}^{|\N|}_+$,
&
$(\hypertarget{dp}{\mathbf{D}})$
\end{tabular*}
\end{equation*}
where $A_\N\in\mathbb{R}^{(m+n-1)\times \cN},\,c_\N\in\mathbb{R}^{\cN}$ and $ b\in\mathbb{R}^{m+n-1}$. 

Let $x_\N^*$ be an optimal solution of \eqf{pp}. Define $X^*=\mat_\N(x^*_\N)$ and $x^*=\ext_\N(x^*_\N)$. Then $x^*$ is feasible for \eqref{full}. 
Following \cite{robinson2009local} and \cite[Definition 2]{liang2025squared}, 
we say that $x^*_\N $ is primal constraint nondegenerate if $\rank(A_{\spt_\G(X^*)})=m+n-1$. 
By Proposition~\ref{rankA}, this condition is equivalent to the
connectivity of the subgraph $\G_{X^*}=(\U\cup\V,\spt_\G(X^*))$. Similarly, let $(\lambda^*,s_\N^*)$ be an optimal solution of~\eqf{dp}. Define
$S^*=\mat_\N(s^*_\N)$. Then we say that $(\lambda^*,s^*_\N)$ is dual constraint nondegenerate if $\rank(A_{\N\setminus\spt_\G(S^*)})=\cN-|\spt(s^*_\N)|$.
To characterize primal constraint nondegeneracy, we introduce the
following subset of $\D(m,n)$:
\begin{equation}\label{nondegene}
   \E(m,n)=\{\bar b\in\D(m,n)\mid
   \rank(A_{\spt_\G(\mat(x))})=m+n-1\text{ for all }x\in\Fb\}.
\end{equation}

Proposition~\ref{propo2} provides a necessary and sufficient condition for
$\bar b\in\E(m,n)$.

\begin{proposition}\label{propo2}
For a vector $\bar b=(\bm\mu,\bm\nu)\in\D(m,n)$, we have $\bar b\in\E(m,n)$ if and only if there exist no nontrivial partitions $\{\I_1, \I_2\}$ of $\U$ and $\{\J_1, \J_2\}$ of $\V$ such that
 \begin{equation}\label{summu}
 \mathbf{1}_{|\I_1|}^\top\bm\mu_{\I_1}=\mathbf{1}_{|\J_1|}^\top\bm\nu_{\J_1},
\quad \mathbf{1}_{|\I_2|}^\top\bm\mu_{\I_2}=\mathbf{1}_{|\J_2|}^\top\bm\nu_{\J_2}.\end{equation}
\end{proposition}
\begin{proof}
``$\Longrightarrow$'' Assume, for contradiction, that $\bar b\in\E(m,n)$ and that there exist nontrivial partitions $\{\I_1,\I_2\}$ of $\U$ and $\{\J_1,\J_2\}$ of $\V$ satisfying~\eqref{summu}. Then
$\bar b_1=(\bm\mu_{\I_1},\bm\nu_{\J_1})
\in\D(|\I_1|,|\J_1|)$ and
$\bar b_2=(\bm\mu_{\I_2},\bm\nu_{\J_2})
\in\D(|\I_2|,|\J_2|)$. Hence, the feasible sets
$\F_{\bar b_1}$ and $\F_{\bar b_2}$ are both nonempty. Choose
$\tilde{x}_1\in\F_{\bar b_1}$ and
$\tilde{x}_2\in\F_{\bar b_2}$, and construct
$\widetilde{X}
=\mat_{\I_1\times\J_1}(\tilde{x}_1)
 +\mat_{\I_2\times\J_2}(\tilde{x}_2)$.
Then
$\widetilde{X}\bm 1_n=\bm\mu$ and
$\widetilde{X}^{\top}\bm 1_m=\bm\nu$, so that
$\tilde{x}:=\operatorname{vec}(\widetilde{X})\in\F_{\bar b}$.
Moreover, the graph $(\U\cup\V,\spt_\G(\widetilde{X}))$ contains no edge
between $\I_1\cup\J_1$ and $\I_2\cup\J_2$ and is therefore
disconnected. This contradicts the assumption that $\bar b\in\E(m,n)$.

``$\Longleftarrow$'' Assume that no nontrivial partitions of $\U$ and $\V$ satisfy
\eqref{summu}, but that $\bar b\notin\E(m,n)$. By the
definition of $\E(m,n)$, there exists
$\tilde{x}\in\F_{\bar b}$ such that the graph $(\U\cup\V,\spt_\G(\mat(\tilde x)))$ is disconnected. Partition its connected components into two nonempty collections. This induces nontrivial partitions
$\{\I_1,\I_2\} $ of $\U$ and
$\{\J_1,\J_2\}$ of $\V$. Since there are no edges
between $\I_1\cup\J_1$ and $\I_2\cup\J_2$, the marginal constraints
imply \eqref{summu}, contradicting the assumption.\qed
\end{proof}
\begin{corollary}\label{aggregate}
Let $\{\I_1,\cdots,\I_p\}$ and $\{\J_1,\cdots,\J_q\}$ be nontrivial partitions of $\U$ and $\V$, respectively. Suppose that $\bar b=(\bm\mu,\bm\nu)\in\mathcal{E}(m,n)$. Define the aggregated marginals $\tilde {\bm\mu}\in\mathbb{R}^p$ and $\tilde  {\bm\nu}\in\mathbb{R}^q$ componentwise by
$\tilde{\mu}_k=\sum_{u_i\in\I_k}\mu_i$, $1\leq k\leq p $ and $\tilde\nu_l=\sum_{v_j\in\J_l}\nu_j$, $1\leq l\leq q$.
Then $(\tilde{\bm\mu},\tilde{\bm\nu})\in\E(p,q)$.
\end{corollary}
\begin{proof}
The conclusion follows from the definition of $\E(m,n)$ and Proposition~\ref{propo2}.\qed
\end{proof}

Here, we provide an example of a vector in $\mathcal{E}(m,n)$. Let $\bm\nu=(e^{m+1},\cdots,e^{m+n})$ and $\bm\mu=(e^1,\cdots,e^{m-1},\mathbf{1}_n^\top\bm\nu-\sum_{j=1}^{m-1}e^{j})$, where $n\geq m$ and $e$ denotes Euler's number. Then $\bar b=(\bm\mu,\bm\nu)\in\mathcal{E}(m,n)$.
Suppose, to the contrary, that $\bar b\notin\E(m,n)$. By Proposition~\ref{propo2}, there exist nontrivial partitions $\{\I_1,\I_2\}$ of $\U$ and $\{\J_1, \J_2\}$ of $\V$ such that the equalities in \eqref{summu} hold.
Without loss of generality, assume that $u_m\in\I_2$. However, the first equality in \eqref{summu} then yields $\sum_{u_i\in\I_1}e^i-\sum_{v_j\in\J_1}e^{m+j}=0$.
This contradicts the transcendence of $e$. Therefore, $\bar b\in\E(m,n)$.

 Equip the space $\mathcal H(m,n)$ with the
$(m+n-1)$-dimensional Lebesgue measure. With respect to this measure, Theorem~\ref{zeromeasure} states that almost every
$\bar b\in\D(m,n)$ belongs to $\E(m,n)$. Its proof is provided in Appendix~\ref{proof:theorem-0}.

\begin{theorem}\label{zeromeasure}
For almost every $\bar b\in\D(m,n)$, the following statement holds:
for every active set $\N$ such that $\G_\N$ is connected and~\eqf{pp} is feasible, and for every cost vector
$c_\N\in\mathbb{R}^{|\N|}_+$, any optimal solution of~\eqf{pp} is
primal constraint nondegenerate.
\end{theorem}

The conclusion of Theorem~\ref{zeromeasure} is stronger than that of \cite[Lemma~2]{ge2025interior}. It shows that a randomly chosen
$\bar b\in\D(m,n)$ belongs to $\E(m,n)$ almost surely.
Consequently, this conclusion is independent of the choices of the
active set $\N$ and the cost vector
$c_\N\in\mathbb{R}^{|\N|}$, provided that $\G_\N$ is connected and~\eqf{pp} is feasible. The primal constraint nondegeneracy condition, in turn, guarantees that the dual optimal solution of \eqf{dp} is unique \cite[Theorem 4.5]{sierksma2001linear}. This result will play a key role in the construction of our multiscale framework.

\section{A Multiscale Primal--Dual Interior-Point Relaxation Method}\label{sectioniprm}
This section presents MSIPRM for solving large-scale OT problems. The multiscale framework constructs a sequence of sparse subproblems
on successively finer scales to solve the original large-scale OT
problem.
At each scale, IPRM is used to solve the resulting sparse subproblems. As an efficient LP solver, IPRM supports effective warm starts, so it is well suited to the multiscale framework.

\subsection{Multiscale Method}\label{multiscale method}
As discussed in Section \ref{updateN}, the choice of the initial active set $\N^{[0]}$ is crucial to the computational efficiency of Algorithm~\ref{update}. 
Multiscale methods have been widely used in OT
\cite{liu2022multiscale,oberman2015efficient,schmitzer2016sparse}. They provide an effective way to construct an initial active set
$\N^{[0]}$ for a finer-scale problem from the support of the coarser-scale solution.

For ease of presentation, we consider OT problems between two grayscale images, so that $\U$ and $\V$ inherit natural grid structures. The cost function $h(u,v)$ is given by $h(u,v)=g(u-v)$, where $g(\cdot)$ is continuous. A typical example is the squared Euclidean cost $g(\cdot)=\|\cdot\|^2$. The discrete measures are obtained by normalizing grayscale values.
The multiscale framework is motivated by both geometric and theoretical considerations.
Geometrically, Figure~\ref{multipic} illustrates that the supports of 
high-accuracy numerical solutions exhibit similar patterns across distinct resolutions.
Theoretically, the framework relies on stability results for OT problems; see~\cite[Theorem 5.20]{villani2008optimal} and~\cite[Theorem 4]{oberman2015efficient}. We next describe the construction of our multiscale framework.

\begin{figure}
\captionsetup{font=small}
    \centering
        \includegraphics[width=0.16\textwidth]{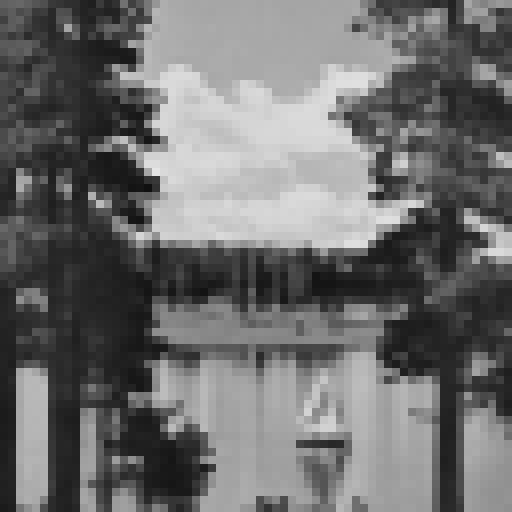}
        \includegraphics[width=0.16\textwidth]{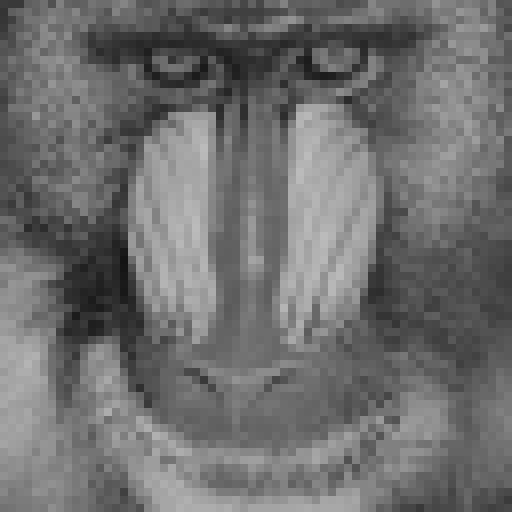}
        \includegraphics[width=0.16\textwidth]{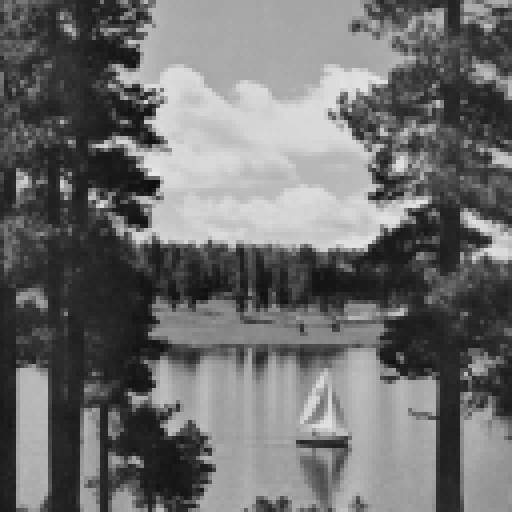}
        \includegraphics[width=0.16\textwidth]{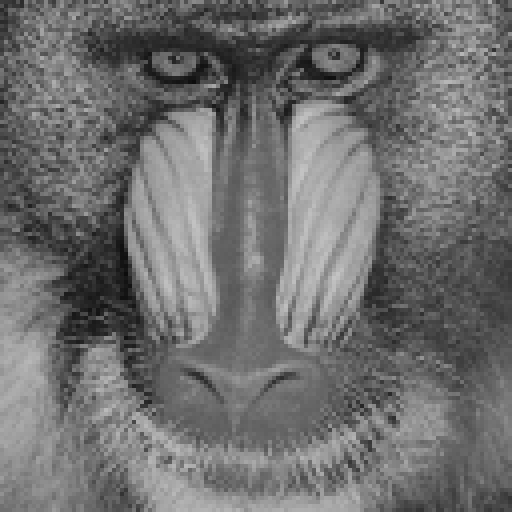}
        \includegraphics[width=0.16\textwidth]{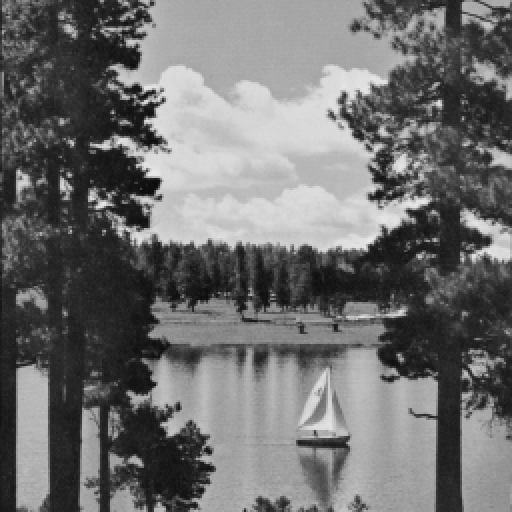}
        \includegraphics[width=0.16\textwidth]{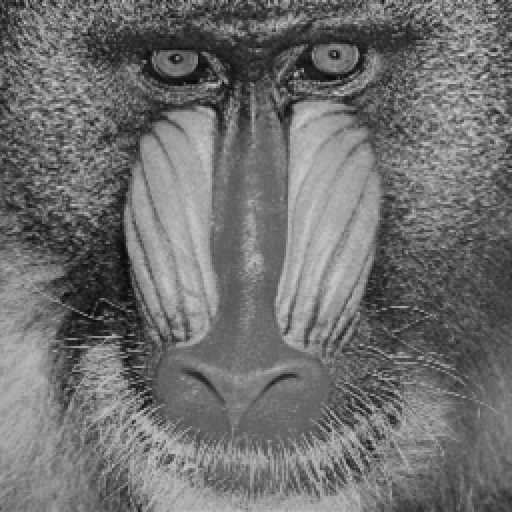}
        \includegraphics[width=0.32\textwidth]{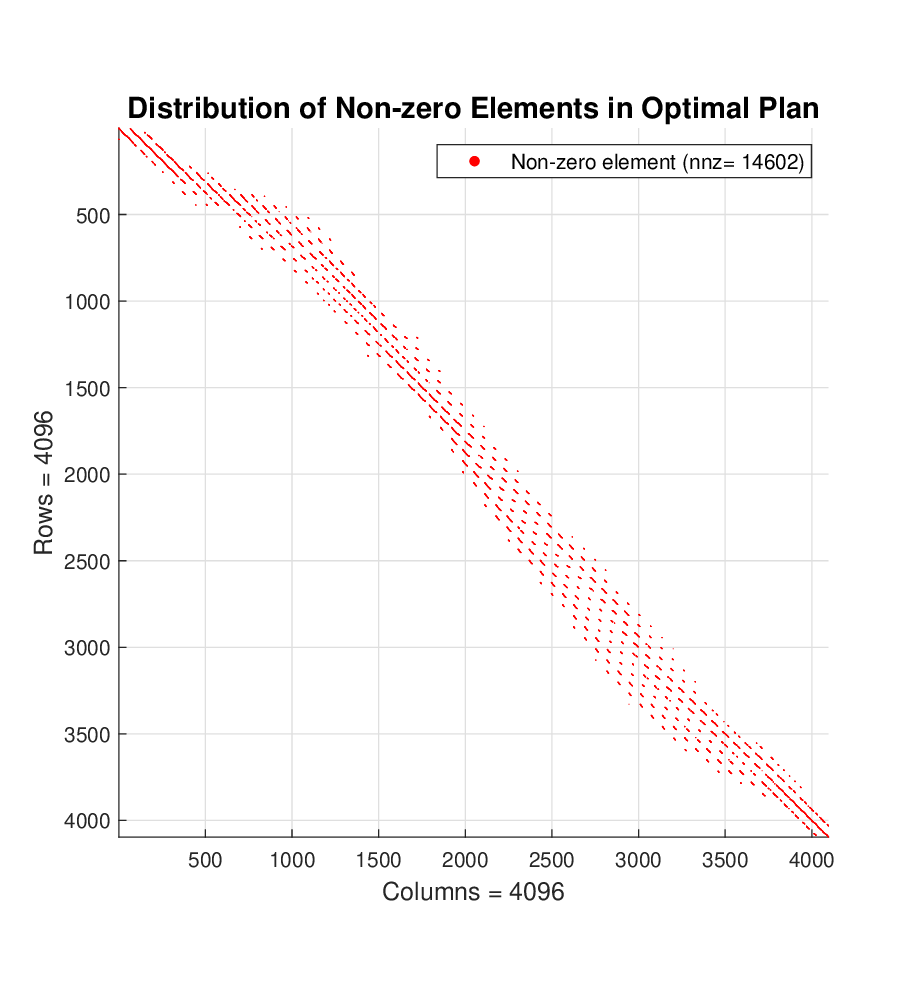}
        \includegraphics[width=0.32\textwidth]{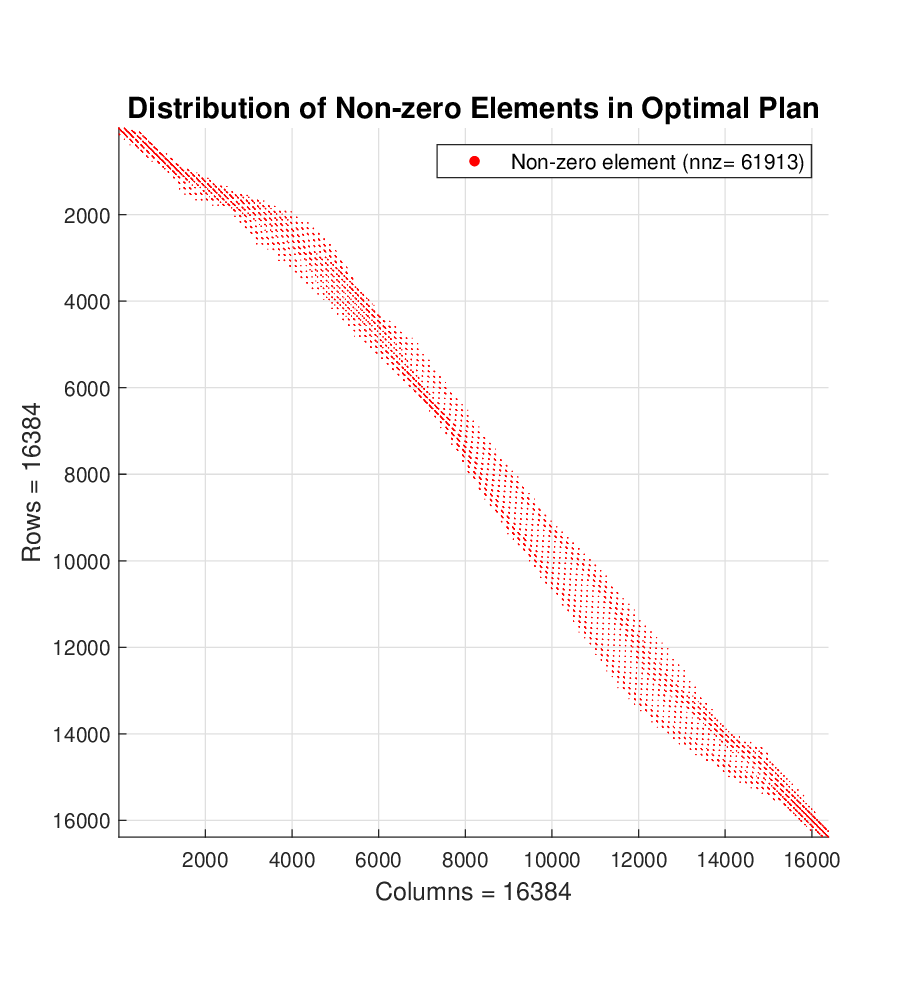}
        \includegraphics[width=0.32\textwidth]{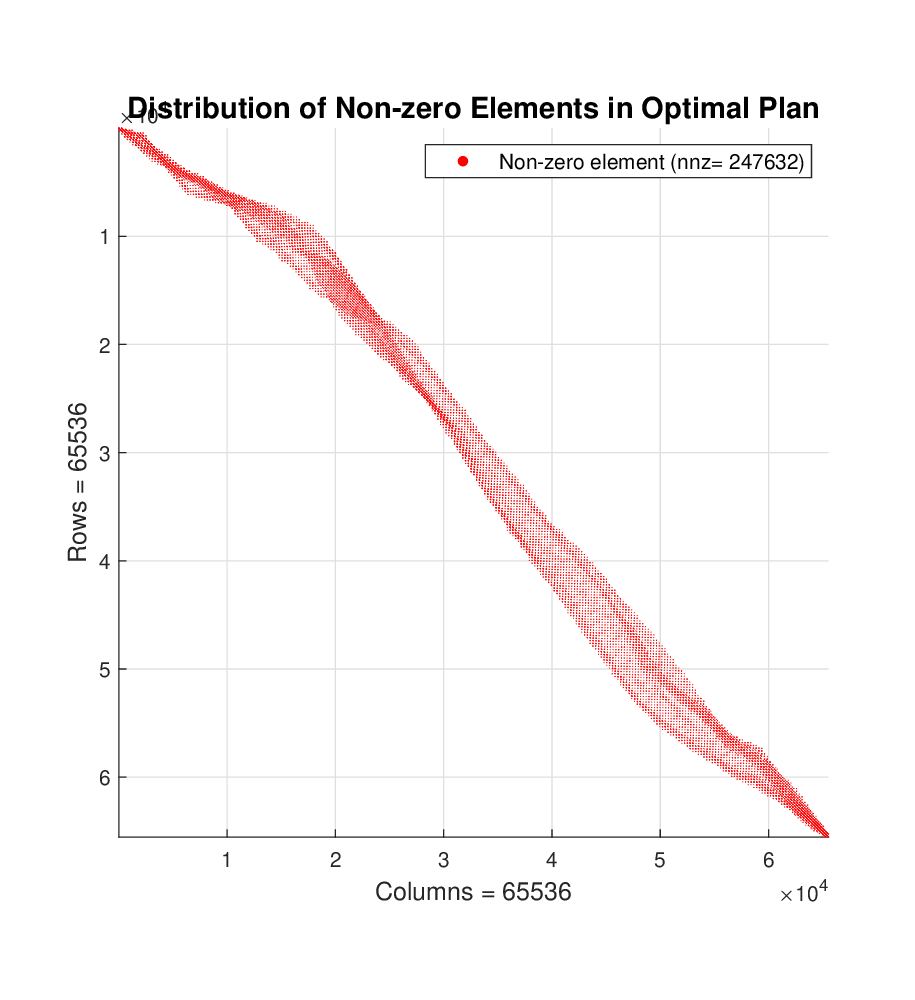}
 \caption{Upper panel: two images at resolutions ranging
    from $64\times 64$ to $256\times 256$. Lower panel: sparsity patterns of the approximate optimal transport plans obtained by discarding entries smaller than 1E-10. The KKT residuals of the approximate solutions are below 1E-12}
 \label{multipic}
\end{figure}

Let $\U$ be a finite set and $K\in\mathbb{Z}_+$. A hierarchical partition of $\U$ is an ordered tuple $(\U_0,\cdots,\U_K)$ such that $\U_0=\{\{u\}\mid u\in \U\}$ and, for each $1\leq\ell\leq K$, $\U_\ell$ is a partition of
$\U$ obtained by merging cells of
$\U_{\ell-1}$. For each $1\leq\ell\leq K$, we write $\U_\ell=\{\bm u_{\ell,i}\}_{i=1}^{m_\ell}$, where $m_\ell:=|\U_\ell|$. More precisely, for each
$\bm u_{\ell,i}\in\U_\ell$, define $
    \child(\bm u_{\ell,i}) :=\left\{
        \bm u_{\ell-1,r}\in\U_{\ell-1}
        \,\middle|\,
        \bm u_{\ell-1,r}\subset\bm u_{\ell,i}\right\}$.
Then $\bm u_{\ell,i}=\bigcup_{\bm u_{\ell-1,r}\in\child(\bm u_{\ell,i})}\bm u_{\ell-1,r}$.
The cells in $\child(\bm u_{\ell,i})$ are called the children of
$\bm u_{\ell,i}$. This hierarchy naturally induces a directed forest whose nodes are the cells in $\U_0,\cdots,\U_K$. For each
$\bm u_{\ell,i}\in\U_\ell$, we choose a representative point
$\rep(\bm u_{\ell,i})$.
Figure~\ref{multiscalex} provides a visual illustration on a grid of $8\times 8$. A hierarchical partition of $\V$ is
constructed in the same manner. We write $\V_\ell=\{\bm v_{\ell,j}\}_{j=1}^{n_\ell}$, where $n_\ell:=|\V_\ell|$ and $1\leq \ell\leq K$.

Let $\bm\mu$ and $\bm\nu$ be discrete measures on $\U$ and
$\V$, respectively. After constructing the hierarchical partitions of $\U$ and $\V$, we define the measures at level $\ell$ by
\begin{equation}\label{multimeasure}
        \bm\mu_\ell(\bm u_{\ell,i})=\sum_{u\in\bm u_{\ell,i}}\bm\mu(u),\quad 
    \bm\nu_\ell(\bm v_{\ell,j})=\sum_{v\in\bm v_{\ell,j}}\bm\nu(v),\quad 1\leq i\leq m_\ell,\,1\leq j\leq n_\ell.
\end{equation}
Consequently, $\bar b_\ell=(\bm\mu_\ell,\bm\nu_\ell)\in\D(m_\ell,n_\ell)$. As shown in Figure \ref{multiscalex}, assume that
$\bm u_{\ell,i}$ and
$\bm v_{\ell,j}$ are contained in hypercubes
centered at $\hat u_{\ell,i}$ and $\hat v_{\ell,j}$, respectively. Their representatives are generated by randomly perturbing their centers
\begin{equation*}
    \rep(\bm u_{\ell,i})=\hat u_{\ell,i}+p_{\bm u_{\ell,i}},
    \quad
    \rep(\bm v_{\ell,j})=\hat v_{\ell,j}+p_{\bm v_{\ell,j}},\quad1\leq i\leq m_\ell,\,1\leq j\leq n_\ell,
\end{equation*}
where $p_{\bm u_{\ell,i}}$ and
$p_{\bm v_{\ell,j}}$ are random points such that $\rep(\bm u_{\ell,i})$ and $\rep(\bm v_{\ell,j})$ lie within their respective hypercubes. The cost function at level $\ell$ is then defined by
\begin{equation}\label{costk}
    h_\ell: \U_\ell\times \V_\ell\to\mathbb{R}_+,\quad h_\ell(\bm u_{\ell,i},\bm v_{\ell,j})=h(\rep(\bm u_{\ell,i}),\rep(\bm v_{\ell,j})),\quad 1\leq \ell\leq K.
\end{equation}
The cost matrix $C_\ell $ and cost vector $c_\ell$ are induced by $h_\ell$. 

The multiscale method begins by solving a small, dense OT problem at the coarsest level $K$. Suppose that an optimal solution $x_{\ell+1}^*$ has been obtained at level $\ell+1$, where $0\leq \ell\leq K-1$. Let $X_{\ell+1}^*=\mat(x^*_{\ell+1})\in\mathbb{R}^{m_{\ell+1}\times n_{\ell+1}}$ be the  optimal transport plan.
We construct the initial active set at level $\ell$ as
\begin{equation}\label{initial}
\N_\ell^{[0]}=\bigcup_{(p,q)\in\spt(X^*_{\ell+1})}\child(\bm u_{\ell+1,p})\times\child(\bm v_{\ell+1,q}).
\end{equation}
We then apply Algorithm~\ref{update} to compute an optimal solution at level $\ell$. Repeating this procedure for
$\ell=K-1,\cdots,0$ yields an optimal solution of the original finest-level OT problem. In what follows, we describe how previously computed solutions are used to warm-start the solver for each sparse subproblem.

 First, suppose that $(\alpha_{\ell+1}^*,\beta_{\ell+1}^*)$ is the dual multiplier associated with $X^*_{\ell+1}$ at level $\ell+1$, where $\alpha_{\ell+1}^*\in\mathbb{R}^{m_{\ell+1}}$ and $\beta_{\ell+1}^*\in\mathbb{R}^{n_{\ell+1}}$. Let the initial active set $\N_\ell^{[0]}$ be generated by~\eqref{initial}. 
For each
$(\bm u_{\ell,i},\bm v_{\ell,j})\in
\U_\ell\times\V_\ell$, there exists a unique pair
$(\bm u_{\ell+1,p},\bm v_{\ell+1,q})\in\U_{\ell+1}\times\V_{\ell+1}$ such that
$\bm u_{\ell,i}\in\child(\bm u_{\ell+1,p})$ and $\bm v_{\ell,j}\in\child(\bm v_{\ell+1,q})$.
We now define an initial transport plan $X_\ell^{(0)}\in\mathbb{R}^{m_\ell\times n_\ell}$ and an approximate cost matrix
$\widehat C_\ell=((\widehat C_\ell)_{ij})\in\mathbb{R}^{m_\ell\times n_\ell}$ as follows:

\begin{equation*}
    (X_\ell^{(0)})_{ij}
    =(X^*_{\ell+1})_{pq}
    \frac{\bm\mu_\ell(\bm u_{\ell,i})}
         {\bm\mu_{\ell+1}(\bm u_{\ell+1,p})}
    \frac{\bm\nu_\ell(\bm v_{\ell,j})}
         {\bm\nu_{\ell+1}(\bm v_{\ell+1,q})},\quad(\widehat C_\ell)_{ij}=h_{\ell+1}(\bm u_{\ell+1,p},\bm v_{\ell+1,q}).
\end{equation*}
Define $(\alpha_\ell^{(0)})_i=(\alpha^*_{\ell+1})_p$ and $(\beta^{(0)}_\ell)_j=(\beta^*_{\ell+1})_q$. The initial point for the sparse problem
\eqref{sparse} with $\N=\N_\ell^{[0]}$ is chosen as
\begin{equation}\label{start1} x^{(0)}=\vecc_{\N^{[0]}_\ell}(X_\ell^{(0)}),\, \lambda^{(0)}=(\alpha_\ell^{(0)},\beta_\ell^{(0)}),\,s^{(0)}=c_{\N^{[0]}_\ell}-\bar A^\top_{\N^{[0]}_\ell}\lambda^{(0)}.\end{equation}

Next, let $\N_\ell^{[i]}$ be the active set associated with the $i$-th
sparse problem at level $\ell$, where $i\ge 0$. Suppose that
$(x_\ell^{[i]},\lambda_\ell^{[i]},s_\ell^{[i]})\in\mathbb{R}^{|\N_\ell^{[i]}|}\times\mathbb{R}^{m_\ell+n_\ell}\times \mathbb{R}^{|\N_\ell^{[i]}|}$ is an optimal primal--dual solution of~\eqref{sparse} with
$\N=\N_\ell^{[i]}$.
Based on the update strategy for $\N_\ell^{[i+1]}$ in Algorithm~\ref{update}, the initial point for
the $(i+1)$-th sparse problem is constructed by retaining the primal
variables indexed by $\N_\ell^{[i]}$ and setting the newly introduced
ones to zero:
\begin{equation}\label{start2}
    x^{(0)}=\vecc_{\N_\ell^{[i+1]}}(\mat_{\N^{[i]}_\ell}(x^{[i]}_\ell)),\quad
    \lambda^{(0)}=\lambda_\ell^{[i]},\quad
    s^{(0)}
    =c_{\N_\ell^{[i+1]}}
     -\bar A_{\N_\ell^{[i+1]}}^\top\lambda_\ell^{[i]}.
\end{equation}

By \cite[Lemma 8]{oberman2015efficient}, $
\bigl(\vecc(X_\ell^{(0)}),\lambda^{(0)},\hat c_\ell-\bar A_\ell^\top\lambda^{(0)})$
is a primal--dual optimal solution of the standard OT problem with
marginal vector $\bar b_\ell$ and cost vector $\hat c_\ell=\vecc(\widehat C_\ell)$. The initialization in~\eqref{start1} therefore
retains the information from the optimal solution of
the perturbed OT problem and it is feasible for the sparse problem~\eqref{sparse} with active set $\N^{[0]}_\ell$.
Similarly, the initialization in~\eqref{start2} is feasible for the sparse problem \eqref{sparse} with active set
$\N_\ell^{[i+1]}$. However, the dual slacks constructed in \eqref{start1} and
\eqref{start2} are not guaranteed to be nonnegative. Hence, these initializations cannot be
directly used as starting points for classical interior-point methods. This limitation motivates the use of IPRM~\cite{liu2022primal}, which
accepts non-interior initial points and is therefore well suited to the warm-start strategy.

The main distinction between our framework and the existing
multiscale methods in~\cite{liu2022multiscale,oberman2015efficient,schmitzer2016sparse} is that we perturb the cost function $h_\ell$ at each level. As discussed above, the support of the optimal solution must be identified from an approximate numerical solution. Solving every
subproblem to very high accuracy would be computationally expensive
and unnecessary.
However, degeneracy in the subproblems may hinder support identification.
Theorem \ref{zeromeasure} shows that almost every $\bar b$ in $\D(m,n)$ belongs to $\E(m,n)$. Furthermore, Corollary~\ref{aggregate} shows that if $\bar b\in\E(m,n)$, then $\bar b_\ell=(\bm\mu_\ell,\bm\nu_\ell)\in\E(m_\ell,n_\ell)$ for all $1\leq \ell\leq K$. 
Combining this inheritance property with the perturbations of
$h_\ell$, we deduce from Lemma~1 of \cite{ge2025interior} that the
resulting subproblems almost surely admit unique solutions satisfying the primal
constraint nondegeneracy condition. These properties are crucial for support identification.
It remains to develop an effective strategy for support
identification and to integrate this strategy into the multiscale framework.
\begin{figure}
    \centering 
    \captionsetup{font=small}

    \begin{minipage}[c]{0.51\textwidth}
        \centering
        \adjustbox{width=1.0\textwidth,valign=c}{
            \begin{tikzpicture}[
                scale=1.8,
                every node/.style={font=\scriptsize},
                blacknode/.style={circle, fill=black, inner sep=0.7pt},
                bluenode/.style={circle, fill=blue!65, draw=black, inner sep=1.2pt},
                greennode/.style={circle, fill=green!45!black, draw=black, inner sep=1.7pt},
                rednode/.style={circle, fill=red, draw=black, inner sep=2.2pt},
                dashline/.style={gray!80!black, dashed, line width=0.8pt},
                blueline/.style={blue!65, line width=1pt},
                greenline/.style={green!45!black, line width=1pt},
                redline/.style={red!65, line width=1pt}
            ]

\draw[redline] (0,0) rectangle (4,4);

\draw[blueline] (1,0) -- (1,4);
\draw[blueline] (3,0) -- (3,4);
\draw[blueline] (0,1) -- (4,1);
\draw[blueline] (0,3) -- (4,3);

\draw[greenline] (2,0) -- (2,4);
\draw[greenline] (0,2) -- (4,2);

\foreach \x in {0.5,1.5,2.5,3.5}
{
    \draw[dashline] (\x,0) -- (\x,4);
}
\foreach \y in {0.5,1.5,2.5,3.5}
{
    \draw[dashline] (0,\y) -- (4,\y);
}

\foreach \i in {0,1,2,3}
{
    \foreach \j in {0,1,2,3}
    {
        \pgfmathsetmacro{\xL}{\i+0.25}
        \pgfmathsetmacro{\xR}{\i+0.75}
        \pgfmathsetmacro{\yB}{\j+0.25}
        \pgfmathsetmacro{\yT}{\j+0.75}

        \draw[dashline] (\xL,\yB) rectangle (\xR,\yT);

        \node[blacknode] at (\xL,\yT) {};
        \node[blacknode] at (\xR,\yT) {};
        \node[blacknode] at (\xL,\yB) {};
        \node[blacknode] at (\xR,\yB) {};
    }
}

\foreach \row in {3,2,1,0}
{
    \foreach \col in {0,1,2,3}
    {
        \pgfmathsetmacro{\xL}{\col+0.25}
        \pgfmathsetmacro{\xR}{\col+0.75}
        \pgfmathsetmacro{\yB}{\row+0.25}
        \pgfmathsetmacro{\yT}{\row+0.75}

        \pgfmathtruncatemacro{\r}{3-\row}
        \pgfmathtruncatemacro{\c}{\col}

        \pgfmathtruncatemacro{\bigrow}{int(\r/2)}
        \pgfmathtruncatemacro{\bigcol}{int(\c/2)}
        \pgfmathtruncatemacro{\bigblock}{2*\bigrow + \bigcol}

        \pgfmathtruncatemacro{\localrow}{mod(\r,2)}
        \pgfmathtruncatemacro{\localcol}{mod(\c,2)}
        \pgfmathtruncatemacro{\localblock}{2*\localrow + \localcol}

        \pgfmathtruncatemacro{\base}{16*\bigblock + 4*\localblock + 1}

        \pgfmathtruncatemacro{\nTL}{\base}
        \pgfmathtruncatemacro{\nTR}{\base+1}
        \pgfmathtruncatemacro{\nBL}{\base+2}
        \pgfmathtruncatemacro{\nBR}{\base+3}

        \node[anchor=south] at (\xL+0.0,\yT+0.005)
            {$u_{\nTL}$};

        \node[anchor=south] at (\xR+0.0,\yT+0.005)
            {$u_{\nTR}$};

        \node[anchor=north] at (\xL+0.0,\yB-0.005)
            {$u_{\nBL}$};

        \node[anchor=north] at (\xR+0.0,\yB-0.005)
            {$u_{\nBR}$};
    }
}
 




%
\node[bluenode] at (0.54,3.50) {};
\node[blue!80!black, anchor=south] at (0.60,3.49) {$u_{1,1}$};

\node[bluenode] at (1.56,3.58) {};
\node[blue!80!black, anchor=south west] at (1.50,3.52) {$u_{1,2}$};

\node[bluenode] at (0.58,2.44) {};
\node[blue!80!black, anchor=north west] at (0.52,2.5) {$u_{1,3}$};

\node[bluenode] at (1.50,2.56) {};
\node[blue!80!black, anchor=south west] at (1.50,2.50) {$u_{1,4}$};

\node[bluenode] at (0.37,1.58) {};
\node[blue!80!black, anchor=north] at (0.37,1.55) {$u_{1,9}$};

\node[bluenode] at (1.55,1.58) {};
\node[blue!80!black, anchor=south west] at (1.52,1.51) {$u_{1,10}$};

\node[bluenode] at (0.58,0.58) {};
\node[blue!80!black, anchor=south west] at (0.52,0.52) {$u_{1,11}$};

\node[bluenode] at (1.57,0.42) {};
\node[blue!80!black, anchor=south west] at (1.53,0.45) {$u_{1,12}$};

\node[bluenode] at (2.54,3.58) {};
\node[blue!80!black, anchor=south west] at (2.52,3.52) {$u_{1,5}$};

\node[bluenode] at (3.48,3.54) {};
\node[blue!80!black, anchor=south west] at (3.46,3.48) {$u_{1,6}$};

\node[bluenode] at (2.49,2.52) {};
\node[blue!80!black, anchor=south west] at (2.47,2.46) {$u_{1,7}$};

\node[bluenode] at (3.50,2.42) {};
\node[blue!80!black, anchor=north west] at (3.46,2.50) {$u_{1,8}$};

\node[bluenode] at (2.57,1.58) {};
\node[blue!80!black, anchor=south west] at (2.54,1.52) {$u_{1,13}$};

\node[bluenode] at (3.48,1.42) {};
\node[blue!80!black, anchor=north west] at (3.44,1.50) {$u_{1,14}$};

\node[bluenode] at (2.44,0.55) {};
\node[blue!80!black, anchor=north west] at (2.30,0.50) {$u_{1,15}$};

\node[bluenode] at (3.57,0.42) {};
\node[blue!80!black, anchor=south west] at (3.47,0.44) {$u_{1,16}$};

\node[greennode] at (1.11,3.09) {};
\node[green!40!black, anchor=south] at (1.1,3.11)
    {\small{$u_{2,1}$}};

\node[greennode] at (0.90,1.05) {};
\node[green!40!black, anchor=south] at (0.89,1.07)
    {\small{$u_{2,3}$}};

\node[greennode] at (3.06,3.01) {};
\node[green!40!black, anchor=south] at (3.01,3.03)
    {\small{$u_{2,2}$}};

\node[greennode] at (3.08,0.90) {};
\node[green!40!black, anchor=south] at (3.09,0.94)
    {\small{$u_{2,4}$}};

\node[rednode] at (2.08,2.13) {};
\node[red, anchor=south] at (2.14,2.19)
    {\small{$u_{3,1}$}};

            \end{tikzpicture}
        }
    \end{minipage}
    \hfill
    \begin{minipage}[c]{0.47\textwidth}
        \centering
        \begin{equation*}
        \begin{aligned}
        &\U_0 = \big{\{}\{u_i\}\big{\}}_{1\leq i\leq 64}; \\ 
        &\U_1 = \big{\{}\{u_{4i-3},\cdots, u_{4i}\}\big{\}}_{1\leq i\le16}; \\
        &\U_2 = \big{\{}\{u_{16i-15}, \cdots,u_{16i}\}\big{\}}_{1\leq i\leq 4}; \\ 
        &\U_3 = \big{\{}\{u_1, \cdots, u_{64}\}\big{\}}.\\
        &\rep(\{u_{4i-3},\cdots, u_{4i}\})=u_{1,i},\,1\leq i\leq 16;\\
        &\rep(\{u_{16i-15}, \cdots,u_{16i}\})=u_{2,i},\,1\leq i\leq 4;\\ 
        &\rep(\{u_1, \cdots, u_{64}\})=u_{3,1}.\\
        &\child(\{u_1,\cdots,u_{64}\})=\\
        &\quad \big{\{}\{u_{16i-15}, \cdots,u_{16i}\}\big{\}}_{1\leq i\leq 4};\\
        &\textup{For}\,1\leq i\leq 4,\,\child(\{u_{16i-15}, \cdots,u_{16i}\})=\\
        &\quad\big{\{}\{u_{16i+4j-19},\cdots,u_{16i+4j-16}\}\big{\}}_{1\leq j\leq 4};\\
        &\textup{For}\,1\leq i\leq 16,\,\child(\{u_{4i-3},\cdots, u_{4i}\})=\\
        &\quad\big{\{}\{u_{4i-3}\},\cdots,\{u_{4i}\}\big{\}}.\\[3mm]
        \end{aligned} 
        \end{equation*}
        
    \end{minipage}

    \caption{
    An example of the multiscale scheme for a discrete set on
    an $8\times8$ grid. Here, $u_{\ell,j}$ denotes the representative
    point of the $j$-th cell at level $\ell$
    }   
    \label{multiscalex} 
\end{figure}

\subsection{A Primal--Dual Interior-Point Relaxation Method}\label{interior-point relaxation method}
In this subsection, we develop IPRM to solve each OT subproblem.
Since every level corresponds to a standard OT problem, we suppress the level index $\ell$ in level-dependent notation throughout this subsection. For example, we write $m$ and $n$ for $m_\ell$ and $n_\ell$, respectively. Recall Proposition~\ref{rankA} and the problem decomposition in~\eqref{subproblem}. We assume that $A_\N$ has full row rank and that $2\leq m\leq n$. Then $|\N|\ge m+n-1$ and the system $A_\N x=0,\,x\geq 0$ has only the zero solution.

By introducing a variable $z\in\mathbb{R}^{\cN}$, we construct an equivalent positive relaxation problem of the logarithmic-barrier subproblem for \eqf{pp}:
\begin{equation*}
\min\,\,c_\N^\top x-\mu\sum_{j=1}^{|\N|}\textup{ln}\,z_j\quad \textup{s.t.}\,\,A_\N x=b,\,z-x=0,\,x\in\mathbb{R}^{\cN},\,z\in \mathbb{R}^{\cN}_{++},
\end{equation*}
where $\mu>0$ is the barrier parameter. To handle the newly introduced constraint $z-x=0$, we consider the augmented Lagrangian subproblem:
		\begin{equation*}
	\min \,\,F_{(\mu,\rho)}(x,s,z):=c_\N^\top x-\mu\sum_{j=1}^{|\N|}\textup{ln}\,z_j+s^\top (z-x)+\frac{\rho}{2}\|z-x\|^2
	\quad\textup{s.t.}\, A_\N x=b,
		\end{equation*}
		where $s\in\mathbb{R}^{\cN}$ is the Lagrange multiplier for $z-x=0$ and $\rho>0$ is the penalty parameter. The function $F_{(\mu,\rho)}(x,s,z)$ is known as
the BAL function. 

For fixed $x$ and $s$, the first-order optimality condition with
respect to $z_j$ is $\frac{\partial F}{\partial z_j}=-\mu\frac{1}{z_j}+s_j+\rho(z_j-x_j)=0$. Since $z_j$ is required to be positive for all $1\leq j\leq\cN$, we get the following explicit form for $z_j$:
\begin{equation}\label{definez}
z_j(x_j,s_j;\mu,\rho)=\frac{1}{2}\big{(}\sqrt{(s_j/\rho- x_j)^2+4\mu/\rho}-(s_j/\rho-x_j)\big{)},\,1\leq j\leq \cN.
\end{equation}
We also define $y_j$ by
\begin{equation}\label{definey}
y_j(x_j,s_j;\mu,\rho)=\frac{1}{2}\big{(}\sqrt{(s_j/\rho- x_j)^2+4\mu/\rho}+(s_j/\rho-x_j)\big{)},\,1\leq j\leq \cN.
\end{equation}	
For notational simplicity, we write
$z(x,s;\mu,\rho)$ and $y(x,s;\mu,\rho)$ simply as $z$ and $y$,
respectively. The following properties can be found in \cite[Lemmas 2.1 and 2.2]{liu2022primal}.
\begin{lemma}\label{gradient}
For $\mu\ge0$ and $\rho>0$, let $z_j$ and $y_j$ be defined by \eqref{definez} and \eqref{definey}. Then, for all
$1\leq j\leq|\N|$, the following statements hold.
\begin{enumerate}[label=(\roman*)]
 \item If $\mu>0$, then $z_j$ and $y_j$ are differentiable with respect to both $x_j$ and $s_j$, and 
\begin{equation}\label{partialxs}
\begin{aligned}
\frac{\partial z_j}{\partial x_j}=-\rho\frac{\partial z_j}{\partial s_j}=\frac{z_j}{z_j+y_j},\quad \frac{\partial y_j}{\partial x_j}=-\rho\frac{\partial y_j}{\partial s_j}=-\frac{y_j}{z_j+y_j}.
\end{aligned}
\end{equation}
\item If $\mu>0$, then $z_j$ and $y_j$ are differentiable with respect to $\mu$, and
\begin{equation}\label{partialmu}
   \frac{\partial z_j}{\partial\mu}=
\frac{\partial y_j}{\partial \mu}=\frac{1}{\rho}\frac{1}{z_j+y_j}.
\end{equation}
\item If $\mu>0$, then $z_j$ and $y_j$ are differentiable with respect to $\rho$, and 
\begin{equation}\label{partial2rho}
   \frac{\partial (z_j-x_j)^2}{\partial\rho}=-\frac{2}{\rho}\frac{z_j}{z_j+y_j}(z_j-x_j)^2\le0;
\end{equation}

\item If $\mu=0$, then
\begin{equation*}
    z_j(x_j,s_j;0,\rho)=\max\{0,x_j-s_j/\rho\},\quad y_j(x_j,s_j;0,\rho)=\max\{0,s_j/\rho-x_j\}.
\end{equation*}
Moreover, for every $\mu\ge0$, 
\begin{equation}\label{z=x}
    z_j(x_j,s_j;\mu,\rho)-x_j=0\,\, \iff\,\, x_j\ge0,\,s_j\ge0,\,x_js_j=\mu.
\end{equation}
\end{enumerate}
\end{lemma}

Parts (i) and (iv) of Lemma \ref{gradient} tell us that, when $\mu>0$,
the residual $z(x,s;\mu,\rho)-x$ provides a smooth relaxation of the
nonnegativity and complementarity conditions. When $\mu=0$, the
equation $z-x=0$ recovers these conditions exactly.

Let $\omega=(x,\lambda,s)$ and $\Delta \omega=(\Delta x,\Delta\lambda,\Delta s)$. We define
\begin{equation}\label{definexi}
\xi_{(\mu,\rho)}(\omega):=z-x\quad\text{and}\quad \psi_{(\mu,\rho)}(\omega):=\begin{pmatrix}
c_\N-A_\N^\top\lambda-s\\	
A_\N x-b\\
\xi_{(\mu,\rho)}(\omega)
			\end{pmatrix}.
\end{equation}
Together with \eqref{z=x}, Theorem~2.5 of \cite{liu2022primal} reformulates the KKT conditions for the primal--dual pair \eqf{pp} and \eqf{dp} as
\begin{equation}\label{newton}
\Psi_{(\rho,\gamma)}(\mu,\omega):=\begin{pmatrix}
\mu+\gamma\phi_{(\mu,\rho)}(\omega)\\	
\psi_{(\mu,\rho)}(\omega)
			\end{pmatrix}=0,
\end{equation}
where $\mu\geq 0$, and $\gamma\in(0,1]$ is the balancing parameter. The function $\phi_{(\mu,\rho)}(\omega):=\frac{1}{2}\|\xi_{(\mu,\rho)}(\omega)\|^2$ serves as a merit function.

The feasible set $\F$ is defined as
\begin{equation}\label{defineF}
\begin{aligned}
    \F:=\{(x,\lambda,s)\;|\; A_\N x=b,\, A_\N ^\top\lambda+s=c_\N\}.
\end{aligned}
\end{equation}
The set $\F$ contains only the primal and dual equality
constraints, because nonnegativity and complementarity are represented
by the relaxed equation $z-x=0$.
Unlike standard IPMs, IPRM does not require the iterates $x^{(k)}$ and $ s^{(k)} $ to be positive. We choose an initial point $\omega^{(0)}\in\F$ using either \eqref{start1} or \eqref{start2}, as appropriate.

Starting from $\omega^{(0)}\in\F$, we apply Newton's method to the system~\eqref{newton}. As shown below, the Newton search
direction satisfies
$A_\N\Delta x=0$ and
$A_\N^\top\Delta\lambda+\Delta s=0$ at each iteration. Consequently, all subsequent iterates remain in $\F$. 
Using the derivative formulas in Lemma~\ref{gradient}, we obtain
the following linear system for the search direction. For brevity, we omit the iteration superscript $(k)$:
\begin{equation}\label{linearsystem}
\begin{aligned}
\Delta\mu
&=-\mu+\gamma\phi_{(\mu,\rho)}(\omega),\\
\begin{pmatrix}
0&A_\N^\top&I_{\cN}\\
-A_\N&0&0\\
(Z+Y)^{-1} Y&0&\frac{1}{\rho}(Z+Y)^{-1}Z
\end{pmatrix}
\begin{pmatrix}
\Delta x\\
\Delta\lambda\\
\Delta s
\end{pmatrix}
&=
\begin{pmatrix}
0\\
0\\
\xi+\frac{\Delta\mu}{\rho}(Z+Y)^{-1}\mathbf{1}_{|\N|}
\end{pmatrix}.
\end{aligned}
\end{equation}
Here $Y=\diag(y)$ and $Z=\diag(z)$. We let $H_{(\mu,\rho)}(\omega)$ denote the coefficient
matrix.

To obtain the Newton direction, define $\widetilde D=\mat_\N(z\circ z)\in\mathbb{R}^{m\times n}$
and let $D\in\mathbb{R}^{m\times(n-1)}$ consist of the first $n-1$ columns of $\widetilde D$. We then compute $U=\diag(\widetilde D\mathbf{1}_n)$, $\widetilde V=\diag(\widetilde D^\top\mathbf{1}_m)$ and $V=\diag(\widetilde V_{11},\cdots,\widetilde V_{n-1,n-1})$. 
As shown in \cite{liu2022primal}, the system~\eqref{linearsystem} can be reduced to the normal equation
\begin{equation}\label{coefficient2} 
A_\N Z^2A^\top_\N\Delta\lambda=\begin{pmatrix}
U& D\\			
D^\top& V\\
\end{pmatrix}\begin{pmatrix}
\Delta\lambda_1\\			
\Delta\lambda_2\\
\end{pmatrix}=\begin{pmatrix}
r_1\\			
r_2\\
\end{pmatrix},
\end{equation}
where $r=(r_1,r_2)=-A_\N(\mu I_{\cN}+\rho Z^2)\xi-\Delta\mu A_\N Z\mathbf{1}_{\cN},\,r_1\in\mathbb{R}^{m},\,r_2\in\mathbb{R}^{n-1}$, and $\Delta\lambda=(\Delta\lambda_1,\Delta\lambda_2),\,\Delta\lambda_1\in\mathbb{R}^{m},\,\Delta\lambda_2\in\mathbb{R}^{n-1}$. 
Since $\mu>0$, we have $z>0$. Because $\G_\N$ is connected, $U$, $V$ and $A_\N Z^2A^\top_\N$ are positive definite. Eliminating
$\Delta\lambda_2$ from \eqref{coefficient2} gives the Schur complement system
\begin{equation}\label{schur}
  (U-DV^{-1}D^\top)\Delta\lambda_1=r_1-DV^{-1}r_2,\quad \Delta\lambda_2=V^{-1}(r_2-D^\top\Delta\lambda_1).
\end{equation}
Once $\Delta\lambda$ has been computed, $\Delta s$ and $\Delta x$ can be
recovered explicitly; see Section~3 of \cite{liu2022primal}.
Further properties of the SCM in \eqref{schur} are established in Section~\ref{AnalysisofIPRM}.
The detailed IPRM for the OT subproblems is presented in
Algorithm~\ref{IPRM}. The well-definedness of the algorithm follows from
\cite[Lemmas 3.2 and 3.3]{liu2022primal}. In Algorithm~\ref{IPRM}, the parameters $\gamma_0$ and $\eta$ balance the values of
$\mu^{(k)}$ and $\phi^{(k)}$. If $\mu^{(k)}\leq\gamma_0\phi^{(k)}$, $\mu^{(k)}$ remains unchanged during the Newton step. If $\mu^{(k)}\ge\eta\phi^{(k)}$, Step~5 reduces $\mu^{(k)}$ to enforce $\mu^{(k)}<\eta\phi^{(k)}$. 
\begin{algorithm}[H]
\caption{A primal--dual interior-point relaxation method}
\label{IPRM}
\begin{algorithmic}[0]
\setlength{\itemsep}{0pt}
\setlength{\parsep}{0pt}
\setlength{\parskip}{0pt}

\Step{0}{Input $A_\N,\,b$ and $c_\N$. Choose $\omega^{(0)}\in\F$, $\mu^{(0)},\rho^{(0)}>0$, $\theta_0>0$, $\delta,\sigma\in(0,1)$, $\tau\in(0,0.5)$, $\varepsilon\in(0,\mu^{(0)})$, $\gamma_0\in(0,1/(1+\theta_0)^2)$ and $\eta\in(1,(1+\sqrt{{2\rho^{(0)}}\theta_0^2/{\cN}})^2)$. Compute $z^{(0)},y^{(0)}$ and $\phi_{(\mu^{(0)},\rho^{(0)})}(\omega^{(0)})$. Set $k=0$.}

\Step{1}{If $\max\{\mu^{(k)},\phi^{(k)}\}\leq \varepsilon$, terminate. Otherwise, set $\gamma=\min\{\gamma_0,\mu^{(k)}/\phi^{(k)}\}$.}

\Step{2}{Compute $\Delta\mu^{(k)}$ by \eqref{linearsystem}. Solve \eqref{schur} to obtain $\Delta\lambda^{(k)}$, then recover
$\Delta s^{(k)}$ and $\Delta x^{(k)}$.}

\Step{3}{Select the largest step size $\alpha^{(k)}\in\{1,\delta,\delta^2,\cdots\}$ such that
\begin{equation}\label{ineq}
\phi_{(\mu^{(k)}+\alpha^{(k)}\Delta\mu^{(k)},\rho^{(k)})}
(\omega^{(k)}+\alpha^{(k)}\Delta\omega^{(k)})
\leq (1-2\tau\alpha^{(k)})\phi^{(k)}.
\end{equation}

\,}

\Step{4}{Update $\mu^{(k+1)}=\mu^{(k)}+\alpha^{(k)}\Delta\mu^{(k)}$, $\omega^{(k+1)}=\omega^{(k)}+\alpha^{(k)}\Delta\omega^{(k)}$ and
$
\rho^{(k+1)}=
\max\big{\{}
\rho^{(k)},
\frac{\sigma\|s^{(k+1)}\|_{\infty}}
{\max\{\|x^{(k+1)}\|,1\}}
\big{\}}.
$
Compute $z^{(k+1)},y^{(k+1)}$ and $\phi^{(k+1)}$.}

\Step{5}{Set $l=0$, while $\mu^{(k+1)}/\eta^l\geq\max\{\epsilon,\eta\phi_{(\mu^{(k+1)}/\eta^l,\rho^{(k+1)})}
(\omega^{(k+1)})\}$, set $l=l+1$. Once this inner loop terminates, set
$\mu^{(k+1)}=\mu^{(k+1)}/\eta^l$ and recompute
$z^{(k+1)}$, $y^{(k+1)}$ and $\phi^{(k+1)}$.
Set $k=k+1$ and return to Step~1.
}
\end{algorithmic}
\end{algorithm}
\begin{lemma}\label{z1-z2}
Fix $\rho>0$ and $(x,s)$, denote $z(x,s;\mu,\rho)$ by $\hat z(\mu)$. Then for any $\hat\mu_1\ge\hat\mu_2\geq0$, one has 
\begin{equation}\label{zmu1mu2}
    \|\hat z(\hat\mu_1)-\hat z(\hat\mu_2)\|\leq\sqrt{\frac{|\N|}{\rho}}(\sqrt{\hat\mu_1}-\sqrt{\hat\mu_2}).
\end{equation}
\end{lemma}
\begin{proof}We have
\begin{equation*}
\begin{split}
      \|\hat z(\hat\mu_1)-\hat z(\hat\mu_2)\|^2&=\sum_{j=1}^{|\N|}\frac{1}{4}\left|\sqrt{({s_j}/{\rho}-x_j)^2+4{\hat\mu_1}/{\rho}}-\sqrt{({s_j}/{\rho}-x_j)^2+4{\hat\mu_2}/{\rho}}\right|^2\\
      &=\sum_{j=1}^{|\N|}\frac{1}{4}\left|\frac{4(\hat\mu_1-\hat\mu_2)/\rho}{\sqrt{({s_j}/{\rho}-x_j)^2+4{\hat\mu_1}/{\rho}}+\sqrt{({s_j}/{\rho}-x_j)^2+4{\hat\mu_2}/{\rho}}}\right|^2\\
      &\leq\sum_{j=1}^{\cN}\Big{(}\frac{(\hat\mu_1-\hat\mu_2)/\rho}{\sqrt{{\hat\mu_1}/{\rho}}+\sqrt{{\hat\mu_2}/{\rho}}}\Big{)}^2\\
      &= \frac{|
      \N|}{\rho}(\sqrt{\hat\mu_1}-\sqrt{\hat\mu_2})^2.
\end{split}
\end{equation*}
The desired inequality follows directly. \qed
\end{proof}

We next interpret the relaxation residual
$\xi_{(\mu,\rho)}(\omega)=z(x,s;\mu,\rho)-x$ in terms of
perturbations of the original primal--dual pair
\eqf{pp} and \eqf{dp}. Given $\rho>0$ and
$\xi\in\mathbb{R}^{|\N|}$, consider the following perturbed
primal--dual pair:
\begin{equation*}
\begin{tabular*}{\textwidth}{@{\extracolsep{\fill}} l r @{}}
$\displaystyle \,\,\,\,\min\,\, (c_\N+\rho\xi)^\top x
\quad \,\,{\rm s.t.}\,\,A_\N x = b+A_\N \xi,\, x\in \mathbb{R}^{|\N|}_+$,
&
$\hypertarget{perpp}{(\mathbf{P}_{(\rho,\xi)})}$
\\
$\displaystyle \,\,\,\,\max\,\,  (b+A_\N\xi)^\top \lambda
\quad{\rm s.t.}\,\,A_\N^\top\lambda+s =c_\N+\rho\xi,\,
{\lambda\in \mathbb{R}^{m+n-1},\,s\in\mathbb{R}^{|\N|}_+}$.
&
$(\hypertarget{perdp}{\mathbf{D}_{(\rho,\xi)})}$
\end{tabular*}
\end{equation*}

The following lemma shows that the iterates generated by IPRM can be
viewed as central-path points of a sequence of perturbed primal--dual pairs. This provides the basis for the sensitivity and global
convergence analyses developed below.

\begin{lemma}\label{central}
For $(x,\lambda,s)\in\F$, $\mu\geq0$, and $\rho>0$, define $z$ and
$y$ by \eqref{definez} and \eqref{definey}, respectively, and set
$\xi=z-x$ as defined in \eqref{definexi}. Then $(z,\lambda,\rho y)$ satisfies
\begin{equation}\label{cent}
    A_\N z=b+A_\N\xi,\,
      A_\N^\top\lambda+\rho y=c_\N+\rho\xi,\, z\ge0,\, \rho y\ge0,\, z\circ(\rho y)=\mu\bm 1_{\cN}.
\end{equation}
In particular, if $\mu>0$, then $z>0$ and $\rho y>0$, so that
$(z,\lambda,\rho y)$ lies on the central path associated with~\eqf{perpp} and~\eqf{perdp}. If
$\mu=0$, then $(z,\lambda,\rho y)$ is a KKT triple for~\eqf{perpp} and~\eqf{perdp}.
\end{lemma}
\begin{proof}
Since $(x,\lambda,s)\in\F$ and $\xi=z-x=y-s/\rho$, we obtain
$ A_\N z=b+A_\N\xi$ and $
      A_\N^\top\lambda+\rho y=c_\N+\rho\xi$. By the definitions of $z$ and $y$, we have $z\ge0,\, \rho y\ge0$ and $z\circ(\rho y)=\mu\bm 1_{\cN}$. Thus, all the relations in \eqref{cent} hold.\qed
\end{proof}

\subsection{Support Identification}\label{supportiden}
In this subsection, we incorporate IPRM into the multiscale framework to develop MSIPRM for OT problems. To address the main issues raised at the end of Section~\ref{updateN}, the multiscale framework is used to select the initial active set $\N^{[0]}$, whereas IPRM solves the resulting sparse subproblems~\eqref{sparse}. A remaining challenge is to reliably identify the optimal support from the approximate solutions generated by IPRM, since solving every subproblem exactly would be impractical. 

To justify support identification from inexact IPRM solutions, we
analyze the stability of the optimal support under perturbations. For this purpose, we consider the perturbed primal--dual pair \eqf{perpp} and \eqf{perdp}. We assume that each primal subproblem~\eqf{pp} admits a unique optimal solution satisfying the primal constraint nondegeneracy condition. In view of Theorem~\ref{zeromeasure} and \cite[Lemma~1]{ge2025interior}, this assumption is justified within our multiscale framework. Moreover, by \cite[Theorem~4.5]{sierksma2001linear}, this assumption is equivalent to the primal and dual constraint nondegeneracy conditions holding at a KKT triple for \eqf{pp} and \eqf{dp}.

Recall that $\bar A$ is totally unimodular; that is, the determinant
of every square submatrix of $\bar A$ lies in $\{-1,0,1\}$. Consequently, $\bar A_\N$ and $A_\N$ are also totally unimodular,
and so are all of their submatrices; see~\cite[Section~19.3]{schrijver1998theory}.
This structure, together with the constraint nondegeneracy assumptions, allows us to control how the perturbation affects the optimal support. In particular, Lemma~\ref{samespt} shows that, whenever the perturbation
vector $\xi$ is sufficiently small, the perturbed primal problem~\eqf{perpp} has an optimal solution
whose support coincides with that of the optimal solution of~\eqf{pp}. In contrast to the result of \cite[Theorem~3]{cartis2016active}, our conclusion does not require $\xi$ to be positive. The proof of Lemma \ref{samespt} is provided in Appendix~\ref{proof:lemma-e}.

\begin{lemma}\label{samespt}
Let $\rho>0$, and $(x^*,\lambda^*,s^*)$ be a KKT triple for the primal--dual pair
\eqf{pp} and \eqf{dp}. Suppose that the primal and dual constraint
nondegeneracy conditions hold at $(x^*,\lambda^*,s^*)$. If the perturbation vector $\xi$ satisfies
\begin{equation}\label{uppernormxi}
    \|\xi\|<\frac{\min\left\{\min_{j\in \spt(x^*)}x^*_j,\min_{j\in\spt(s^*)}s^*_j/\rho\right\}}{\sqrt{1+4(m+n-1)(|\N|-(m+n-1))}},
\end{equation}
then there exists a KKT triple $(\hat x,\hat\lambda,\hat s)$ for the perturbed primal--dual pair \eqf{perpp} and \eqf{perdp} such that $\spt(\hat x)=\spt(x^*)$ and $\spt(\hat s)=\spt(s^*)$.
\end{lemma}

\begin{corollary}\label{integer}
Assume that $\rho\ge 1$, $c_\N\ge 0$ and $\bar b>0$ are both integer vectors, and that $\cN=\chi_\N (m+n)$, where $\chi_\N\geq 2$. Suppose further that the primal and dual constraint nondegeneracy conditions hold at a KKT triple
$(x^*,\lambda^*,s^*)$ for~\eqf{pp} and~\eqf{dp}. Then the conclusion in Lemma \ref{samespt} holds if $\xi$ satisfies
 \[\|\xi\|<\frac{1}{2\rho\sqrt{\chi_\N-1}(m+n)}.\]
\end{corollary}
\begin{proof}
Since $\cN=\chi_\N (m+n)$ and $\chi_\N\geq 2$, we have
\begin{equation}\label{m+n-1}
    1+4(m+n-1)(|\N|-(m+n-1))\le4(\chi_\N-1)(m+n)^2.
\end{equation}
By the integrality of the data, primal and dual constraint nondegeneracy, and the total unimodularity of $A_\N$, the KKT triple
$(x^*,\lambda^*,s^*)$ is integral. Hence, 
\begin{equation}\label{minmin}
  \min_{j\in \spt(x^*)}x^*_j\ge 1\quad\textup{and}\quad\min_{j\in\spt(s^*)}s^*_j/\rho\geq {1}/{\rho}.
\end{equation}
Substituting \eqref{m+n-1} and \eqref{minmin} into \eqref{uppernormxi}
yields the desired conclusion.\qed
\end{proof}

Let $\{\mu^{(k)}\},\,\{\rho^{(k)}\},\,\{z^{(k)}\},\,\{y^{(k)}\}$, and $\{\omega^{(k)}\}$ be generated by Algorithm~\ref{IPRM}. Building on Lemma \ref{samespt}, we now derive a
support-identification criterion for these iterates. The proof of Theorem~\ref{uppermuk} is provided in Appendix~\ref{proof:theorem-d}.

\begin{theorem}\label{uppermuk}
Assume that $\{\rho^{(k)}\}$ is bounded and that there exists $k_1\in\mathbb{Z}_+$ such that $\gamma_0\phi^{(k)}<\mu^{(k)}<\eta\phi^{(k)}$ for all $k\ge k_1$. Suppose further that the primal and dual constraint nondegeneracy conditions hold at a KKT triple $(x^*,\lambda^*,s^*)$ for \eqf{pp} and~\eqf{dp}. If $k\ge k_1$ and $\mu^{(k)}$ satisfies
\begin{equation}\label{sqrtmu}
\sqrt{\mu^{(k)}}\leq\frac{\min\left\{\min_{j\in \spt(x^*)}x^*_j,\min_{j\in\spt(s^*)}s^*_j/\rho^{(k)}\right\}}{|\N|/\sqrt{\rho^{(k)}}+\sqrt{2(1+4(m+n-1)(|\N|-(m+n-1)))/\gamma_0}},
\end{equation}
then the following support identities hold:
\begin{equation}\label{partionp}
\begin{aligned}
\spt(x^*)&=\Big{\{}j\,\Big{|}\, z^{(k)}_j\ge \sqrt{\mu^{(k)}/\rho^{(k)}},\,1\leq j\leq |\N|\Big{\}},\\
\spt(s^*)&=\Big\{j\,\Big|\, z^{(k)}_j<\sqrt{\mu^{(k)}/\rho^{(k)}},\,1\leq j\leq |\N|\Big\}.
\end{aligned}
\end{equation}
\end{theorem}

\begin{remark}
The assumptions of Theorem~\ref{uppermuk} are supported by the
subsequent analysis. First, Theorem~\ref{infeasible} shows
that if $\rho^{(k)}\to\infty$, then the primal problem \eqf{pp} can be
infeasible. Second, Theorem~\ref{Hoffman} indicates that there exists $k_1\in\mathbb{Z}_+$ such that
$\gamma_0\phi^{(k)}<\mu^{(k)}<\eta\phi^{(k)}$ for all $k\ge k_1$.
Moreover, the global convergence result in Theorem~\ref{Hoffman} guarantees that $\mu^{(k)}\to 0$, ensuring that the condition in~\eqref{sqrtmu} is
satisfied for all sufficiently large $k$.   
\end{remark}

In practice, $m+n\gg1$, while $\chi_\N\geq2$ is typically of moderate size. Since $\{\rho^{(k)}\}$ is nondecreasing and bounded, there exists $\rho^*>0$ such that $\rho^{(k)}\to\rho^*$. If the assumptions of Corollary~\ref{integer} hold, then the upper bound in \eqref{sqrtmu} can be approximated by
\[\sqrt{\mu^{(k)}}\leq
\frac{1}{\chi_\N\sqrt{\rho^*}+\rho^*\sqrt{8(\chi_\N-1)/\gamma_0}}\frac{1}{m+n}.\]

Equation~\eqref{partionp} shows that $\spt(x^*)$ consists precisely of the indices associated with the relatively large components of
$z^{(k)}$. Under the primal and dual constraint nondegeneracy conditions, $|\spt(x^*)|=m+n-1$. 
These observations motivate constructing a candidate support for the
optimal transport plan by selecting the largest components of
$z^{(k)}$. Specifically, we choose a parameter $\chi>1$ and
define $q_\chi=\min\{\lceil\chi(m+n)\rceil,\cN\}$. Let $\pi^{(k)}$ be a permutation of $\{1,\cdots,|\N|\}$ such that
$z_{\pi^{(k)}(1)}^{(k)}
\ge z_{\pi^{(k)}(2)}^{(k)}
\ge\cdots\ge
z_{\pi^{(k)}(|\N|)}^{(k)}$. We define the relaxed support estimate by
\begin{equation}\label{relaxp}
\Omega^{(k)}(\chi)
:=\left\{(u_{i_{\pi^{(k)}(r)}},v_{j_{\pi^{(k)}(r)}})\in\N\,\Big{|}\,1\leq r\leq q_\chi\right\}.
\end{equation}
Thus, $\Omega^{(k)}(\chi)$ approximates the support of the optimal
transport plan at the current level and is used to construct the
initial active set at the next level. The complete MSIPRM is
summarized in Algorithm~\ref{MSIPRM}.

\begin{algorithm}[H]
\caption{A multiscale primal--dual interior-point relaxation method}
\label{MSIPRM}
\begin{algorithmic}[0]
\setlength{\itemsep}{0pt}
\setlength{\parsep}{0pt}
\setlength{\parskip}{0pt}

\Step{0}{
Construct two hierarchical partitions
$(\U_0,\cdots,\U_K)$ and
$(\V_0,\cdots,\V_K)$, and choose $\bar\chi>1$.
Compute the measures and cost functions by
\eqref{multimeasure} and \eqref{costk}, respectively.
}

\Step{1}{Solve the OT problem at level $K$ using IPRM with
$\N_K=\U_K\times\V_K$ as the active set,
and obtain $\mu_K$, $\rho_K$, $\omega_K$, and $z_K$.
Set $\ell= K-1$.
}

\Step{2}{
Choose $\chi\in(1,\bar{\chi})$ and let $q_\chi=\min\left\{\left\lceil\chi\bigl(|\U_{\ell+1}|+|\V_{\ell+1}|\bigr)\right\rceil,|\N_{\ell+1}|\right\}$. Using $z_{\ell+1}$, construct the relaxed support estimate
$\Omega_{\ell+1}(\chi)\subseteq\N_{\ell+1}$ according to~\eqref{relaxp}.
}

\Step{3}{
Construct the initial active set for level $\ell$ by
\[
\N^{[0]}_\ell\gets\bigcup_{(\bm{u}_{\ell+1,i_t},\bm{v}_{\ell+1,j_t})\in\Omega_{\ell+1}(\chi)}\child(\bm{u}_{\ell+1,i_t})\times\child(\bm{v}_{\ell+1,j_t}).
\]
}

\Step{4}{Apply Algorithm~\ref{update} at level $\ell$, using
$\N^{[0]}_\ell$ as the initial active set and IPRM as the LP solver.
Let $\mu_\ell$, $\rho_\ell$, $\omega_\ell$, and $z_\ell$
denote the returned quantities, and let $\N_\ell\subseteq\U_\ell\times\V_\ell$ denote the resulting active set.
}

\Step{5}{
If $\ell=0$, terminate. Otherwise, set $\ell=\ell-1$ and return
to Step~2.
}

\end{algorithmic}
\end{algorithm}

\subsection{Analysis of the Schur Complement System}\label{AnalysisofIPRM}
For large-scale OT problems, the matrix-vector multiplications $A_\N  x$ and $A_\N ^\top\lambda$, as well as the computation of the Newton direction, can be time-consuming.
The Schur complement formulation in \eqref{schur} reduces the cost of
computing the Newton direction by exploiting the structure of
$A_\N$. Since $|\N|$ is typically much smaller than $mn$, the matrix $D$ is highly sparse. 
The Schur complement $M:=U-DV^{-1}D^\top$ is often sparse in practice as well. Recall that $\mathcal{G}_{\mathcal{N}}$ is connected. Then
$A_{\mathcal{N}}Z^2A_{\mathcal{N}}^\top$, $U$, and $V$ are positive definite. It then follows from the Schur complement criterion that
$M$ is also positive definite.
Consequently, the symmetrically scaled matrix
$I_m-QQ^\top=U^{-\frac{1}{2}}MU^{-\frac{1}{2}}$ is also positive definite, where $Q:=U^{-\frac{1}{2}}DV^{-\frac{1}{2}}\in\mathbb{R}^{m\times(n-1)}$.
In practice, we solve the Schur complement system~\eqref{schur} either by applying
the PCG method with $U$ as a preconditioner or by directly computing a Cholesky factorization of $M$.

In numerical linear algebra, the condition number $\kappa(M)$ governs
the sensitivity of the computed Newton direction to numerical errors, whereas
 $\kappa(I_m-QQ^\top)$ influences the convergence rate of the PCG method. In this subsection, we provide 
upper bounds on both $\kappa(M)$ and $\kappa(I_m-QQ^\top)$. We further show that, if $\mu^{(k)}\to 0$ and $z^{(k)}$ converges to a solution $x^*$
satisfying the primal constraint nondegeneracy condition, then the condition numbers of these two matrices remain uniformly bounded as $k\to\infty$.
Consequently, the Schur complement
system~\eqref{schur} avoids the ill-conditioning that commonly arises in the Newton systems of IPMs; refer to \cite[Chapter 14]{nocedal2006numerical} for example.

\begin{proposition}
Let $\G_\N=(\U\cup\V,\N)$ be a connected bipartite graph with $2\leq m\leq n$, and let $\varrho_\N$ be the maximum degree of $\G_\N$. For the Schur complement system~\eqref{schur}, the following
statements hold.
\begin{enumerate}[label=(\roman*)]
    \item  For every $z\in\mathbb{R}^{|\N|}_{++}$,
\begin{equation}\label{kappa1}
       \kappa(M)\leq \frac{(m+n)^2}{2}\frac{U_{\max}}{U_{\min}}\frac{\max\{U_{\max},V_{\max}\}}{z^2_{\min}},
\end{equation}
and
\begin{equation}\label{kappa2}
    \kappa(I_m-QQ^\top)\leq \frac{(m+n)^2}{2}\frac{\max\{U_{\max},V_{\max}\}}{z^2_{\min}},
\end{equation}
where $z_{\min}=\min_{1\leq t\leq\cN}z_t$,
$U_{\min}=\min_{1\leq i\leq m} U_{ii}$, 
$U_{\max}=\max_{1\leq i\leq m} U_{ii}$, and $V_{\max}=\max_{1\leq j\leq n-1}V_{jj}$.

    \item Suppose that the sequence $\{z^{(k)}\}$ generated by Algorithm {\ref{IPRM}} converges to 
    a solution $x^*$ of \eqf{pp} satisfying the primal constraint nondegeneracy condition. For any constants $0<\eta_1<1<\eta_2$, there exists $K_\eta\in\mathbb{N}$ such that for all $k\ge K_\eta$, 
\begin{equation}\label{kappa3}
  \kappa(M^{(k)})\leq \frac{\varrho^2_\N (m+n)^2}{2}\frac{\eta^2_2}{\eta^2_1}\left(\frac{x^*_{\max}}{x^*_{\min}}\right)^4,
\end{equation}
and 
\begin{equation}\label{kappa4}
\kappa(I_m-Q^{(k)}(Q^{(k)})^\top)\leq \frac{\varrho_\N (m+n)^2}{2}\frac{\eta_2}{\eta_1}\left(\frac{x^*_{\max}}{x^*_{\min}}\right)^2,
\end{equation}
where $x^*_{\max}=\max_{j\in\spt(x^*)}x^*_j>0$ and $x^*_{\min}=\min_{j\in\spt(x^*)}x^*_j>0$.
\end{enumerate}
\end{proposition}
\begin{proof}
(i) Let $w=(p,q)\in\mathbb{R}^{m+(n-1)}$ satisfy $\|w\|=1$, and set $q_n:=0$. Then
\begin{equation*}   w^\top(A_\N Z^2A_\N ^\top)w=\sum_{(u_{i_t},v_{j_t})\in\N}z^2_t(p_{i_t}+q_{j_t})^2.
\end{equation*} 
Since $\G_\N$ is connected, there exists $\widehat{\N}\subseteq\N$ such that $|\widehat{\N}|=m+n-1$ and $A_{\widehat{\N}}$ is 
nonsingular. Then we have $\|A_{\widehat{\N}}^{-1}\|\leq\|A_{\widehat{\N}}^{-1}\|_F\leq m+n-1$ due to the total unimodularity of $A_{\widehat\N}$; see
\cite[Proposition 1]{lu2024pdot}. It follows that
\begin{equation}\label{lower}
\begin{aligned}  w^\top(A_\N Z^2A_\N ^\top)w\geq\sum_{(u_{i_t},v_{j_t})\in\widehat\N}z^2_t(p_{i_t}+q_{j_t})^2\geq \frac{\min\{z^2_t\mid(i_t,j_t)\in \widehat{\N}\}}{(m+n-1)^2}.\\
    \end{aligned}
\end{equation}
Therefore, $\lambda_{\min}(A_\N Z^2A_\N ^\top)\ge \frac{\min\{z^2_t\mid(u_{i_t},v_{j_t})\in \widehat{\N}\}}{(m+n-1)^2}$.
Consider the following matrix
\begin{equation*}
N=\begin{pmatrix}
U^{-\frac{1}{2}}& 0\\			
0& V^{-\frac{1}{2}}\\
\end{pmatrix}
\begin{pmatrix}
U& D\\			
D^\top&V\\
\end{pmatrix}\begin{pmatrix}
U^{-\frac{1}{2}}& 0\\			
0& V^{-\frac{1}{2}}\\
\end{pmatrix}=\begin{pmatrix}
I_m& Q\\			
Q^\top& I_{n-1}\\
\end{pmatrix}.
\end{equation*}
The matrix $N$ is positive definite and  $\kappa(N)=\frac{1+\sigma_{\max}(Q)}{1-\sigma_{\max}(Q)}$. Equivalently, $\sigma_{\max}(Q)=\frac{\kappa(N)-1}{\kappa(N)+1}<1$. Thus, $1-\sigma^2_{\max}(Q)=\frac{4}{(\sqrt{\kappa(N)}+1/\sqrt{\kappa(N)})^2}$. Moreover, $\lambda_{\max}(N)=1+\sigma_{\max}(Q)\leq 2$, and the diagonal scaling gives
$\lambda_{\min}(N)\geq \frac{\lambda_{\min}(A_\N Z^2A^\top_\N)}{\max\{U_{\max},V_{\max}\}}$. Then
\begin{equation}\label{kappaN}
    \kappa(N)\leq\frac{2\max\{U_{\max},V_{\max}\}}{\lambda_{\min}(A_\N Z^2A_\N^\top)}\le2(m+n-1)^2\frac{\max\{U_{\max},V_{\max}\}}{\min\{z^2_t\mid(i_t,j_t)\in \widehat{\N}\}}.
\end{equation}

Since
$M=U^{\frac{1}{2}}(I_m-QQ^{\top})U^{\frac{1}{2}}$, we have $\lambda_{\min}(M)\ge \lambda_{\min}(U)\lambda_{\min}(I_m-QQ^\top)= U_{\min}(1-\sigma_{\max}^2(Q))$ and $\lambda_{\max}(M)\leq\lambda_{\max}(U)= U_{\max}$. Therefore,
\begin{equation}\label{kappaM}
  \kappa(M)\leq\frac{U_{\max}}{U_{\min}(1-\sigma^2_{\max}(Q))}=\frac{U_{\max}}{U_{\min}}\frac{(\sqrt{\kappa(N)}+1/\sqrt{\kappa(N)})^{2}}{4},
\end{equation}
and similarly
\begin{equation}\label{kappaIQ}
      \kappa(I_m-QQ^\top)
      \leq \frac{1}{1-\sigma_{\max}^2(Q)}=\frac{(\sqrt{\kappa(N)}+1/\sqrt{\kappa(N)})^{2}}{4}.
\end{equation}

Because $\min\{z^2_t\mid(u_{i_t},v_{j_t})\in \widehat{\N}\}\ge z_{\min}^2$ and $\max\{U_{\max},V_{\max}\}\ge z_{\min}^2$,  inequality \eqref{kappaN}, together with $2\leq m\leq n$, gives
$$\frac{(\sqrt{\kappa(N)}+1/\sqrt{\kappa(N)})^{2}}{4}\leq \frac{\kappa(N)+3}{4}\leq (m+n)^2\frac{\max\{U_{\max},V_{\max}\}}{2z_{\min}^2} $$
Substituting the upper bound into \eqref{kappaM} and \eqref{kappaIQ} yields~\eqref{kappa1} and~\eqref{kappa2}, respectively.

(ii) Fix any constants $0<\eta_1<1<\eta_2$. Since
$z^{(k)}\to x^*$, there exists $K_\eta\in\mathbb{N}$ such that, for
all $k\geq K_\eta$,
$\sqrt{\eta_1}x^*_{\spt(x^*)}\leq z^{(k)}_{\spt(x^*)}\leq \sqrt{\eta_2}x^*_{\spt(x^*)}$ and $\max_{1\leq t\leq\cN}\{z^{(k)}_t\}\leq\sqrt{\eta_2}x^*_{\max}$. By the primal constraint nondegeneracy of $x^*$, there exists a subset $\widehat{\N}\subseteq \spt_\G(\mat_\N(x^*))$ such that $A_{\widehat\N}$ is nonsingular and $\min\{(z^{(k)}_t)^2\mid(u_{i_t},v_{j_t})\in \widehat{\N}\}\geq \eta_1(x_{\min}^*)^2$. Furthermore, because every vertex is incident to at most $\varrho_\N$ edges, $\max\{U^{(k)}_{\max},V^{(k)}_{\max}\}\leq \eta_2\varrho_\N(x^*_{\max})^2$. Every vertex in $\U$ is incident to at least one edge in $\widehat{\N}$. So we have $U^{(k)}_{\min}\ge\eta_1 
(x^*_{\min})^2$. Substituting these three estimates into
\eqref{kappa1} and \eqref{kappa2} yields
\eqref{kappa3} and \eqref{kappa4}.
\qed
\end{proof}


\section{Convergence Analysis}\label{Convergence analyses}
Based on Proposition~\ref{converseries} and the structure of
Algorithm~\ref{MSIPRM}, the convergence analysis of MSIPRM can be reduced to that of IPRM. Let $\{\mu^{(k)}\}$, $\{\rho^{(k)}\}$, $\{\phi^{(k)}\}$, $\{\omega^{(k)}\}$, $\{z^{(k)}\}$, and $\{y^{(k)}\}$ denote the sequences generated by
Algorithm~\ref{IPRM}. Lemma~\ref{boundx} records several basic properties of the relevant sequences generated
by IPRM, including their boundedness. The proof is deferred to Appendix~\ref{proof:lemma-c}.

\begin{lemma}\label{boundx}
The following statements hold.
\begin{enumerate}[label=(\roman*)]

\item The sequence $\{\mu^{(k)}\}$ is nonincreasing, and $\{\phi^{(k)}\}$ is bounded.

\item The sequences $\{x^{(k)}\}$ and $\{z^{(k)}\}$ are bounded. 
\end{enumerate}
\end{lemma}

Before establishing the global convergence of IPRM, we consider the case of an unbounded penalty sequence.
\begin{theorem}\label{infeasible}
Suppose that $\rho^{(k)}\to\infty$ and $\phi^{(k)}\to0$. Then there exist an index set $\mathcal{K}=\{k_j\}$ and a pair
$(\bar\lambda,\bar s)$ such that
\[\lim\limits_{j\to\infty}\lambda^{(k_j)}/{\rho^{(k_j)}}=\bar\lambda \quad \textup{and}\quad \lim\limits_{j\to\infty}s^{(k_j)}/\rho^{(k_j)}=\bar s.\] Furthermore, if $b^\top\bar\lambda>0$, then the primal problem \eqf{pp} is infeasible.
\end{theorem}
\begin{proof}
The first assertion follows from the boundedness of $\{x^{(k)}\}$ and the update rule for $\{\rho^{(k)}\}$. The second assertion can be derived by an argument similar to that in \cite[Theorem 4.1]{zhang2023iprqp}.\qed
\end{proof}

\begin{remark}\label{theorem3}
Theorem~\ref{infeasible} shows that an unbounded penalty sequence can yield a certificate of primal infeasibility under the stated
conditions. In the remainder of this section, we focus on the case in which $\{\rho^{(k)}\}$ is bounded. Since $\{\rho^{(k)}\}$ is
nondecreasing, it converges to some
$\rho^*\geq\rho^{(0)}>0$.
\end{remark}

To derive a global error bound for IPRM, we represent the
solution set of the primal--dual pair~\eqf{pp} and \eqf{dp} as
follows:
\begin{equation*}
\begin{aligned}
&\quad\mathcal{S}(A_\N,b,c_\N)\\
&:=
\left\{\begin{pmatrix}
x\\			
				\lambda
			\end{pmatrix}\,\Bigg{|}\, \begin{pmatrix}
-I_{\cN}& 0\\			
				0&A_\N^\top
			\end{pmatrix}\begin{pmatrix}
x\\			
				\lambda
			\end{pmatrix}\leq\begin{pmatrix}
0\\		
c_\N
			\end{pmatrix},\,\begin{pmatrix}
A_\N& 0\\			
				c_\N^\top&-b^\top
			\end{pmatrix}\begin{pmatrix}
x\\			
				\lambda
			\end{pmatrix}=\begin{pmatrix}
b\\		
0
			\end{pmatrix}\right\}.
            \end{aligned}
\end{equation*}
As shown in Lemma~\ref{central}, the iterates generated by IPRM are
associated with a sequence of perturbed LP problems. By bounding the
distance from the iterates to the solution set $\mathcal{S}(A_\N,b,c_\N)$, we establish global convergence and a
global error bound for IPRM. The proof of Theorem~\ref{Hoffman} is deferred to Appendix~\ref{proof:theorem-6}.

\begin{theorem}\label{Hoffman}
Suppose that $\{\rho^{(k)}\}$ is bounded. Then the following statements hold. 
\begin{enumerate}[label=(\roman*)]
  \item If the inner loop in Step~5 of Algorithm~\ref{IPRM} does not
terminate at some iteration $k_0\geq1$, then $\omega^{(k_0)}$ is a KKT triple for the primal--dual pair 
  \eqf{pp} and \eqf{dp}.
\item If the inner loop in Step~5 always terminates finitely, then $\lim_{k\to\infty}\mu^{(k)}=0$ and $\lim_{k\to\infty}\phi^{(k)}=0$. Moreover, every cluster point of $\{\omega^{(k)}\}$ is a KKT triple for the
primal--dual pair \eqf{pp} and \eqf{dp}. Furthermore, 
 there exist an integer $k_1\in\mathbb{Z}_+$ and a constant $\widetilde C>0$ such that, for all $k\ge k_1$, $\gamma_0\phi^{(k)}<\mu^{(k)}<\eta\phi^{(k)}$ and
\begin{equation}\label{Errobound}
    \dist\left((x^{(k)},\lambda^{(k)}), \mathcal{S}(A_\N,b,c_\N)\right)\leq \widetilde C\sqrt{\mu^{(k)}}.
\end{equation}
\end{enumerate}
\end{theorem}

Building on classical local convergence results for Newton-type methods \cite{qi1993nonsmooth,chan2008constraint}, we next establish that IPRM achieves local quadratic convergence under suitable regularity conditions. Motivated by Theorem~\ref{zeromeasure} and the construction of the cost functions in \eqf{costk}, we 
assume that the primal and dual constraint nondegeneracy conditions hold at the KKT triple $\omega^*=(x^*,\lambda^*,s^*)$. Under these conditions, $\omega^*$ is strictly complementary. 
Consequently, the mapping
$\xi(x,s;\mu,\rho)$ is smooth in a
neighborhood of $(x^*,s^*;0,\rho^*)$ and is therefore
strongly semismooth in this neighborhood.

\begin{lemma}\label{nonsingular}
Suppose that $\mu^{(k)}\to0$, $\rho^{(k)}\to\rho^*$, and $\omega^{(k)}\to\omega^*=(x^*,\lambda^*,s^*)$. Suppose further that the primal and dual constraint nondegeneracy conditions hold at $\omega^*$. Then the matrix $H^{(k)}=H_{(\mu^{(k)},\rho^{(k)})}(\omega^{(k)})$ appearing in~\eqref{linearsystem} converges to a nonsingular matrix:
\begin{equation*}
    H^*=\begin{pmatrix}
0& A_\N^\top  &I_{\cN}\\			
				-A_\N&0&0\\  (Z^{*}+Y^*)^{-1}Y^{*}&0&\frac{1}{\rho^{*}}(Z^{*}+Y^{*})^{-1}Z^{*}
			\end{pmatrix},
\end{equation*}
where $Y^*=\diag(s^*/\rho^*)$ and $Z^*=\diag(x^*)$.
\end{lemma}
\begin{proof}
Since $y^{(k)}\to s^*/\rho^*$ and $z^{(k)}\to x^*$, we have
$Y^{(k)}\to Y^*$ and $Z^{(k)}\to Z^*$. Moreover, strict
complementarity and $\rho^*>0$ ensure that $Z^*+Y^*$ is positive definite.
By the continuity of matrix inversion on the set of
nonsingular matrices, we have $H^{(k)}\to H^*$. It suffices to prove that the homogeneous system
\begin{equation*}
        A_\N^\top d_\lambda+d_s=0,\quad A_\N d_x=0,\quad Y^*d_x+\frac{1}{\rho^*}Z^*d_s=0
\end{equation*}
has only the zero solution. The first equality gives that
$d_s=-A_\N^\top d_\lambda$. By the third equality and strict
complementarity, $(d_x)_j=0$ whenever $s_j^*>0$, whereas
$(d_s)_j=0$ whenever $x_j^*>0$. It follows that $Y^*d_x=0$, $Z^*d_s=0$, and $Z^*A_\N^\top d_\lambda=0$. The primal constraint nondegeneracy condition implies that
$Z^*A_\N^\top$ has full column rank. Hence,
$d_\lambda=0$, and the first equality yields $d_s=0$. Let $S^*=\mat_\N(s^*)$. The dual constraint nondegeneracy condition implies that $A_{\N\setminus\spt_\G(S^*)}$ has full column rank $\cN-|\spt(s^*)|$.
 Since $A_\N d_x=0$ and $(d_x)_j=0$ when $s_j^*>0$, it follows
that $d_x=0$. Thus, $H^*$ is nonsingular.
\qed
\end{proof}

We are now ready to state the local quadratic convergence result for
Algorithm~\ref{IPRM}. The proof is deferred to Appendix~\ref{proof:theorem-g}.
\begin{theorem}\label{quadratic}
Let $(0,\omega^*)$ be a cluster point of the sequence $\{(\mu^{(k)},\omega^{(k)})\}$ generated by Algorithm \ref{IPRM}.
If the primal and dual constraint nondegeneracy conditions hold at $\omega^*$, then the entire sequence $\{(\mu^{(k)}, \omega^{(k)})\}$ converges to $(0,\omega^*)$ with local quadratic convergence.
\end{theorem}

Although the primal and dual constraint nondegeneracy assumptions are relatively strong, the construction of the multiscale framework provides a natural setting in which they can be expected to hold. In our numerical experiments, we indeed observe local quadratic convergence.

\section{Numerical Results} \label{numerical results}
In this section, we compare MSIPRM with several state-of-the-art methods on the DOTmark dataset\footnote{https://www.stochastik.math.uni-goettingen.de/index.php?id=215/}. It consists of 10 image classes, each comprising 10 images
available at five resolutions ranging from $32\times 32$ to
$512\times 512$. To assess the performance of the methods on larger
problems, we additionally construct instances at resolutions
$1024\times 1024$ and $2048\times 2048$.
MSIPRM is implemented in MATLAB R2024b, and all experiments are performed on a machine equipped with an Intel Core i7-14700 processor and 64 GB of RAM.
    
  \subsection{Experimental Setup}
In our experiments, we solve OT problems between pairs of images of resolution $\texttt{row}\times\texttt{col}$, where $\texttt{row}=\texttt{col}$ ranges from $32$ to $2048$. The cost function is given by the squared Euclidean distance
between pixel locations. Specifically, for pixel locations $u_i=(k_i,l_i)$ and $v_j=(p_j,q_j)$, we define
 \[C_{ij}=h(u_i,v_j)=\|(k_i,l_i)-(p_j,q_j)\|^2=(k_i-p_j)^2+(l_i-q_j)^2.\]
For each image pair, we normalize the pixel intensities so that each image has unit total mass, thereby obtaining two discrete measures $\bm\mu,\bm\nu\in\mathbb{R}_+^{\texttt{row}\cdot\texttt{col}}$. 
Table~\ref{basicdata} reports the dimensions of representative
instances and illustrates the large scale of the resulting OT
problems.

To evaluate the computational performance of MSIPRM, we compare it with the following four methods.

\begin{itemize}
    \item \textbf{HOT} is a recently proposed fast method for discrete OT problems based on an equivalent reduced formulation \cite{zhang2025hot}. We directly use the Python code provided by the authors.

    \item \textbf{MCPLEX-NS} is a multiscale method that employs the network simplex solver in CPLEX (version 12.10.0.0) to solve each subproblem. It uses the same multiscale framework as MSIPRM and is called from MATLAB.

    \item \textbf{Sp-Sinkhorn}, proposed by Schmitzer~\cite{schmitzer2019stabilized}, is a sparse algorithm combining an $\epsilon$-scaling with a multiscale strategy for entropically regularized OT. We use the publicly available C code with a Python interface\footnote{{https://bernhard-schmitzer.github.io/MultiScaleOT/build/html/index.html}}.

    \item \textbf{LEMON-NS} is the network simplex solver provided by the C++ network optimization library LEMON (version 1.3.1)\footnote{{https://lemon.cs.elte.hu/}}. It solves the discrete OT problem through the equivalent minimum-cost flow formulation.
\end{itemize}

For the outer iterations of MSIPRM, we use a stopping criterion
based on the KKT residual. At each level, an approximate solution to the corresponding OT problem is accepted when
\[\textup{KKT}_{\textup{res,out}}:=\max\left\{\frac{\|\bar Ax-\bar b\|}{1+\|\bar b\|},\,\frac{\|(c-\bar A^\top\lambda)_-\|}{1+\|c\|},\,\frac{|c^\top x-\bar b^\top\lambda |}{1+|c^\top x|+|\bar b^\top \lambda|}\right\}\leq\texttt{tol}_{\textup{out}},\]
where the tolerance $\texttt{tol}_{\textup{out}}$ is set to 1E{-6} unless specified otherwise. 

For each inner subproblem, we terminate IPRM when the corresponding inner KKT residual, barrier parameter, and merit function satisfy
\[\max\{\textup{KKT}_{\textup{res,in}},\,\mu_{\textup{in}},\,\phi_{\textup{in}}\}\leq\texttt{tol}_{\textup{in}},\]
where $\texttt{tol}_{\textup{in}}$ is increased from 1E-8 to 1E-6, depending on the current resolution level. MCPLEX-NS shares an identical termination criterion with MSIPRM for the outer iterations. The stopping tolerance is set to 1E-6 for the other solvers except LEMON-NS, which is an exact algorithm for OT problems.

\subsection{Impact of the Multiscale Strategy}
The numerical results reported in \cite{liu2022multiscale} indicate that running times can vary substantially across image classes, even at the same resolution. The slowest instances may take several times as long as the fastest ones.
To reduce the influence of unusually sparse marginals when assessing the effect of image resolution, we exclude the MicroscopyImages and Shapes classes. Images in these
classes contain many zero-valued pixels, which lead to smaller problem sizes after zero-mass support points are removed.

We first compare IPRM directly with MSIPRM to isolate the effect of
the multiscale strategy. For each resolution, we randomly select five image pairs from each of the remaining DOTmark classes and run both methods on the same instances.
We report the average running time, total iteration count, KKT residual, and memory occupied by the transport variables. The results are presented in Table~\ref{withwithout}.

\begin{table}[H]
\centering
\captionsetup{font=small}
\caption{Comparison between IPRM and MSIPRM. OoM means out of memory}
\resizebox{0.9\textwidth}{!}{ 
    \footnotesize 
    \begin{tabular}{c |c c c c | c c c c}
        \toprule 
      \multirow{2}{*}{Resolution} & \multicolumn{4}{c|}{IPRM} & \multicolumn{4}{c}{MSIPRM} \\
        \cline{2-5} \cline{6-9}
                      & time (s) &  Iter. & $\textup{KKT}_{\textup{res, out}}$& Memory & time (s)& Iter. & $\textup{KKT}_{\textup{res, out}}$ &Memory\\
        \hline 
        $32\times 32$ &   8.53  &  67  &  4.3E-08 & 8 MB & 0.62   &145  &  3.3E-08 & 0.1 MB\\
        $64\times64$  &   90.74 &  80  &  2.3E-07 &128 MB & 3.11   &250  &  3.7E-07 & 0.5 MB\\
        $128\times128$& 2790.78 &  140 &  3.6E-08 & 2 GB  & 18.14  &324  &  1.8E-07 & 2.1 MB\\
        $256\times256$&        &  OoM &          & 32 GB & 53.26  &403  &  6.8E-07 & 8.5 MB\\
        \bottomrule 
    \end{tabular}   
}
\label{withwithout}
\end{table}

Table~\ref{withwithout} shows that the multiscale strategy substantially reduces both the running time and the storage required for the transport variables. Although MSIPRM solves a sequence of sparse subproblems and therefore requires more inner iterations in total, each subproblem is much smaller than the original full problem. As a result, the additional iterations do not lead to a higher computational cost. More importantly, the multiscale framework makes the computation feasible at scales where the full formulation is already impractical. In particular, IPRM applied directly to the full problem runs out of
memory at a resolution of $256\times 256$.
At the same resolution, MSIPRM solves the tested instances in
53.26 seconds on average and requires only 8.5~MB to store the
transport variables.

\subsection{Comparison with Other Solvers on the DOTmark Dataset}
We next compare MSIPRM with the four competing solvers introduced
above on the DOTmark dataset. To summarize running times over multiple instances, we report the shifted geometric mean (SGM), defined by
\[
\operatorname{SGM}
=
\left(\prod_{i=1}^N (t_i+\texttt{sh})\right)^{1/N}
-\texttt{sh}.
\]
Here, $N$ is the number of test instances at a given resolution, and $t_i$ denotes the running time for the $i$-th instance. 
We set $\texttt{sh}=0.1$ for resolutions $32\times 32$ and
$64\times 64$, $\texttt{sh}=1$ for $128\times 128$ and
$256\times 256$, and $\texttt{sh}=10$ for larger resolutions. For resolutions of $32\times 32$ to $512\times 512$, the time limits are set to 10, 100, 1000, 4000 and 40000 seconds, respectively. An instance is regarded as unsolved if a solver fails to meet the
prescribed accuracy requirement, exceeds the time limit, or
terminates with a numerical error. For each such instance, $t_i$ is
set to the corresponding time limit. At each resolution, the SGM of each solver is divided by the
smallest SGM among all solvers, so that the best normalized SGM is $1.00$.

These solvers provide complementary baselines for three reasons. First, Sp-Sinkhorn and MCPLEX-NS also use a multiscale framework, so they show whether IPRM is well suited to this framework. Second, HOT is a recent method based on a different reduced formulation of discrete OT and provides an alternative modern baseline. Third, LEMON-NS is a classical network simplex solver for the equivalent minimum-cost flow formulation and serves as a standard reference.
\begin{table}
\captionsetup{font=small}
\caption{Computational results for image resolutions ranging from $32\times32$ to $512\times512$. Each SGM entry is reported as normalized SGM/SGM time. Pinf and Gap\protect\footnotemark{} denote the average relative primal infeasibility and the average relative objective gap, respectively.
``$-$'' indicates that Pinf is numerically negligible, and OoM means out of memory}  
\resizebox{\textwidth}{!}{ 
    \footnotesize 
    \centering 
    \begin{tabular}{c c c c c c c} 
        \toprule 
        Resolution &  & \textbf{MSIPRM} & HOT & MCPLEX-NS & Sp-Sinkhorn & LEMON-NS \\ 
        \midrule 
\multirow{3}{*}{$32\times32$} 
        & SGM & {1.34/0.55s} & 6.54/2.68s & 3.27/1.34s & 1.68/0.69s & \textbf{1.00/0.41s} \\ 
        & Pinf\tnote{1} & 1.6E-08 & 1.0E-06 & $-$ & 3.5E-08 & $-$ \\ 
        & Gap & 3.7E-08 & 2.4E-05 & 9.7E-13 & 2.7E-03 & 1.5E-13\\ 
        \midrule 
       \multirow{3}{*}{$64\times64$} 
        & SGM & \textbf{1.00/3.05s} & 2.52/7.70s & 1.42/4.33s & 3.95/12.05s & 6.13/18.69s \\ 
         & Pinf &  2.9E-08 & 2.5E-06 & $-$  & 2.2E-08 & $-$\\ 
        & Gap & 6.6E-07 & 7.0E-05 & 4.3E-06 & 1.0E-03 & 1.0E-13\\ 
        \midrule 
        \multirow{3}{*}{$128\times128$} 
        & SGM & \textbf{1.00/10.54s} & 12.67/133.52s & 1.94/20.49s & 17.91/188.67s & {88.62/934.02s} \\ 
          & Pinf &  1.3E-06  & 4.1E-06 & $-$ & 1.1E-08 & $-$ \\ 
        & Gap &  1.7E-05 & 1.2E-04 &  1.2E-04 & 2.8E-04 & 1.8E-12  \\ 
            \midrule 
       \multirow{4}{*}{$256\times256$} 
        & Solved & {100\%} & 100\%   &  100\%   & 74\%  & \multirow{2}{*}{OoM} \\ 
        & SGM    & \textbf{1.00/51.79s}  & 32.32/1673.69s    &  3.39/175.67s       &   54.73/2834.59s&                      \\ 
        & Pinf   & 7.2E-06 & 9.4E-06 &   $-$    &  2.4E-06  & \multirow{2}{*}{OoM} \\ 
        & Gap    & 8.2E-05 & 2.7E-04 & 2.2E-04  &  1.1E-03  &                      \\ 
          \midrule
\multirow{4}{*}{$512\times512$}
& Solved&  100\% & 100\%&  100\% &  70\%& \multirow{2}{*}{OoM} \\
&  SGM  &  \textbf{1.00/107.17s}  & 145.08/15547.68s &   14.70/1575.76s   &  272.58/29211.89s   &  \\
&  Pinf & 1.6E-05 &  6.7E-06 &  $-$  &1.2E-05 &  \multirow{2}{*}{OoM}\\
&  Gap  & 1.6E-04 &  1.1E-04 &5.7E-04&8.8E-04&  \\
        \bottomrule 
    \end{tabular}   } 
    \label{table32-512}
\end{table}
\footnotetext{Since the stopping criteria of these solvers are different, we follow the comparison methodology in \cite{zhang2025hot}
and use
$\operatorname{Pinf}
:=\max\left\{
\frac{\|\min(x,0)\|}{1+\|x\|},
\frac{\|\bar A x-\bar b\|}{1+\|\bar b\|}
\right\}
$ and $
\operatorname{Gap}:=\frac{|c^\top x-c^\top x_r|}{1+|c^\top x_r|},
$
to evaluate the quality of the solutions, 
where $x_r$ is a high-accuracy reference solution computed by MCPLEX-NS with $\texttt{tol}_{\textup{out}}$=1E-10.}

\begin{figure}
    \centering 
    \captionsetup{font=small}
    \begin{subfigure}[t]{0.48\textwidth} 
        \centering 
        \includegraphics[width=\linewidth, keepaspectratio]{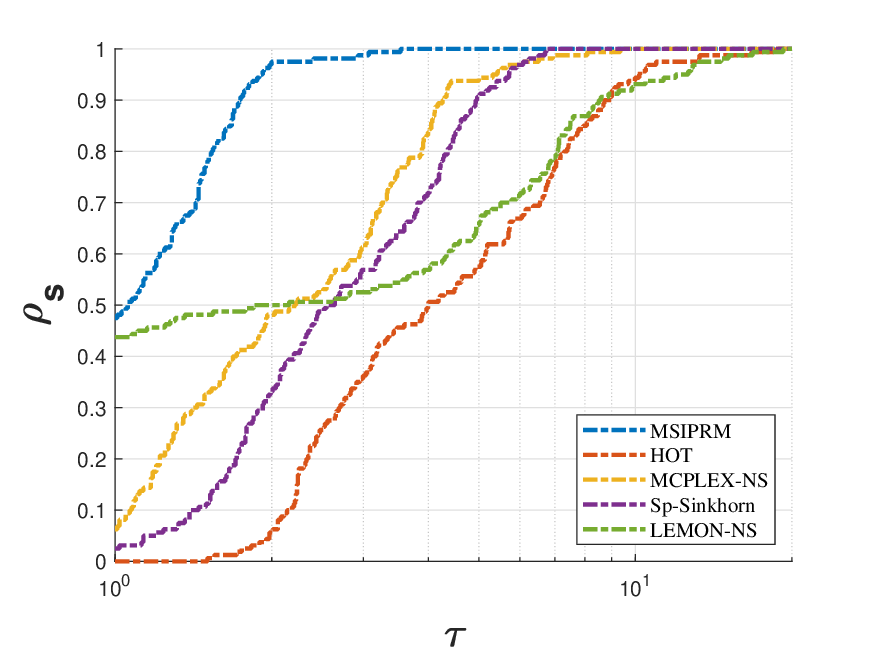} 
        \caption{\,Resolutions of $32\times32$ and $64\times64$} 
        \label{subfig:3264} 
    \end{subfigure} 
    \hfill 
    \begin{subfigure}[t]{0.48\textwidth}  
        \centering 
        \includegraphics[width=\linewidth, keepaspectratio]{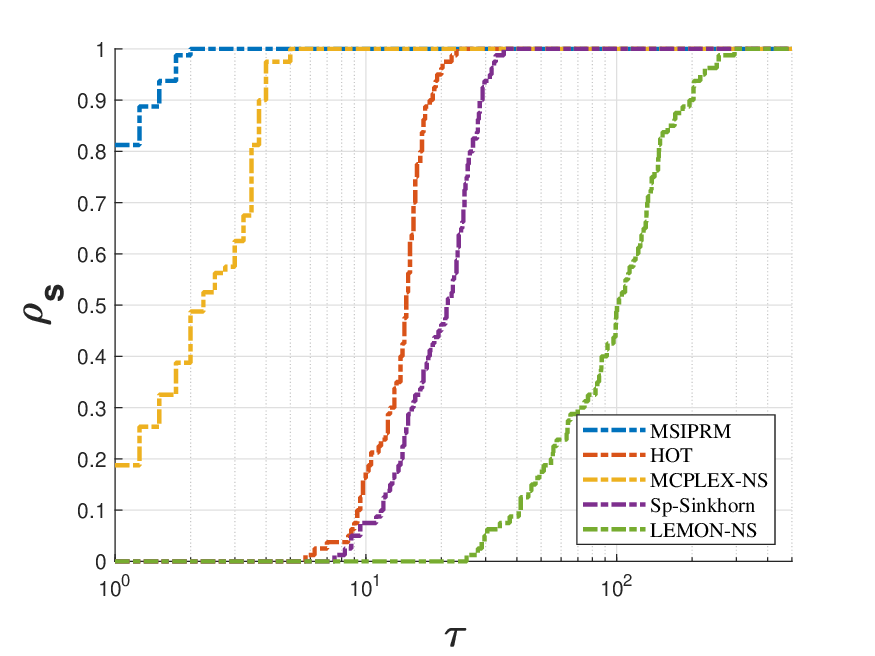} 
        \caption{\,Resolution of $128\times 128$} 
        \label{subfig:128} 
    \end{subfigure} 
    \vspace{5pt}  
    \begin{subfigure}[t]{0.48\textwidth} 
        \centering 
        \includegraphics[width=\linewidth, keepaspectratio]{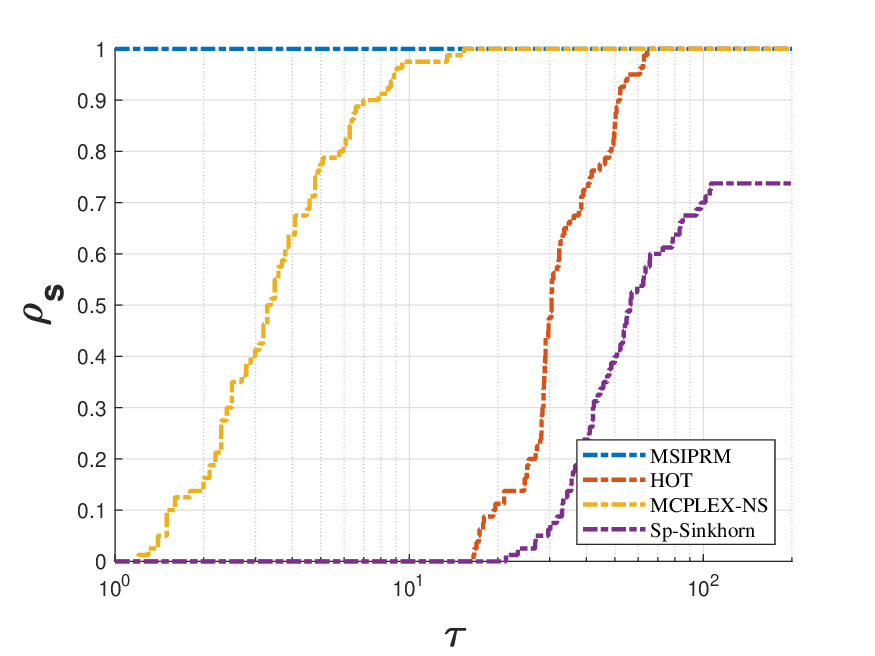} 
        \caption{\,Resolution of $256\times 256$} 
        \label{subfig:256} 
    \end{subfigure} 
    \hfill 
    \begin{subfigure}[t]{0.48\textwidth}  
        \centering 
        \includegraphics[width=\linewidth, keepaspectratio]{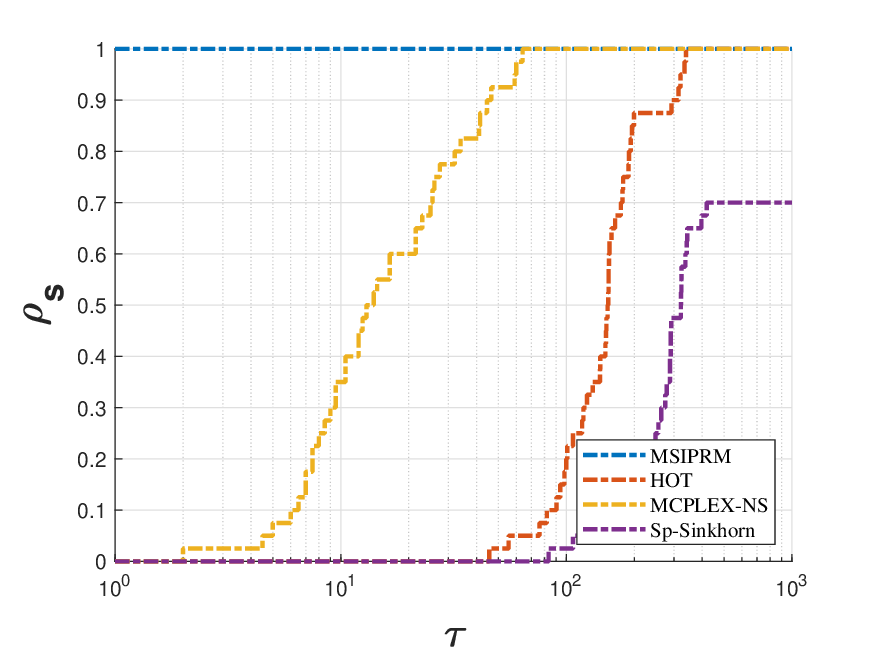} 
        \caption{\,Resolution of $512\times 512$} 
        \label{subfig:512} 
    \end{subfigure} 
    \caption{Running-time performance profiles for OT problems at different image resolutions
    }  
    \label{fig:32to512} 
\end{figure} 

Table~\ref{table32-512} reports the detailed numerical results, while Figure~\ref{fig:32to512} gives the corresponding performance profiles. Because all solvers successfully solved all instances at resolutions from $32\times32$ to $128\times128$, we omit the solved percentages from Table~\ref{table32-512}. At $32\times 32$, LEMON-NS is the fastest solver, with an SGM time of 0.41 seconds, while MSIPRM is already competitive at 0.55 seconds. This is not surprising, since for very small instances, the benefits of the multiscale strategy do
not yet offset its overhead, whereas a highly optimized network
simplex implementation remains very efficient. Starting at $64\times 64$, however, MSIPRM becomes the fastest method, and its advantage increases steadily with the problem size.
At $256\times256$ and $512\times512$, MSIPRM solves all instances and has substantially smaller SGM times than the competing methods. HOT and MCPLEX-NS also solve all instances but require considerably more time. Sp-Sinkhorn solves only $74\%$ and $70\%$ of the instances, respectively, while LEMON-NS runs out of memory.
In particular, LEMON-NS is very efficient only at the smallest resolution, but its memory requirements grow rapidly afterward. HOT remains robust in terms of success rate, but its running time becomes much larger than that of MSIPRM. Among the methods employing multiscale strategies, MSIPRM is consistently better than both MCPLEX-NS and Sp-Sinkhorn once the problem is no longer very small. These results suggest that IPRM is especially suitable as the subproblem solver in the proposed multiscale framework.

Figure~\ref{fig:32to512} illustrates this trend more clearly. The curve closer to the upper-left corner indicates better performance. At the lower resolutions, MSIPRM is already among the best methods, although LEMON-NS still has an advantage at $32\times 32$. At $128\times 128$, the MSIPRM curve is clearly more favorable than the others. This pattern becomes even stronger for $256\times 256$ and $512\times 512$ instances. The MSIPRM profile rises faster and remains above the other profiles. This indicates that MSIPRM solves a larger portion of the test problems within a smaller factor of the best running time. These profiles reinforce the conclusions drawn from Table~\ref{table32-512}.

\subsection{Numerical Performance on Grids of $1024\times1024$ and $2048\times2048$}
The DOTmark dataset provides instances only up to a resolution of
$512\times 512$ and therefore does not fully demonstrate the
large-scale capabilities of the proposed method. We perform additional tests on $1024\times 1024$ and $2048\times 2048$ grids. The full OT problems
contain approximately 1.1E+12 and 1.8E+13 variables, respectively.

We first construct five test measures by discretizing classical
two-dimensional probability distributions on a
$2048\times 2048$ grid. The test set includes unimodal, bimodal, and mixed-type distributions, producing a variety of transport patterns. 
As illustrated in Figure~\ref{fig:test-measures}, these measures
exhibit distinct geometric structures. This diversity makes the test set suitable for evaluating the robustness of MSIPRM across different measure pairs. We then downsample these measures to a $1024\times 1024$ grid and compute the K--W distance for each of the ten distinct measure pairs. The preceding DOTmark experiments indicate MCPLEX-NS as the strongest
multiscale competitor on the DOTmark dataset. We therefore focus the following comparison on MCPLEX-NS. The results are reported in Table~\ref{tab:KW-1024}.
\begin{figure}
    \centering 
    \captionsetup{font=small}
    \begin{subfigure}[t]{0.192\textwidth} 
        \centering 
        \includegraphics[width=\linewidth, keepaspectratio]{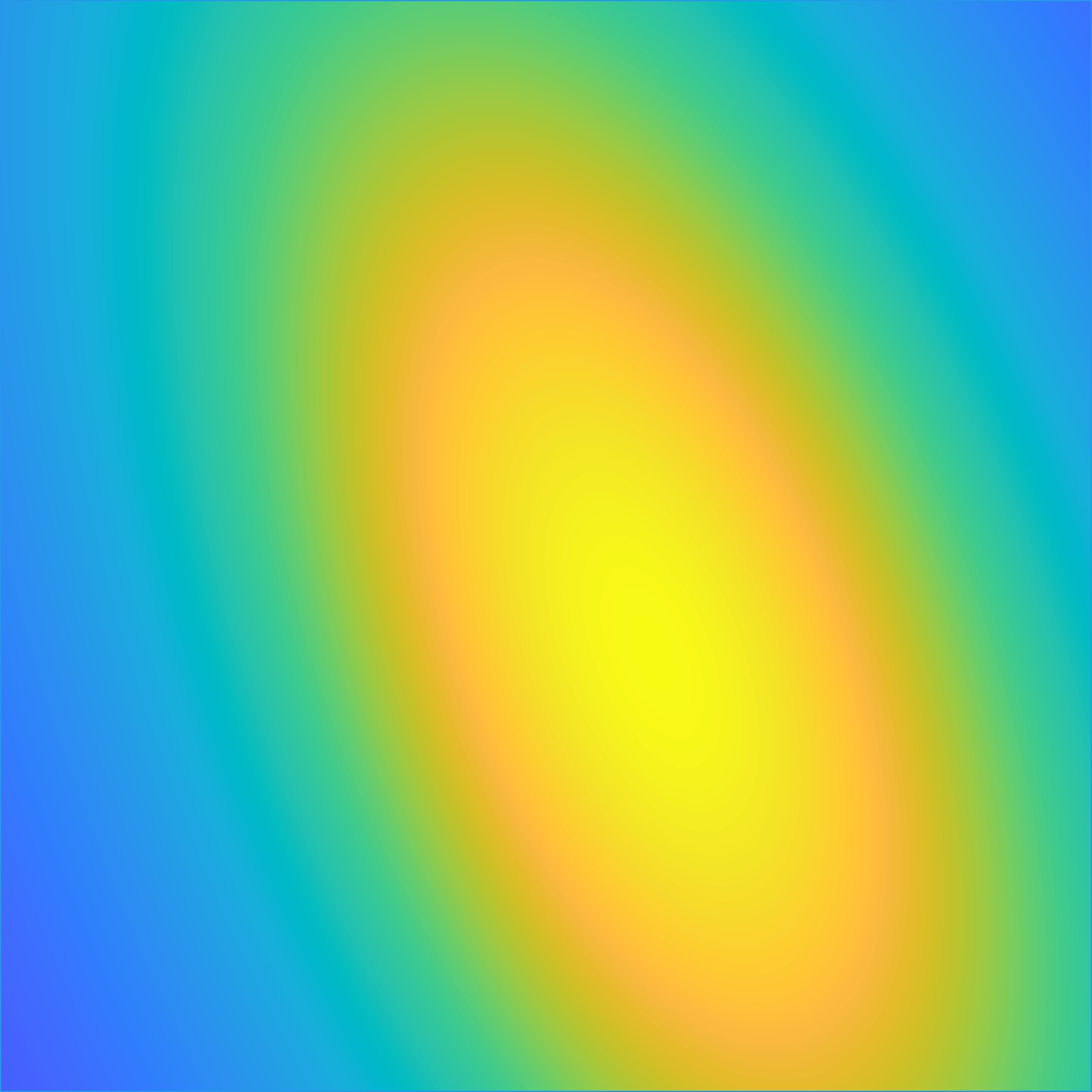} 
        \caption{ } 
        \label{subfig:cauchy} 
    \end{subfigure} 
    \hfill 
    \begin{subfigure}[t]{0.192\textwidth}  
        \centering 
        \includegraphics[width=\linewidth, keepaspectratio]{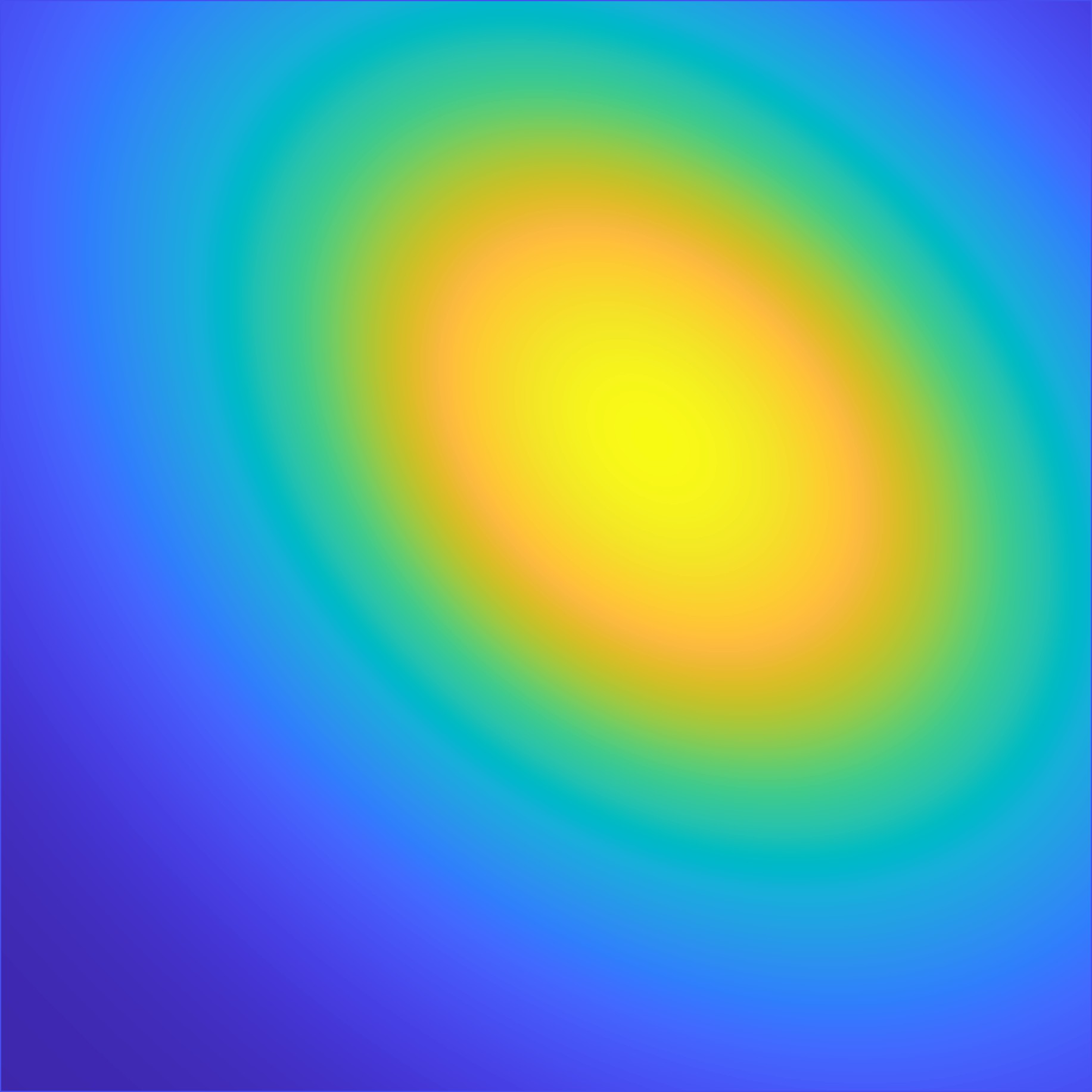} 
        \caption{ } 
        \label{subfig:gauss} 
    \end{subfigure} 
    \hfill 
    \begin{subfigure}[t]{0.192\textwidth}  
        \centering 
        \includegraphics[width=\linewidth, keepaspectratio]{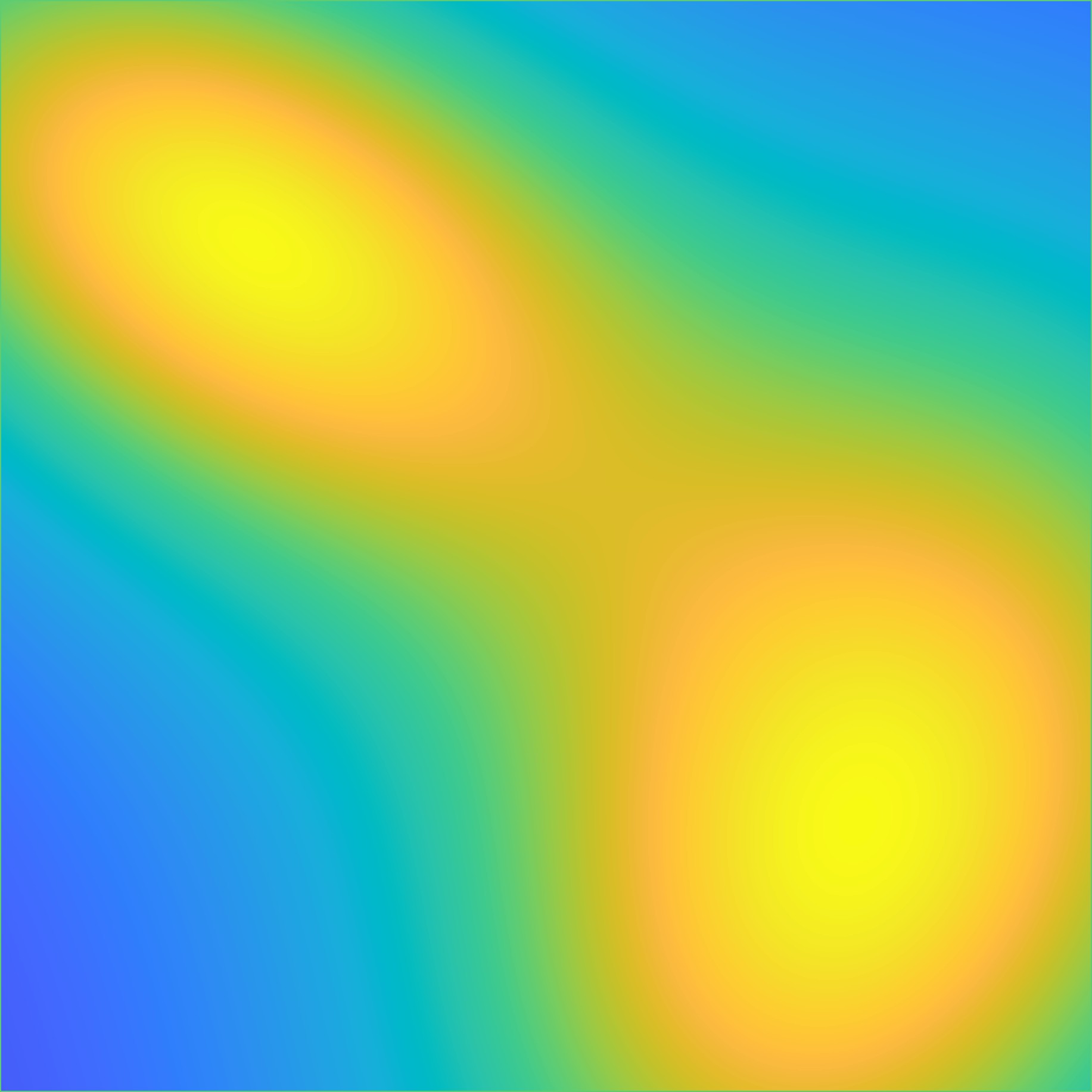} 
        \caption{ } 
        \label{subfig:dual-cauchy} 
    \end{subfigure} 
    \hfill 
    \begin{subfigure}[t]{0.192\textwidth}  
        \centering 
        \includegraphics[width=\linewidth, keepaspectratio]{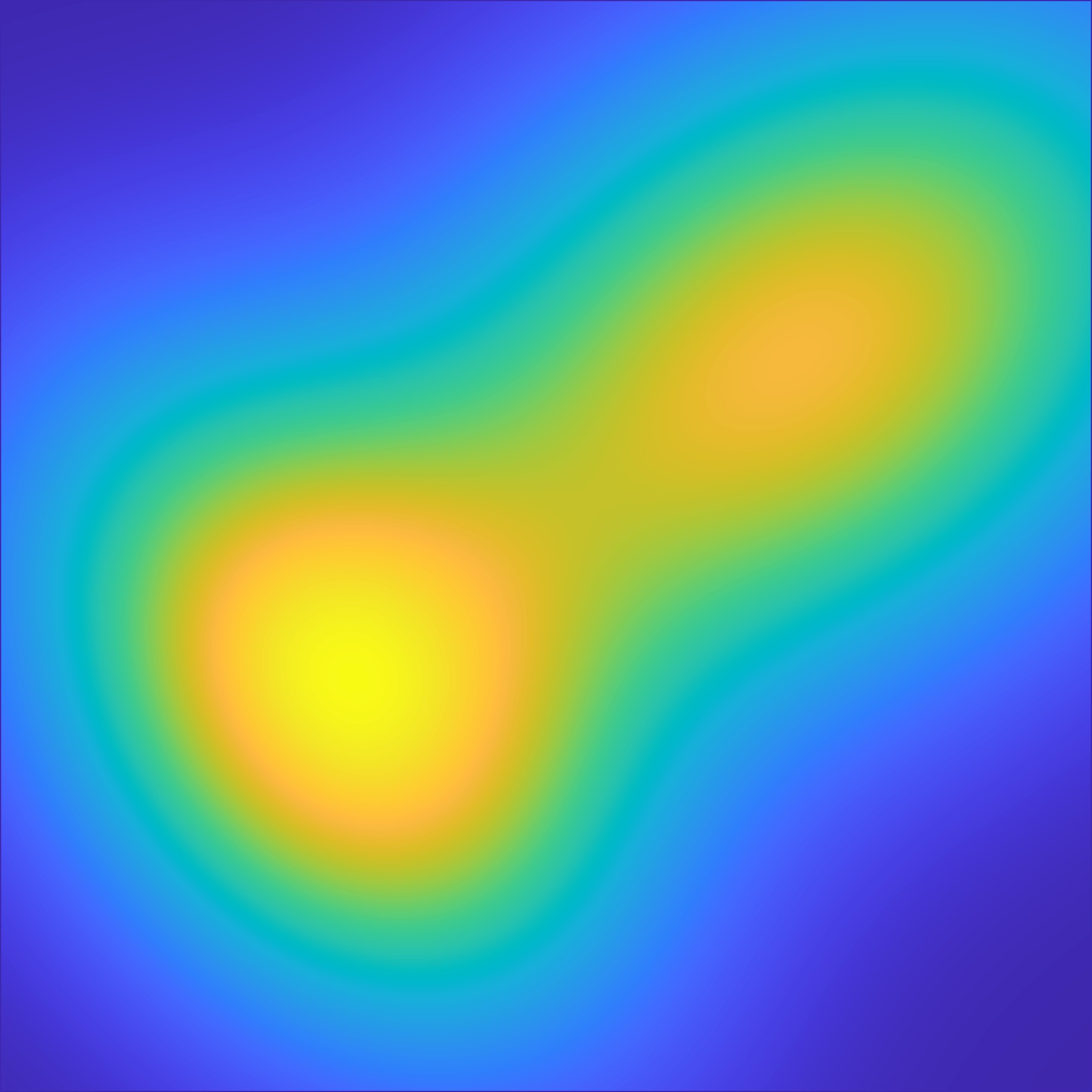} 
        \caption{ } 
        \label{subfig:dual-gauss} 
    \end{subfigure} 
    \hfill 
    \begin{subfigure}[t]{0.192\textwidth}  
        \centering 
        \includegraphics[width=\linewidth, keepaspectratio]{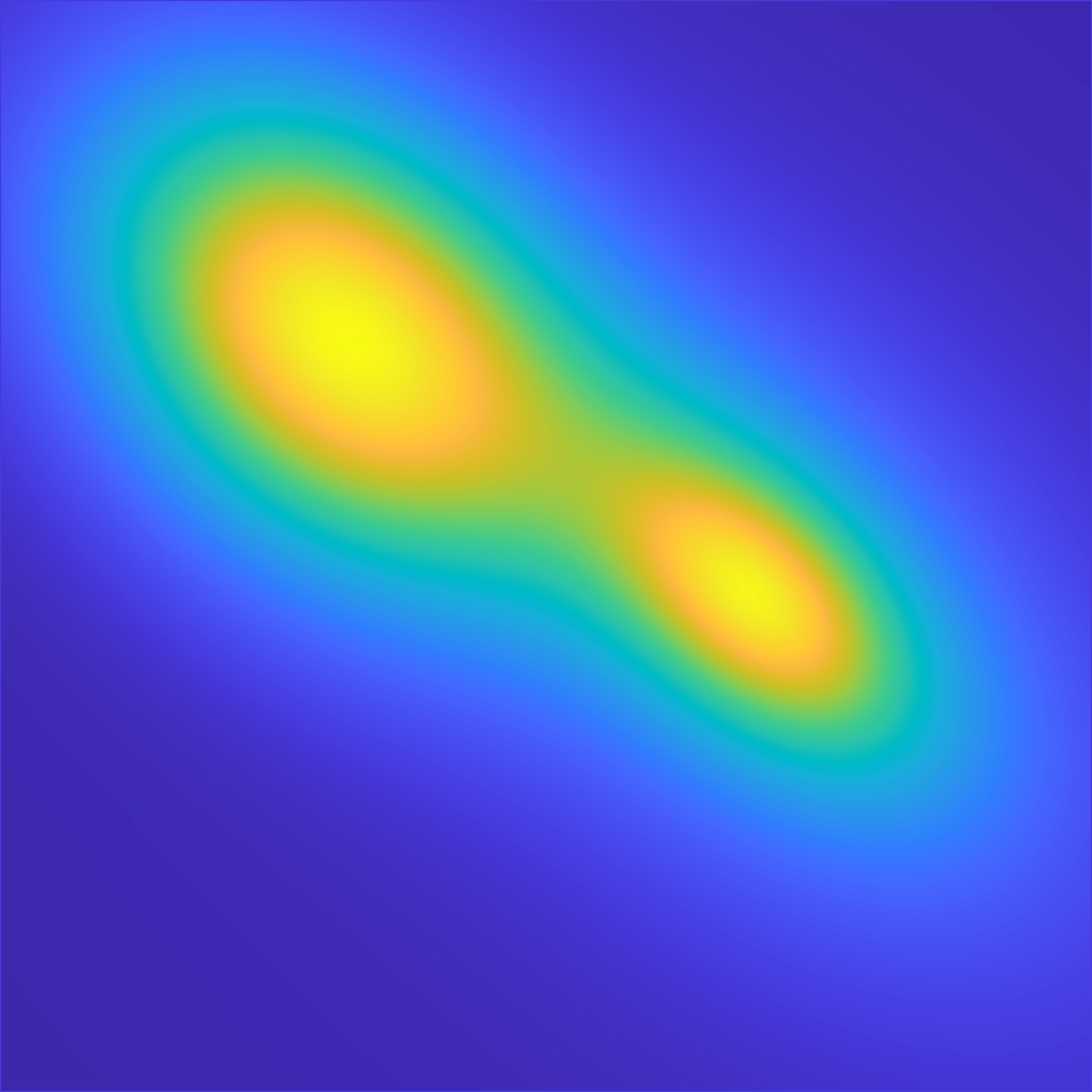} 
        \caption{ } 
        \label{subfig:cauchy-gauss} 
    \end{subfigure} 
    \caption{Test measures generated by discretizing two-dimensional
probability distributions on a $2048\times 2048$ grid. (a) a
    single Cauchy distribution; (b) a single Gaussian distribution; (c) two Cauchy distributions; (d)  two Gaussian distributions; (e) a mixture of a Cauchy distribution and a Gaussian distribution
    }
    \label{fig:test-measures} 
\end{figure}

\begin{table}
\centering
\captionsetup{font=small}
\caption{Numerical results for the ten pairs of test measures on a
$1024\times 1024$ grid. ``IPRM time'' and ``Network time'' denote
the cumulative times spent solving sparse subproblems with IPRM and
the CPLEX network simplex solver, respectively. ``Total time''
includes all components of the corresponding multiscale algorithm.
The symbol ``$\times $'' indicates that the instance is unsolved}
\resizebox{\textwidth}{!}{ 
    \footnotesize 
    \begin{tabular}{c |c c c c c | c c c}
        \toprule 
       \multirow{2}{*}{Instance}  & \multicolumn{5}{c|}{MSIPRM} &\multicolumn{3}{c}{MCPLEX-NS}\\
        \cline{2-9}
       & \makecell{IPRM/Total time (s)} & \makecell{Iter.}& $\textup{KKT}_\textup{res,out}$ & Pinf & K--W &  \makecell{Network/Total time (s)}& $\textup{KKT}_\textup{res,out}$ & K--W
       \\ [1.5ex] 
        \hline 
Test 1 -- Test 2 & 258.24/336.64 & 460 &8.3E-07&2.6E-06&103.32 &2756.41/2870.14&9.6E-07&103.33\\
Test 1 -- Test 3 & 313.49/431.11 & 601 &6.8E-07&5.4E-06& 47.65 &2683.64/2838.21&9.8E-07&47.67\\
Test 1 -- Test 4 & 278.90/602.26 & 467 &4.4E-07&4.2E-06&126.15 &3107.08/3542.92&6.7E-07&126.15\\
Test 1 -- Test 5 & 268.56/539.89 & 425 &4.0E-07&2.9E-06&162.57 &29287.06/31740.50&8.3E-07&162.58\\
Test 2 -- Test 3 & 292.90/380.49 & 519 &4.2E-07&2.3E-05&105.43 &2230.44/2354.85&5.6E-07&105.44\\
Test 2 -- Test 4 & 286.45/838.48 & 484 &4.1E-07&6.3E-06&140.02 &13573.05/15250.20&8.4E-07&140.02\\
Test 2 -- Test 5 & 438.29/693.06 & 489 &6.8E-07&2.7E-06&121.61 & $\times$ & $\times$ & $\times$\\
Test 3 -- Test 4 & 305.24/925.62 & 429 &3.1E-07&6.8E-09&149.21 &3549.03/4056.61&4.0E-07&149.21\\
Test 3 -- Test 5 & 326.13/593.74 & 409 &4.2E-07&2.0E-06&147.39 &2416.54/2695.94&7.9E-07&147.39\\
Test 4 -- Test 5 & 351.72/878.62 & 470 &3.0E-07&2.5E-06&190.73 &2555.96/3200.50&3.8E-07&190.73\\
        \bottomrule 
    \end{tabular}   
}
\label{tab:KW-1024}
\end{table}

MSIPRM solves all ten problems with small KKT residuals and low primal
infeasibility, whereas MCPLEX-NS fails on one problem. For the nine problems solved by both methods, the reported K--W values are nearly identical. On these problems, MCPLEX-NS takes more total time than MSIPRM. Furthermore, the time spent by CPLEX's network simplex solver is typically much longer than that spent by IPRM, often by an order of magnitude and sometimes more.
Overall, MSIPRM's total running times remain within a reasonable range at this scale, whereas MCPLEX-NS exhibits significantly longer running times on two instances. These results show that MSIPRM remains stable on trillion-variable OT problems and works well for different measure pairs.

Table~\ref{tab:KW-1024} also reveals how the total running time is
distributed. In addition to the time spent by IPRM solving the sparse subproblems, a nonnegligible portion of the total running time is spent elsewhere in MSIPRM. Since the full cost matrix is not stored explicitly,
most of this additional computational overhead is
attributable to the evaluations of
$c-\bar A^\top\lambda$, which are needed to update the active
set and assess dual infeasibility. In practice, the overhead is relatively small at resolutions below $512\times512$, but becomes substantial at $512\times512$ and above. As shown in Table~\ref{tab:KW-1024}, this overhead is comparable to the IPRM solution time for some instances.
This overhead is not a weakness of the multiscale strategy. It suggests that parallelizing the evaluation of $c-\bar A^\top\lambda$ could further improve the efficiency of MSIPRM. Finally, the discrepancy between the total running time and the cumulative
solver time provides a rough indication of the number of active-set
updates. The smaller difference observed for MSIPRM suggests fewer such
updates and, consequently, more accurate support predictions.
The following scaling experiment examines this behavior more directly.

\begin{table}
\centering
\captionsetup{font=small}
\caption{Numerical results for the K--W distance between Test measure 2 and Test measure 5 at grid resolutions ranging from $32\times 32$ to
$2048\times 2048$}
\resizebox{0.7\textwidth}{!}{ 
    \footnotesize 
    \begin{tabular}{c |c c c c c c}
        \toprule 
       \multirow{2}{*}{Resolution}  & \multicolumn{6}{c}{MSIPRM}\\
        \cline{2-7}
       & \makecell{IPRM time\,(s)} & \makecell{Iter} & Updates of $\N$ &K--W & $\textup{KKT}_{\textup{res,out}}$ & \makecell{$|\N|$ }\\ [1.5ex] 
        \hline 
$32\times 32$   &   0.38 & 74& 3 &245.37& 4.4E-09    & 1.7E+04   \\
$64\times 64$   &   1.60 &104& 3 &243.67& 6.7E-07    & 6.7E+04   \\
$128\times128$  &   5.95 & 88& 2 &243.25& 2.1E-07    & 2.7E+05   \\
$256\times256$  &  15.45 & 57& 1 &243.16& 1.0E-06    & 1.1E+06   \\
$512\times512$  &  89.67 & 83& 1 &243.13& 1.0E-06    & 4.2E+06   \\
$1024\times1024$& 181.20 & 39& 1 &243.13 \protect\footnotemark{}& 7.6E-07    & 1.8E+07   \\
$2048\times2048$&1186.14 & 59& 1 &243.12& 7.5E-07    & 7.1E+07   \\
        \hline 
  & \multicolumn{6}{c}{Total IPRM time=1480.39 s;\quad \textup{Total time}=4701.13 s}  \\
        \bottomrule 
    \end{tabular}   
}
\label{tab:KW-scaling}
\end{table}
\footnotetext{The values in this table follow the scaling convention used
in our implementation. Under this convention, the K--W distance on the
$2048\times2048$ grid is approximately twice the value on the
$1024\times1024$ grid reported in Table~\ref{tab:KW-1024}.}

\begin{figure}
    \centering 
    \captionsetup{font=small}
    \begin{subfigure}[t]{0.49\textwidth} 
        \centering 
        \includegraphics[width=\linewidth, keepaspectratio]{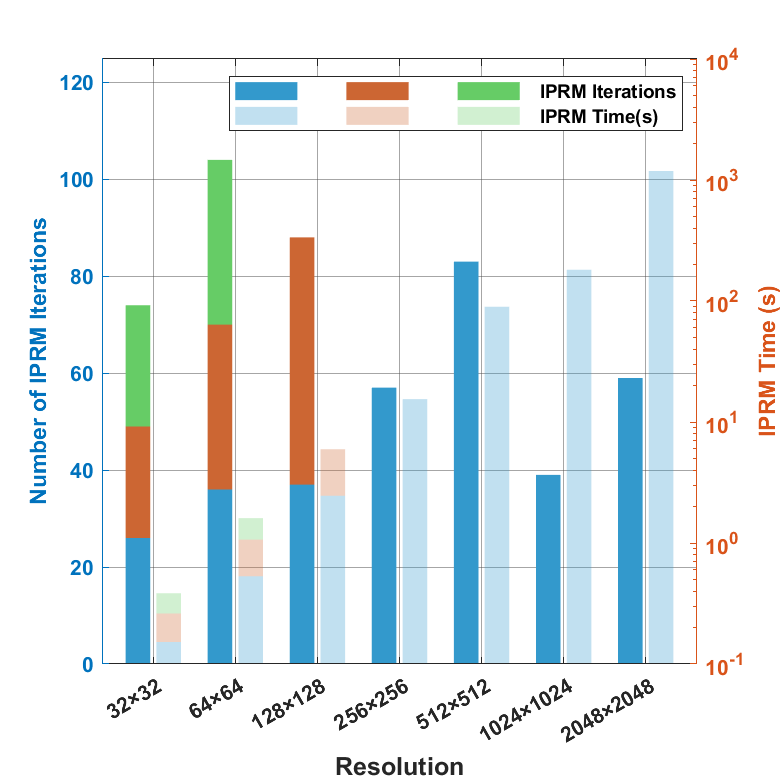} 
        \caption{\,Number of iterations and running time of IPRM on grids with
resolutions up to $2048\times 2048$}
        \label{figure5a}
    \end{subfigure} 
    \hfill 
    \begin{subfigure}[t]{0.49\textwidth}  
        \centering 
        \includegraphics[width=\linewidth, keepaspectratio]{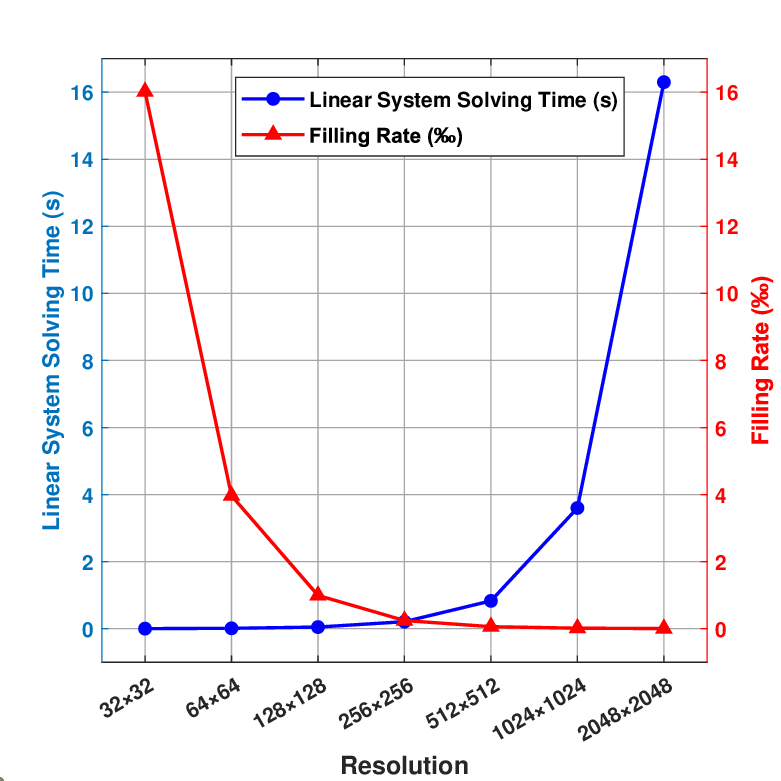} 
        \caption{\,Average time per step for solving the linear system and 
average filling rate of active set $\N$ on grids with resolutions up to $2048\times 2048$}
        \label{figure5b}
    \end{subfigure}  
    \caption{ Computational details of MSIPRM for solving an OT problem on a $2048\times 2048$ grid}   
\label{fig:scaling-details} 
\end{figure}

To study the scaling behavior in more detail, we next fix test measures 2 and 5 and compute the corresponding K--W distance on grids from $32\times 32$ up to $2048\times 2048$; see Table~\ref{tab:KW-scaling} and Figure~\ref{fig:scaling-details}. The computed K--W distance quickly stabilizes as the grid is refined, which indicates that the multiscale sequence gives a numerically consistent approximation of the large-scale OT problems.

Table~\ref{tab:KW-scaling} shows that the total number of IPRM iterations over all scales is $504$, but this number should be interpreted together with Subfigure~\ref{figure5a}. The left panel shows that a large part of the iterations is performed on coarse and medium grids, where each iteration is relatively cheap. In contrast, the finest levels require only a moderate number of iterations. Thus, the iteration count does not grow rapidly with grid resolution. This is an important feature of the method, because it means that the multiscale transfer from one level to the next is effective.

The active-set updates show a similar pattern. Table~\ref{tab:KW-scaling} shows that only two or three updates of $\N$ are required at the coarse levels, while at $256\times 256$ and above a single update is sufficient at each level. 
This strongly suggests that the support predicted from the previous level is already very accurate for the next level. In other words, our multiscale method provides not only a good initial active set but also an efficient update strategy at each level.

Subfigure~\ref{figure5b} provides further evidence of the scalability of MSIPRM. The average time per linear-system solve increases by approximately a factor of four when the grid resolution doubles.
The active-set size grows more slowly than the total number of transport variables, resulting in increasing relative sparsity. This is crucial to the ability of the method to handle very large OT instances.
For example, on a $2048\times2048$ grid, the original problem involves 
8.4E+06 marginal constraints and 
1.8E+13 transport variables. The formulation used by HOT contains approximately 1.3E+07 constraints and 1.7E+10 variables. 
Even after this reduction, the model remains prohibitively
large for our computational environment. In contrast, at the finest level MSIPRM solves a sparse problem with about 8.4E+06 constraints and only 7.1E+07 variables in our experiments. Together with the preceding benchmark results, this substantial reduction helps explain the favorable large-scale performance of MSIPRM relative to HOT and LEMON-NS. 

Taken together, Tables~\ref{tab:KW-1024}--\ref{tab:KW-scaling} and Figure~\ref{fig:scaling-details} provide a clear view of the large-scale behavior of MSIPRM. The method remains stable across different probability measures. As the grid is refined, the active set becomes increasingly sparse, and IPRM efficiently solves the resulting sparse subproblems.
Meanwhile, MSIPRM continues to accurately identify the support
of the optimal transport plan. This demonstrates that the support identification and active-set update strategies are well suited to the multiscale framework. These observations explain why the method can efficiently solve large-scale OT problems.

\section{Conclusion}\label{conclusion}
This paper proposes MSIPRM for solving large-scale OT problems. The method combines a multiscale outer framework with IPRM for sparse inner subproblems. By restricting the transport variables to a small active set and refining this set adaptively, MSIPRM avoids the prohibitive memory requirements of the full transport plan. We also analyze several structural properties of sparse OT problems, including row reduction, problem decomposition, constraint nondegeneracy, and support identification. To improve the efficiency of solving the inner subproblems, we employ a warm-start strategy and compute the Newton direction by solving a reduced Schur complement system.
In addition, we investigate the global and local convergence properties of IPRM in this setting. The numerical results show that MSIPRM is scalable across a broad range of test problems and remains competitive with existing solvers. In particular, it can handle problems on grids of size up to $2048\times 2048$. These results indicate that combining a multiscale framework with IPRM is a promising approach to solving large-scale OT problems.

	
{\small
\,

\noindent\textbf{Declarations}

\medskip

\noindent\textbf{Funding}   R.J.~Zhang was supported by the National Natural Science Foundation of China (No. 1250012017), the Natural Science Foundation of Tianjin, China (No. 24JCQNJC01970), and the Fundamental Research Funds for the Central Universities (Nos. 050-63253088 and 050-63263071). Y.H.~Dai was supported by the National Key  R\&D Program of China (Nos. 2021YFA1000300 and 2021YFA1000301) and the National Natural Science Foundation of China (No. 92473208).

\medskip

\noindent\textbf{Conflict of Interest} The authors have no relevant financial or non-financial interests to disclose.

\medskip

\noindent\textbf{Data Availability}  
Some of the experiments used the public DOTmark benchmark dataset. Data generated during this study will be made available upon reasonable request.}

\section*{Appendix: Proofs}
\refstepcounter{proofcnt} 
\noindent\textbf{\theproofcnt. Proof of Proposition \ref{rankA}:}\label{proof:theorem-1} By \cite[Theorem 8.2.1]{royle2001algebraic}, the rank of $\bar A_\N $ is determined by the connected components of $\G_\N $. For any $1\leq i\ne j\leq c_0 $, if $u_s\in\I_i$, $v_t\in\J_j$, then $(u_s,v_t)\notin\N$. Hence, 
there exist permutation matrices $P$ and $Q$ such that
\[
P\bar A_\N Q=
\begin{pmatrix}
\bar A_1 & 0 & \cdots & 0\\
0 & \bar A_2 & \cdots & 0\\
\vdots & \vdots & \ddots & \vdots\\
0 & 0 & \cdots & \bar A_{c_0}
\end{pmatrix},
\]
where $\bar A_k\in\mathbb{R}^{(|\I_k|+|\J_k|)\times|\N_k|}$
is the incidence matrix associated with the connected component
$\G_{\N_k}$. Substituting this block decomposition into \eqref{sparse} immediately yields the $c_0$ subproblems in \eqref{subproblem}.
\qed

\refstepcounter{proofcnt} 
\noindent\textbf{\theproofcnt. Proof of Theorem \ref{zeromeasure}:}\label{proof:theorem-0} 
It suffices to prove that in
$\mathcal H(m,n)$, $\D(m,n)$ has positive measure and
$\D(m,n)\setminus\E(m,n)$ has measure zero.

It is not hard to prove that $\mathcal H(m,n)=\{\bar Ax\,|\, x\in\mathbb{R}^{mn}\}$ and $\D(m,n)=\{\bar Ax\,|\, x\in\mathbb{R}^{mn}_{++}\}$. Thus, $\bar A$ can be regarded as a surjective linear map from
$\mathbb R^{mn}$ onto $\mathcal H(m,n)$. So $\bar A$ is an open map. Because $\mathbb{R}_{++}^{mn}$ is open, and $\bar A$ is
an open map, $\D(m,n)$ is a nonempty open subset of
$\mathcal H(m,n)$. It therefore has positive measure.
 
If $\bar b=(\bm\mu,\bm\nu)\in\D(m,n)\setminus\E(m,n)$, then there
exist nontrivial partitions
$\{\I_1,\I_2\}$ of $\U$ and
$\{\J_1,\J_2\}$ of $\V$ such that
\eqref{summu} holds. The
equalities in \eqref{summu} impose a nontrivial linear constraint
on $\mathcal H(m,n)$. Consequently, $\bar b$ lies in a proper linear subspace of $\mathcal H(m,n)$ of dimension $m+n-2$, which has measure
zero. Taking the union over all nontrivial partitions of $\U$ and
$\V$ yields a finite union of measure-zero sets containing
$\D(m,n)\setminus\E(m,n)$. Thus,
$\D(m,n)\setminus\E(m,n)$ has measure zero.\qed

\refstepcounter{proofcnt} 
\noindent\textbf{\theproofcnt. Proof of Lemma \ref{samespt}:}\label{proof:lemma-e} Let $\spt(x^*)=\PP$ and $\spt(s^*)=\Q$.
Since the primal and dual constraint nondegeneracy conditions hold at $\omega^*$, we have
\begin{equation*}
    \begin{aligned}
        \PP=\{j\mid s^*_j=0,\,1\leq j\leq |\N|\}\quad\textup{and}\quad \Q=\{j\mid x^*_j=0,\,1\leq j\leq |\N|\}.
    \end{aligned}
\end{equation*}
Moreover, $|\PP|=m+n-1$ and $|\Q|=|\N|-|\PP|$. Assume without loss of generality that $A_\N$ admits a partition $A_\N=[A_{\PP}\;A_{\Q}]$, where $ \,A_{\PP}\in\mathbb{R}^{|\PP|\times |\PP|}$ and $ A_{\Q}\in\mathbb{R}^{|\PP|\times |\Q|}$. Then $A_\PP$ is nonsingular. It follows that
$x^*_{\Q}=0,\,x^*_{\PP}>0,\, A_{\PP}x_{\PP}^*=b$, and $s^*_{\Q}>0,\,s^*_{\PP}=0,\, A_{\PP}^\top \lambda^*=c_{\PP},\,A^\top_{\Q}\lambda^*+s_{\Q}^*=c_{\Q}$.
Let the perturbation vector $\xi$ satisfy
\begin{equation}\label{boundofxi}
\|\xi\|<\frac{\min\big\{\min_{j\in \PP}x^*_j,\min_{j\in\Q}s^*_j/\rho\big\}}{\sqrt{1+\|A_{\PP}^{-1}A_{\Q}\|^2}}.
\end{equation}
We claim that the triple $(\hat x,\hat\lambda,\hat s)$ constructed
as follows is a KKT triple for the perturbed primal--dual pair
\eqf{perpp} and \eqf{perdp}:
\begin{equation}
\begin{aligned}
    \hat x_{\Q}&=0,\quad \hat x_{\PP}=x^*_{\PP}+\xi_\PP+A_{\PP}^{-1}A_{\Q}\xi_{\Q},\\
   \hat s_{\PP}&=0,\quad \hat s_{\Q}=s^*_{\Q}+\rho\xi_\Q-\rho A_{\Q}^\top A_{\PP}^{-\top}\xi_{\PP},\quad \hat\lambda=\lambda^*+\rho A_{\PP}^{-\top}\xi_{\PP}. 
   \end{aligned}
\end{equation}
First, we have
$A_\N\hat x=b+A_\N\xi$, $A_\N^\top\hat\lambda+\hat s=c_\N+\rho\xi$ and $\hat x\circ\hat s=0$.  
Moreover,
\begin{equation*}
    \begin{aligned}
&\hat x_{\PP}\geq x^*_{\PP}-(\|\xi_{\PP}\|+\|A_{\PP}^{-1}A_{\Q}\|\|\xi_{\Q}\|)\mathbf{1}_{|\PP|}\geq x^*_{\PP}-\sqrt{1+\|A_{\PP}^{-1}A_{\Q}\|^2}\|\xi\|\textbf{1}_{|\PP|}>0,\\
&\hat s_{\Q}\geq s^*_{\Q}-\rho(\|A_{\Q}^{\top}A_{\PP}^{-\top}\|\|\xi_{\PP}\|+\|\xi_{\Q}\|)\textbf{1}_{|\Q|}\geq s^*_{\Q}-\rho\sqrt{1+\|A_{\PP}^{-1}A_{\Q}\|^2}\|\xi\|\textbf{1}_{|\Q|}>0.
    \end{aligned}
\end{equation*}
Thus, $(\hat x,\hat\lambda,\hat s)$ is a KKT triple for the
perturbed primal--dual pair \eqf{perpp} and \eqf{perdp}, and $\spt(\hat x)=\spt(x^*),\,\spt(\hat s)=\spt(s^*)$. It remains to prove that
\[\|A_{\PP}^{-1}A_{\Q}\|\leq 2\sqrt{(m+n-1)(|\N|-(m+n-1))}.\]
Since $A_\PP$ is a nonsingular submatrix of the totally
unimodular matrix $A_\N$, it is totally unimodular, and $(A_{\PP}^{-1})_{ij}\in\{0,\pm1\}$. Each column of $A_\Q$ contains at most two nonzero entries, both equal to $1$.
Then
$(A_{\PP}^{-1}A_{\Q})_{ij}\in\{0,\pm1,\pm2\}$. Hence
\[\|A_{\PP}^{-1}A_{\Q}\|\leq\|A_{\PP}^{-1}A_{\Q}\|_F\leq 2\sqrt{(m+n-1)(|\N|-(m+n-1))}.\]
This completes the proof. \qed

\refstepcounter{proofcnt} 
\noindent\textbf{\theproofcnt. Proof of Lemma \ref{boundx}:}\label{proof:lemma-c}
(i) By the update rules for $\mu^{(k)}$ and $\rho^{(k)}$ in Algorithm~\ref{IPRM}, $\{\mu^{(k)}\}$ is nonincreasing and $\{\rho^{(k)}\}$ is nondecreasing. For $k\in\mathbb{Z}_+$ and $l\in\mathbb{Z}_+$, define $\mu^{(k)}_l:=\mu^{(k)}/\eta^l$. With $x$, $s$, and $\rho$ fixed, we write
$\hat z(\mu):=z(x,s;\mu,\rho)$. We consider the following three cases.

First, suppose that there exists $k_0\in\mathbb{Z}_+$ such that
$\mu^{(k_0)}\ge\eta\phi^{(k_0)}$ and the inner loop in Step~5 of
Algorithm~\ref{IPRM} does not terminate. Since $\eta>1$, we have
\begin{equation}\label{phimu=0}
    \lim_{l\to\infty}\mu^{(k_0)}_l=0\quad \textup{and}\quad\lim_{l\to\infty}\phi_{(\mu^{(k_0)}_l,\rho^{(k_0)})}(\omega^{(k_0)})=0.
\end{equation}

Second, by the definition of $\gamma$, we have
\begin{equation*}\label{Deltamu}  \Delta\mu^{(k)}=\begin{cases}
\quad 0, & \text{if } \frac{\mu^{(k)}}{\phi^{(k)}}\leq \gamma_0, \\
-\mu^{(k)}+\gamma_0\phi^{(k)},  &\text{if } \frac{\mu^{(k)}}{\phi^{(k)}}>\gamma_0.
\end{cases}
\end{equation*} 
The inequality \eqref{partial2rho} states that $\phi_{(\mu^{(k+1)},\rho)}(\omega^{(k+1)})$ is nonincreasing with respect to $\rho$.
If $\mu^{(k)}\geq\gamma_0\phi^{(k)}$, then after Step~3 of
Algorithm~\ref{IPRM}, inequality~\eqref{ineq} implies that
\begin{equation}\label{muk+1}
    \frac{\mu^{(k+1)}}{\phi^{(k+1)}}\geq\frac{(1-\alpha^{(k)})\mu^{(k)}+\alpha^{(k)}\gamma_0\phi^{(k)}}{(1-2\tau\alpha^{(k)})\phi^{(k)}}\geq\frac{(1-\alpha^{(k)})\gamma_0}{1-2\tau\alpha^{(k)}}+\frac{\alpha^{(k)}\gamma_0}{1-2\tau\alpha^{(k)}}> \gamma_0.
\end{equation}

Third, suppose that the inner loop in Step~5 of
Algorithm~\ref{IPRM}, executed during the $k$-th iteration, terminates
after finitely many steps. Then, for some integer $l\geq 1$, the termination conditions yield
\begin{equation}\label{mul-1}
    \begin{aligned}
        \mu^{(k+1)}_{l-1}&\ge\eta\phi_{(\mu_{l-1}^{(k+1)},\rho^{(k+1)})}(\omega^{(k+1)})=\frac{\eta}{2}\|\hat z(\mu_{l-1}^{(k+1)})-x^{(k+1)}\|^2,\\
 \mu^{(k+1)}_{l}&<\eta\phi_{(\mu_{l}^{(k+1)},\rho^{(k+1)})}(\omega^{(k+1)})=\frac{\eta}{2}\|\hat z(\mu_{l}^{(k+1)})-x^{(k+1)}\|^2.
    \end{aligned}
\end{equation}
Step~5 updates $\mu^{(k+1)}\gets\mu_{l}^{(k+1)}$ and $\phi^{(k+1)}\gets \phi_{(\mu_{l}^{(k+1)},\rho^{(k+1)})}(\omega^{(k+1)})$. From \eqref{zmu1mu2} and \eqref{mul-1}, we obtain
\begin{equation}\label{upperphi2}
\begin{split}
      \phi^{(k+1)}&\leq \frac{1}{2}(\|\hat z(\mu_{l}^{(k+1)})-\hat z(\mu_{l-1}^{(k+1)})\|+\|\hat z(\mu_{l-1}^{(k+1)})-x^{(k+1)}\|)^2\\
      &\leq \frac{1}{2}\left(\sqrt\frac{\cN}{\rho^{(k+1)}}(\sqrt{\mu_{l-1}^{(k+1)}}-\sqrt{\mu_{l}^{(k+1)}})+\sqrt\frac{2\mu_{l-1}^{(k+1)}}{\eta}\right)^2\\
      &\leq\frac{1}{2}\left(\sqrt{\frac{\cN}{\rho^{(0)}}}(\sqrt\eta-1)+\sqrt2\right)^2\mu^{(k+1)}\\
      &< \mu^{(k+1)}/\gamma_0,
\end{split}
\end{equation}
where the last inequality follows from the parameter choices in Algorithm~\ref{IPRM}: 
\[0<\theta_0,\,0<\gamma_0<\frac{1}{(1+\theta_0)^2},\, \textup{and}\,1<\eta<\left(1+\theta_0\sqrt{{2\rho^{(0)}}/{{\cN}}}\right)^2.\]
Therefore, after Step~5, one has $\gamma_0\phi^{(k+1)}<\mu^{(k+1)}<\eta\phi^{(k+1)}$.

Equation \eqref{phimu=0} indicates that $\omega^{(k_0)}$ is a KKT triple for the primal--dual pair \eqf{pp} and \eqf{dp}. By the update rules for $\mu^{(k)}$ and $\phi^{(k)}$, once either $\Delta\mu^{(k)}<0$ or the inner loop in Step~5 is entered, inequalities~\eqref{upperphi2} and
\eqref{muk+1} ensure that
$\gamma_0\phi^{(k)}< \mu^{(k)}<\eta\phi^{(k)}$ 
holds for all subsequent iterates. Thus, $\{\phi^{(k)}\}$ is bounded. If neither event ever occurs, then $\mu^{(k)}$ remains constant. It then follows
from \eqref{ineq} that $\{\phi^{(k)}\}$ is decreasing.
Therefore, the sequence $\{\phi^{(k)}\}$ is bounded in either case.

(ii) By part (i), the sequence $\{\|\xi^{(k)}\|\}$ is bounded.   
Suppose, to the contrary, that $\{z^{(k)}\}$ is unbounded. Without loss of generality, we assume that $\|z^{(k)}\|\to\infty$. Then,
$A_\N\frac{z^{(k)}}{\|z^{(k)}\|}=\frac{b+A_\N\xi^{(k)}}{\| z^{(k)}\|}\to 0$.
By passing to a convergent subsequence, we obtain
$\lim_{j\to\infty}\frac{z^{(k_j)}}{\|z^{(k_j)}\|}=\hat z$
for some $\hat z\geq 0$ satisfying $\|\hat z\|=1$ and
$A_\N\hat z=0$. This yields a contradiction, because $\hat z$ must be the zero solution.
Therefore, $\{z^{(k)}\}$ is bounded. Since $\|x^{(k)}\|\leq\|z^{(k)}\|+\|\xi^{(k)}\|$, $\{x^{(k)}\}$ is bounded as well. \qed

\refstepcounter{proofcnt} 
\noindent\textbf{\theproofcnt. Proof of Theorem \ref{uppermuk}:}\label{proof:theorem-d}
       Let $(x^*_{\xi^{(k)}},\lambda^*_{\xi^{(k)}},s^*_{\xi^{(k)}})$ be a KKT triple for the perturbed primal--dual pair \eqf{perpp} and \eqf{perdp} with $\rho=\rho^{(k)}$ and $\xi=\xi^{(k)}$.
Then, $(z^{(k)},\lambda^{(k)},\rho^{(k)}y^{(k)})$ is on the 	
central path of the perturbed pair and satisfies
$z^{(k)}_jy^{(k)}_j=\mu^{(k)}/\rho^{(k)}$. Since
\[(x^*_{\xi^{(k)}}-z^{(k)})^\top(s^*_{\xi^{(k)}}-\rho^{(k)}y^{(k)})=0\quad \textup{and}\quad x^*_{\xi^{(k)}}\circ s^*_{\xi^{(k)}}=0,\]
expanding the first equality and using the second one yields
\begin{equation*}
    \begin{aligned}
      \sum_{j\in\spt(s^*_{\xi^{(k)}})}{(s^*_{\xi^{(k)}})_j}{ z^{(k)}_j}+\sum_{j\in\spt(x^*_{\xi^{(k)}})}\rho^{(k)}{ (x^*_{\xi^{(k)}})_j}{y^{(k)}_j}&=\mu^{(k)}|\N|,\\
      \sum_{j\in\spt(s^*_{\xi^{(k)}})}\frac{(s^*_{\xi^{(k)}})_j}{\rho^{(k)} y^{(k)}_j}+\sum_{j\in\spt(x^*_{\xi^{(k)}})}\frac{(x^*_{\xi^{(k)}})_j}{z^{(k)}_j}&=|\N|.  
    \end{aligned}
\end{equation*}
By Lemma \ref{samespt}, if $\xi^{(k)}$ satisfies inequality \eqref{boundofxi}, one has
\begin{equation}\label{deltaP}
\begin{aligned}
\delta^{(k)}_P&:=\min\limits_{j\in\spt(x^*)}(x^*_{\xi^{(k)}})_j\geq\min\limits_{j\in\spt(x^*)}x^*_j-\sqrt{1+\|A_{\PP}^{-1}A_{\Q}\|^2}\|\xi^{(k)}\|>0,\\
\delta^{(k)}_D&:=\min\limits_{j\in\spt(s^*)}(s^*_{\xi^{(k)}})_j\geq\min\limits_{j\in\spt(s^*)}s^*_j-\rho^{(k)}\sqrt{1+\|A_{\mathcal{Q}}^{\top}A_{\PP}^{-\top}\|^2}\|\xi^{(k)}\|>0.       
\end{aligned}
\end{equation}
Furthermore, their supports coincide with those of the original KKT
triple; namely,
$\spt(x^*_{\xi^{(k)}})=\spt(x^*)$ and $\spt(s^*_{\xi^{(k)}})=\spt(s^*)$.
Following the argument of
\cite[Theorem~2.1]{andersen1996combining}, we require that
\begin{equation}\label{requiremu}
  \sqrt{\frac{\mu^{(k)}}{\rho^{(k)}}}\leq\frac{\delta_P^{(k)}}{|\N|}\quad \textup{and}\quad \sqrt{\frac{\mu^{(k)}}{\rho^{(k)}}}\leq\frac{\delta_D^{(k)}/\rho^{(k)}}{|\N|}.
\end{equation}
Note that $\|\xi^{(k)}\|\leq \sqrt{2\mu^{(k)}/\gamma_0}$ for all $k\ge k_1$. If $\mu^{(k)}$ satisfies
\[\sqrt{{\mu^{(k)}}}\leq\frac{\min\{\min_{j\in \spt(x^*)}x^*_j,\min_{j\in\spt(s^*)}s^*_j/\rho^{(k)}\}}{|\N|/\sqrt{\rho^{(k)}}+\sqrt{2(1+4(m+n-1)(|\N|-(m+n-1)))/\gamma_0}},\]
then the condition of Lemma~\ref{samespt}, the estimates
in \eqref{deltaP} and the inequalities in \eqref{requiremu} are
satisfied. Thus, the supports of $x^*$ and $s^*$ are given
by the following partition:
\begin{equation*}
\begin{aligned}
&\spt(x^*)=\spt(x^*_{\xi^{(k)}})=\left\{j\mid z^{(k)}_j\geq y^{(k)}_j\right\}=\Big\{j\,\Big{|}\, z_j^{(k)}
\ge \sqrt{\mu^{(k)}/\rho^{(k)}}\Big\},\\
&\spt(s^*)=\spt(s^*_{\xi^{(k)}})=\left\{j\mid z^{(k)}_j< y^{(k)}_j\right\}=\Big\{j\,\Big{|}\, z_j^{(k)}
< \sqrt{\mu^{(k)}/\rho^{(k)}}\Big\}.
\end{aligned}
\end{equation*}
This completes the proof. \qed


\refstepcounter{proofcnt} 
\noindent\textbf{\theproofcnt. Proof of Theorem \ref{Hoffman}:}\label{proof:theorem-6} 
(i) The conclusion follows in part~(i) directly from~\eqref{phimu=0}.

(ii) By Lemma~\ref{boundx}, the sequence $\{x^{(k)}\}$ is bounded. Moreover, the update rule for $\rho^{(k)}$, together with the boundedness of $\{\rho^{(k)}\}$, implies that $\{s^{(k)}\}$ is bounded. It then follows from
$\lambda^{(k)}=(A_\N A_\N^\top)^{-1}A_\N(c_\N-s^{(k)})$ that $\{\lambda^{(k)}\}$ is also bounded. 

We first prove that $\lim_{k\to \infty}\mu^{(k)}=0$. If Step~5 of Algorithm~\ref{IPRM} is invoked infinitely often,
the update rule for $\mu^{(k)}$ directly implies that $\mu^{(k)}\to0$.
Otherwise, Step~5 is invoked only finitely many times and hence is not invoked for all sufficiently large $k$. Then $\{\phi^{(k)}\}$ is eventually decreasing.
An argument analogous to the proof of \cite[Theorem 5.2]{zhang2024iprsdp} shows that $\mu^{(k)}\to 0$ in this case. Thus, $\{\mu^{(k)}\}$ cannot remain unchanged and there exists $k_1\in\mathbb Z_+$ such that $\gamma_0\phi^{(k)}<\mu^{(k)}<\eta\phi^{(k)}$ for all $k\geq k_1$. Then $\lim_{k\to\infty}\phi^{(k)}=0$ as well. By the boundedness of $\{\omega^{(k)}\}$, every cluster point of $\{\omega^{(k)}\}$ is a KKT triple for~\eqf{pp} and~\eqf{dp}.

We now establish the error bound~\eqref{Errobound}.
Let $\tilde z^{(k)}=z(x^{(k)},s^{(k)};0,\rho^{(k)})$ and $\tilde \xi^{(k)}=\tilde z^{(k)}-x^{(k)}$. Setting $\hat\mu_1=\mu^{(k)}$ and $\hat\mu_2=0$ in~\eqref{zmu1mu2} and using that $\rho^{(k)}\ge\rho^{(0)}$, we obtain, for all $k\geq k_1$,
\begin{equation}\label{ineqxi}
\begin{aligned}
    \|\tilde \xi^{(k)}\|&\leq \|\tilde z^{(k)}-z^{(k)}\|+\|z^{(k)}-x^{(k)}\|\leq\left(\sqrt{\frac{\cN}{\rho^{(0)}}}+\sqrt{\frac{2}{\gamma_0}}\right)\sqrt{\mu^{(k)}}.
\end{aligned}
		\end{equation}
        
By Lemma \ref{central}, we conclude that
$(\tilde z^{(k)},\lambda^{(k)}, s^{(k)}+\rho^{(k)}\tilde\xi^{(k)})$
is a KKT triple for the perturbed primal--dual pair
\eqf{perpp} and
\eqf{perdp} with $\rho=\rho^{(k)}$ and $\xi=\tilde\xi^{(k)}$, i.e.,
$
    (\tilde z^{(k)},\lambda^{(k)})\in 
       \mathcal{S}(A_\N,b+A_\N\tilde\xi^{(k)},c_\N+\rho^{(k)}\tilde\xi^{(k)})$.
By the Hoffman error bound \cite[Theorem~2]{robinson1973bounds}, there exists $C_H>0$, depending only on
$A_\N$, $b$, and $c_\N$, such that for all $k\ge k_1$, one can find $(\bar x^{(k)},\bar \lambda^{(k)})\in\mathcal{S}(A_\N,b,c_\N)$ satisfying
\begin{equation*}
\begin{aligned}
    \left\Arrowvert\begin{pmatrix}
\tilde z^{(k)}\\			
\lambda^{(k)} 
\end{pmatrix}-\begin{pmatrix}
\bar x^{(k)}\\			
\bar \lambda^{(k)} 
\end{pmatrix}\right\Arrowvert&\leq C_H  \left\Arrowvert\left(\begin{array}{ccc}
			 \rho^{(k)}\tilde\xi^{(k)}_+\\
		    A_\N\tilde\xi^{(k)} \\
		   (\rho^{(k)}\tilde\xi^{(k)})^\top\tilde z^{(k)}-(A_\N\tilde\xi^{(k)})^\top\lambda^{(k)}
		   \end{array}\right)\right\Arrowvert\\
		   &\leq C_H\sqrt{(\rho^{(k)})^2+\|A_\N\|^2+\|\rho^{(k)}\tilde z^{(k)}-A_\N^\top\lambda^{(k)}\|^2}\|\tilde\xi^{(k)}\|.
		   \end{aligned}
\end{equation*}
Define $C':=C_H\cdot\sup\limits_{k\geq0}\sqrt{(\rho^{(k)})^2+\|A_\N\|^2+\|\rho^{(k)}\tilde z^{(k)}-A_\N^\top\lambda^{(k)}\|^2}$. The boundedness of the relevant sequences ensures that $C'<\infty$.
Consequently, \eqref{ineqxi} yields
\begin{equation*}
\begin{aligned}
  \dist((x^{(k)},\lambda^{(k)}), \mathcal{S}(A_\N,b,c_\N))&\leq\left\Arrowvert\begin{pmatrix}
\tilde z^{(k)}\\			
\lambda^{(k)} 
\end{pmatrix}-\begin{pmatrix}
\bar x^{(k)}\\			
\bar \lambda^{(k)} 
\end{pmatrix}\right\Arrowvert+\|\tilde z^{(k)}-x^{(k)}\|\\
&\leq(C'+1)\|\tilde \xi^{(k)}\|\\
&\leq \left(\sqrt\frac{\cN}{\rho^{(0)}}+\sqrt{\frac{2}{\gamma_0}}\right)(C'+1)\sqrt{\mu^{(k)}}.
\end{aligned}
\end{equation*}
   Let $\widetilde C=\left(\sqrt\frac{\cN}{\rho^{(0)}}+\sqrt{\frac{2}{\gamma_0}}\right)(C'+1)$ and the proof is complete. 
   \qed
	
\refstepcounter{proofcnt} 
\noindent\textbf{\theproofcnt. Proof of Theorem \ref{quadratic}:}\label{proof:theorem-g} 
The primal and dual constraint nondegeneracy conditions imply that $x^*+s^*>0$. Then $\xi$ admits the following second-order expansion at
$(0,x^*,s^*)$: for every sufficiently small $(d_\mu,d_x,d_s)\in\mathbb{R}_+\times \mathbb{R}^{\cN}\times\mathbb{R}^{\cN}$,
\begin{equation}\label{xirho}
    \begin{aligned}
         &\quad\|\xi(x^*+d_x,s^*+d_s;d_\mu,\rho)-\xi(x^*,s^*;0,\rho)-V(d_\mu,d_x,d_s)\|\\
         &=O(\|(d_\mu,d_x,d_s)\|^2),
    \end{aligned}
\end{equation}
where $V(d_\mu,d_x,d_s)=\frac{d_\mu}{\rho}(Z+Y)^{-1}\mathbf1_{\cN}
-(Z+Y)^{-1}Yd_x
-\frac{1}{\rho} (Z+Y)^{-1}Zd_s$. Since $\xi^*:=\xi(x^*,s^*;0,\rho)=0$, we apply~\eqref{xirho} with $(d_\mu^{(k)},d_x^{(k)},d_s^{(k)})=(\mu^{(k)},x^{(k)},s^{(k)})-(0,x^*,s^*)$. Using $\omega^{(k)}\in\F$ for all $k\ge 0$, we have
\begin{equation}\label{ineq1}
  \|\psi^{(k)}\|=\|\xi^{(k)}\|=\|\xi^{(k)}-\xi^*\|=O(\|(\mu^{(k)},\omega^{(k)})-(0,\omega^*)\|).
\end{equation}
Recall that $\mu^{(k)}<\eta\phi^{(k)}$ and $\gamma=\min\{\gamma_0,\mu^{(k)}/\phi^{(k)}\}$. It follows from \eqref{ineq1} that 
\begin{equation}\label{esti}
|\Delta\mu^{(k)}|\leq \mu^{(k)}<\eta\phi^{(k)}=\frac{\eta}{2}\|\xi^{(k)}\|^2=O(\|(\mu^{(k)},\omega^{(k)})-(0,\omega^*)\|^2).
\end{equation}
Since $\mu^{(k)}\to 0$, the estimate~\eqref{esti} implies that
\begin{equation}\label{absmu}
 \mu^{(k)}=O(\|(\omega^{(k)}-\omega^*)\|^2)\quad\text{and}\quad |\Delta\mu^{(k)}|=O(\|(\omega^{(k)}-\omega^*)\|^2).
\end{equation}
Let $\Theta^{(k)}:=\psi^{(k)}+H^{(k)}(\omega^{(k)}-\omega^*)$. Using the definition of $H_{(\mu,\rho)}(\omega)$ in~\eqref{linearsystem},
together with~\eqref{ineq1} and~\eqref{absmu}, we obtain
\begin{equation}\label{ineq2}
\begin{aligned}
\|\Theta^{(k)}\|
&=\left\|\xi^{(k)}-V^{(k)}(\mu^{(k)},x^{(k)}-x^*,s^{(k)}-s^*)+\frac{\mu^{(k)}}{\rho^{(k)}}(Z^{(k)}+Y^{(k)})^{-1}\bm 1_{\cN}\right\|\\
&=O(\|(\mu^{(k)},x^{(k)}-x^*,s^{(k)}-s^*)\|^2)+O(\mu^{(k)})\\
&=O(\|\omega^{(k)}-\omega^*\|^2).
\end{aligned}
\end{equation}
If $\mu>0$, the matrix $H_{(\mu,\rho)}(\omega)$ is nonsingular. Lemma \ref{nonsingular} tells us that $H^*$ is nonsingular. For any two nonsingular matrices $H_1$ and $H$, it holds that
\begin{equation}\label{h1-h}
\|H_1^{-1}-H^{-1}\|\leq\frac{\|H^{-1}\|^2\|H_1-H\|}{1-\|H^{-1}\|\|H_1-H\|},\quad \textup{if }\|H_1-H\|<\frac{1}{\|H^{-1}\|}.
\end{equation}
Therefore, for any fixed constant $0<\delta<1$, 
one can find a sufficiently small $\epsilon>0$ such that $\|H_{(\mu,\rho)}(\omega)-H^*\|<\frac{\delta}{\|(H^*)^{-1}\|}$ for all $(\mu,\omega)\in U_\epsilon:=\{(\mu,\omega)\mid\|(\mu,\omega)-(0,\omega^*)\|<\epsilon,\,\mu>0\}$.
It then follows from~\eqref{h1-h} that
\begin{equation}\label{hmurho}
\begin{aligned}
    \|(H_{(\mu,\rho)}(\omega))^{-1}\|
    &\leq \frac{\|(H^*)^{-1}\|}{1-\delta},
    \end{aligned}
\end{equation}
which shows that $\|(H_{(\mu,\rho)}(\omega))^{-1}\|$ is uniformly bounded for $(\mu,\omega)\in U_\epsilon$. 
Because $(0,\omega^{*})$ is a cluster point, there exists a sufficiently large $k\in\mathbb{Z}_+$ such that $(\mu^{(k)},\omega^{(k)})\in U_\epsilon$ and 
all the preceding local estimates hold at $k$.
Then, 
\begin{equation*}\begin{aligned}
\|\omega^{(k)}-\omega^*\|&=\|(H^{(k)})^{-1}(\Theta^{(k)}-\psi^{(k)})\|= O(\|\psi^{(k)}\|)+O(\|\omega^{(k)}-\omega^*\|^2).
\end{aligned}
\end{equation*}
Hence, $\|\omega^{(k)}-\omega^*\|=O(\|\psi^{(k)}\|)=O(\|\xi^{(k)}\|)$.

Using the Newton system~\eqref{linearsystem} together
with~\eqref{ineq1} and~\eqref{ineq2}, we obtain
\begin{equation}\label{muomega}
\begin{aligned}
   &\quad\|(\mu^{(k)},\omega^{(k)})+(\Delta\mu^{(k)},\Delta\omega^{(k)})-(0,\omega^*)\|
    \\
    &\leq\gamma\phi^{(k)}+\|\omega^{(k)}+\Delta\omega^{(k)}-\omega^*\|\\
     &\leq\gamma_0\phi^{(k)}+\|(H^{(k)})^{-1}(H^{(k)}(\omega^{(k)}-\omega^*)+H^{(k)}\Delta\omega^{(k)})\| \\
     &\leq O(\|\xi^{(k)}\|^2)+O(\|\Theta^{(k)}\|)+O\left(\left\|\frac{\Delta\mu^{(k)}}{\rho^{(k)}}(Z^{(k)}+Y^{(k)})^{-1}\mathbf{1}_{\cN}
			\right\|\right)\\
            &= O(\|\xi^{(k)}\|^2)+O(\|\Theta^{(k)}\|)+O(|\Delta\mu^{(k)}|)\\
		&= O(\|(\mu^{(k)},\omega^{(k)})-(0,\omega^*)\|^2).
\end{aligned}
\end{equation}
In particular, by \eqref{absmu} and~\eqref{muomega},
$\|\omega^{(k)}+\Delta\omega^{(k)}-\omega^*\|=O(\|\omega^{(k)}-\omega^*\|^2)$. By the local Lipschitz continuity of the residual mapping,
\begin{equation*}
\begin{aligned}
\phi_{(\mu^{(k)}+\Delta\mu^{(k)},\rho^{(k)})}(\omega^{(k)}+\Delta\omega^{(k)})&=O(\|(\mu^{(k)},\omega^{(k)})+(\Delta\mu^{(k)},\Delta\omega^{(k)})-(0,\omega^*)\|^2)\\
      &=O(\|(\mu^{(k)},\omega^{(k)})-(0,\omega^*)\|^4)\\
     &=O((\phi_{(\mu^{(k)},\rho^{(k)})}(\omega^{(k)}))^2).
     \end{aligned}
\end{equation*}
This implies that the full step is accepted by the line search; that is, in Step~4 of Algorithm~\ref{IPRM}, we have
\begin{equation}\label{k+1}
    (\mu^{(k+1)},\omega^{(k+1)})=(\mu^{(k)},\omega^{(k)})+(\Delta\mu^{(k)},\Delta\omega^{(k)}).
\end{equation}
Step~5 of Algorithm~\ref{IPRM} leaves $\omega^{(k+1)}$ unchanged and may only reduce $\mu^{(k+1)}$. Thus, regardless of whether the inner loop in
Step~5 performs any updates, estimate~\eqref{muomega} and the update relation show that $(\mu^{(k+1)},\omega^{(k+1)})\in U_\epsilon$,
provided that $\epsilon$ is chosen sufficiently small.
It follows by induction that all subsequent iterates remain in $U_\epsilon$ and satisfy the preceding estimates. This completes the proof.\qed

\bibliographystyle{spmpsci}
\bibliography{references}


%
%




\end{document}